\documentclass[12pt]{article}
\usepackage{color}
\usepackage{graphicx}
\usepackage{amsmath}
\usepackage{amssymb}
\usepackage{mathrsfs}
\usepackage[numbers,sort&compress]{natbib}
\usepackage{amsthm}
\numberwithin{equation}{section}
\usepackage{pgf}
\usepackage{CJK}
\usepackage{graphicx}
\usepackage{tikz}
\usepackage{stmaryrd}
\usepackage{graphicx}
\usepackage{hyperref}
\usepackage[latin1]{inputenc}
\usepackage[T1]{fontenc}
\usepackage[english]{babel}
\usepackage{listings}
\usepackage{xcolor,mathrsfs,url}
\usepackage{amssymb}
\usepackage{amsmath}
\usepackage{ifthen}
\newtheorem{theorem}{Theorem}
\newtheorem{lemma}{Lemma}
\newtheorem{corollary}{Corollary}
\newtheorem{Proposition}{Proposition}
\newtheorem{RHP}{RHP}
\newtheorem{Assumption}{Assumption}
\usepackage {caption}
\usepackage{float}
\usepackage{subfigure}

\DeclareMathOperator*{\res}{Res}

\begin{document}

\title{ On asymptotic approximation   of the  modified Camassa-Holm equation in different  space-time  solitonic regions  }
\author{Yiling YANG$^1$\thanks{\ Email address: 19110180006@fudan.edu.cn } \  and  \  Engui FAN$^{1}$\thanks{\ Corresponding author and email address: faneg@fudan.edu.cn } }
\footnotetext[1]{ \  School of Mathematical Sciences  and Key Laboratory   for Nonlinear Science, Fudan   University, Shanghai 200433, P.R. China.}

\date{ }

\maketitle
\begin{abstract}
	\baselineskip=17pt

In this paper, we   study the long time asymptotic behavior for   the initial value problem   of the modified Camassa-Holm (mCH) equation in the solitonic region
\begin{align}
	&m_{t}+\left(m\left(u^{2}-u_{x}^{2}\right)\right)_{x}+\kappa u_{x}=0, \quad m=u-u_{x x}, \nonumber \\
	&u(x, 0)=u_{0}(x),\nonumber
\end{align}	
where $\kappa$ is a positive constant.  Based on the spectral analysis of the  Lax pair associated with the   mCH  equation and scattering matrix,
  the solution of the  Cauchy problem  is  characterized   via the solution of a  Riemann-Hilbert (RH) problem.
   Further using  the $\overline\partial$ generalization of Deift-Zhou steepest descent method,
  we derive   different long time asymptotic expansion of the solution $u(x,t)$   in  different space-time solitonic region of $x/t$.
These  asymptotic   approximations can be   characterized with  an $N(\Lambda)$-soliton  whose parameters are modulated by
a sum of localized soliton-soliton interactions as one moves through the region with diverse  residual error order from  $\overline\partial$ equation:
 $\mathcal{O}(|t|^{-1+2\rho})$ for $\xi=\frac{y}{t}\in(-\infty,-0.25)\cup(2,+\infty)$ and  $\mathcal{O}(|t|^{-3/4})$ for $\xi=\frac{y}{t}\in(-0.25,2)$.
Our results also confirm  the soliton resolution conjecture  and  asymptotically stability of N-soliton solutions for the  mCH equation.  

\end{abstract}

\baselineskip=18pt

\newpage

\tableofcontents

\section {Introduction}
\quad
The inverse  scattering  transform (IST) procedure, as one of the most powerful tool to investigate
solitons of nonlinear integrable  models, was first discovered by Gardner,
Green, Kruskal and Miura \cite{Gardner1967}.   The
modern version of IST is based on the dressing method proposed by Zakharov
and Shabat, first in terms of the factorization of integral operators
on a line into a product of two Volterra integral operators  \cite{Zakharov1974}  and then
using the Riemann-Hilbert (RH) problem  \cite{Zakharov1979}. In general,  the initial value  problems  of integrable systems  can be solved
  by suing IST or RH  method only  in the case of  refectioness potentials.
     So  a natural idea is  to study the asymptotic behavior of solutions
to  integrable systems.

 The  study  on the long-time behavior of nonlinear wave equations   was first  carried out with IST method  by Manakov  in 1974 \cite{Manakov1974}.
  Later,  by using this method, Zakharov and Manakov   gave   the first result on the  large-time asymptotic  of solutions for the  NLS equation  with  decaying initial value \cite{ZM1976}.
The IST method    also    worked  for long-time behavior of integrable systems    such as  KdV,  Landau-Lifshitz  and the reduced Maxwell-Bloch   system \cite{SPC,BRF,Foka}.
    In 1993,
    Deift and Zhou  developed a  nonlinear steepest descent method to rigorously obtain the long-time asymptotics behavior of the solution for the MKdV equation
by deforming contours to reduce the original Riemann-Hilbert (RH)  problem   to a  model one  whose solution is calculated in terms of parabolic cylinder functions \cite{RN6}.
Since then    this method
has been widely  applied  to  the focusing NLS equation, KdV equation, Camassa-Holm equation,Degasperis-Procesi,  Fokas-Lenells equation, Sasa-Satuma equation,  short-pulse equation   etc. \cite{RN9,RN10,Grunert2009,MonvelCH,Monvel1,Monvel2,xu2015,xusp,Geng3,XF2020}.

In recent years,   McLaughlin and   Miller further  presented a $\bar\partial$ steepest descent method which combine   steepest descent  with  $\bar{\partial}$-problem  rather than the asymptotic analysis
 of singular integrals on contours to analyze asymptotic of orthogonal polynomials with non-analytical weights  \cite{MandM2006,MandM2008}.
When  it  is applied  to integrable systems,   the $\bar\partial$ steepest descent method  also has  displayed some advantages,  such as   avoiding delicate estimates involving $L^p$ estimates  of Cauchy projection operators, and leading  the non-analyticity in the RH problem  reductions to a $\bar{\partial}$-problem in some sectors of the complex plane.
  Dieng and  McLaughin used it to study the defocusing NLS equation  under essentially minimal regularity assumptions on finite mass initial data \cite{DandMNLS}; This   method   was also successfully applied to prove asymptotic stability of N-soliton solutions to focusing NLS equation \cite{fNLS}; Jenkins et.al  studied  soliton resolution for the derivative nonlinear NLS equation  for generic initial data in a weighted Sobolev space \cite{Liu3}.  For finite density initial data, Cussagna and  Jenkins improved $\bar\partial$ steepest descent method to   study  the asymptotic stability for   defocusing NLS equation with non-zero boundary conditions \cite{SandRNLS}.  Recently   $\bar\partial$ steepest descent method has been successfully used  to
 study  the  short pulse, three-wave, modifed Camassa-Holm and  Fokas-Lenells equations \cite{YF1,YF2,YF3,YF4}.

In the present paper, we study the long time asymptotic behavior for the initial value problem  for the
modified Camassa-Holm (mCH) equation:
\begin{align}\label{mch}
	&m_{t}+\left(m\left(u^{2}-u_{x}^{2}\right)\right)_{x}+\kappa u_{x}=0, \quad m=u-u_{x x}, \\
	&u(x, 0)=u_{0}(x), \quad x \in \mathbb{R},\  t>0,
\end{align}	
where $\kappa$ is a positive constant, and   $u=u(x,t)$  is a real-valued function of   $x$ and
  $t$.   The mCH equation (\ref{mch}) as a new integrable system was derived independently by Fokas  \cite{Fokas}, Fuchssteiner \cite{BF1996}, Olver and
Rosenau \cite{PP1996}, and Qiao \cite{Qiao}, where the equation was derived from the two-dimensional
Euler system, and Lax pair, the M/W-shape solitons and peakon/cuspon solutions were presented. So the mCH equation (\ref{mch})
   is also   referred to as the Fokas-Olver-Rosenau-Qiao equation  \cite{THFan},
but is mostly known as the  mCH  equation.

In recent years,  the mCH equation (\ref{mch})  has attracted considerable interest  due to its rich mathematical structure and  remarkable properties such as algebro-geometric quasiperiodic solutions \cite{THFan}, Backlund transformation \cite{WLM},  conservative peakons \cite{CS1,CS2}, local well-posedness for classical solutions and global weak
solutions to (\ref{mch}) in Lagrangian coordinates \cite{Gao2018} and solitary wave solutions \cite{Gui2013}.
 Under a  simple  transformation
\begin{align}
		x=\tilde{x}, \ \
		t=\frac{2}{\kappa} \tilde{t},  \ \
		u(x, t)=\sqrt{\frac{\kappa}{2}} \tilde{u}(\tilde{x}, \tilde{t}),
\end{align}
the mCH equation (\ref{mch}) becomes
\begin{equation}
	\tilde{m}_{\tilde{t}}+\left(\tilde{m}\left(\tilde{u}^{2}-\tilde{u}_{\tilde{x}}^{2}\right)\right)_{\tilde{x}}+2 \tilde{u}_{\tilde{x}}=0, \quad \tilde{m}=\tilde{u}-\tilde{u}_{x x}.
\end{equation}
So without loss of generally, we fix $\kappa=2$. Applying the scaling transformation
\begin{align}
		x=\epsilon\hat{x}, \ \
		t= {\hat{t}}/{\epsilon},  \ \
		u(x, t)= \epsilon^2\hat{u}(\hat{x}, \hat{t}),
\end{align}
and let $\epsilon\to0$, then the  mCH  reduces  short pulse equation
\begin{align}
	&u_{xt}=u+\frac{1}{6}(u^3)_{xx}.
\end{align}

Recently, Boutet de Monvel,   Kostenko,  Shepelsky and Teschl developed a RH approach to the mCH equation (\ref{mch}) with nonzero boundary
conditions \cite{Mon}.  They further present the results of
the asymptotic analysis in the solitonless case for the two sectors $\frac{3}{4} <\frac{x}{t}<1, \ 1<\frac{x}{t}<3$ \cite{Mon2}.   Xu and Fan  applied Deift-Zhou
steepest decedent method to obtain  long-time asymptotic behavior of (\ref{mch})  with zero boundary conditions     \cite{Xurhp}.
\begin{align}
	u(x,t)=f(x,t,\xi)t^{-1/2}+\mathcal{O}(t^{-1} \log t ),
\end{align}
where $\xi=\frac{x}{t}$, and $f$ has different structure for $\xi$ in  different cases respectively.

In our results,  for the   weighted  Sobolev  initial data $u_0(x) \in H^{1,1}(\mathbb{R})$,  we   obtain  the leading  order  asymptotic approximation  for the  mCH equation (\ref{mch})
(see  Theorem \ref{last} in the section 9):  when $\xi\in(-\infty,-0.25)\cup(2,+\infty)$,
 	\begin{align}
	u(x,t)=&=u^r(x,t;\tilde{\mathcal{D}}) +\mathcal{O}(t^{-1+2\rho}),\label{ours}
	\end{align}
and when $\xi\in(-0.25,2)$,
\begin{align}
u(x,t)=u^r(x,t;\tilde{\mathcal{D}}) +f_{11}t^{-1/2}+\mathcal{O}(t^{-3/4}).
\end{align}
 Our results is different from it  in   \cite{Mon2,Xurhp}.

This  paper is arranged as follows.  In section \ref{sec2},   we recall  some main  results on
 the construction  process  of  RH  problem    \cite{Xurhp},  which will be used
 to analyze   long-time asymptotics  of the mCH equation in our paper.  In section \ref{secr}, we shown that  the reflection coefficient $r(z)$  belongs in $H^{1,1}(\mathbb{R})$.
  In section \ref{sec3},
a $T(z)$ function  is introduced to  define a new   RH problem  for  $M^{(1)}(z)$,  which  admits a regular discrete spectrum and  two  triangular  decompositions of the jump matrix
near 0.
In section \ref{sec4},  by introducing a matrix-valued  function  $R(z)$,  we obtain  a mixed $\bar{\partial}$-RH problem  for  $M^{(2)}(z)$  by continuous extension of  $M^{(1)}(z)$.
     In section \ref{sec5},  we decompose  $M^{(2)}(z)$    into a
 model RH   problem  for  $M^{R}(z)$ and a  pure $\bar{\partial}$ Problem for  $M^{(3)}(z)$.
 The  $M^{R}(z)$  can be obtained  via  an modified reflectionless RH problem $M^{(r)}(z)$   for the soliton components which  is solved   in Section \ref{sec6} and an inner model $M^{lo}(z)$  for the stationary phase point $\xi_k$   which are   approximated   by  parabolic cylinder model    obtained  in Section \ref{sec8} when $\xi=\frac{y}{t}\in(-0.25,2)$. But when $\xi=\frac{y}{t}\in(-\infty,-0.25)\cup(2,+\infty)$, $M^{R}(z)=M^{(r)}(z)$. This is a more simple case.
  In section \ref{sec7},   the error function   can be computed  with a  small-norm RH problem.
 In Section \ref{sec8},   we analyze  the $\bar{\partial}$-problem  for $M^{(3)}$.
   Finally, in Section \ref{sec9},   based on  the result obtained above,   a relation formula
   is found
\begin{align}
 M(z) = M^{(3)}(z)E(z)M^{(r)}(z)R^{(2)}(z)^{-1}T(z)^{-\sigma_3},\nonumber
\end{align}
from which   we then obtain the   long-time   asymptotic behavior  for the mch equation (\ref{mch}) via reconstruction formula.

\section {The spectral analysis and the   RH problem}\label{sec2}

\subsection{Some notations}
\quad
In this subsection, we fix some notations used this paper.
$\sigma_1$, $\sigma_2$ and $\sigma_3$  are  Pauli matrices
\begin{equation}
	\sigma_1=\left(\begin{array}{cc}
		0 & 1  \\
		1 & 0
	\end{array}\right),\hspace{0.5cm}\sigma_2=\left(\begin{array}{cc}
		0 & -i  \\
		i & 0
	\end{array}\right),\hspace{0.5cm}\sigma_3=\left(\begin{array}{cc}
		1 & 0   \\
		0 & -1
	\end{array}\right).\hspace{0.5cm}\nonumber
\end{equation}

If $I$ is an interval on the real line $\mathbb{R}$ and $X$ is a  Banach space, then $C^0(I,X)$ denotes the space of continuous functions on $I$ taking values in $X$. It is equipped with the norm
\begin{equation*}
	\|f\|_{C^{0}(I, X)}=\sup _{x \in I}\|f(x)\|_{X}.
\end{equation*}
Moreover, denote $C^0_B(X)$ as a  space of bounded continuous functions on $X$.

If the  entries  $f_1$ and $f_2$  are in space $X$,  then we call vector  $\vec{f}=(f_1,f_2)^T$  is in space $X$ with $\parallel \vec{f}\parallel_X\triangleq \parallel f_1\parallel_X+\parallel f_2\parallel_X$. Similarly, if every  entries of  matrix $A$ are in space $X$, then we call $A$ is also in space $X$.

We introduce  the normed spaces:
\begin{itemize}
\item A weighted $L^p(\mathbb{R})$ space is specified by
$$L^{p,s}(\mathbb{R})  =  \left\lbrace f(x)\in L^p(\mathbb{R}) | \hspace{0.1cm} |x|^sf(x)\in L^p(\mathbb{R}) \right\rbrace; $$
\item  A Sobolev space is defined by
$$W^{k,p}(\mathbb{R})  =  \left\lbrace f(x)\in L^p(\mathbb{R}) | \hspace{0.1cm} \partial^j f(x)\in L^p(\mathbb{R})  \text{ for }j=1,2,...,k \right\rbrace;$$

\item A weighted Sobolev space  is defined by
 $$H^{k,s}(\mathbb{R})   =  \left\lbrace f(x)\in L^2(\mathbb{R}) | \hspace{0.1cm} (1+|x|^s)\partial^jf(x)\in L^2(\mathbb{R}),  \text{ for }j=1,...,k \right\rbrace.$$

\end{itemize}
And the norm of $f(x)\in L^{p}(\mathbb{R})$ and $g(x)\in L^{p,s}(\mathbb{R})$ are  abbreviated to $\parallel f\parallel_{p}$,  $\parallel g\parallel_{p,s}$ respectively.
 In our paper, we only need the initial value $m(x,0)=u(x,0)-u_{xx}(x,0)$ in $H^{2,2}(\mathbb{R})$.

\subsection{Spectral analysis on the Lax pair}

\quad The mCH equation (\ref{mch}) is completely integrable and    admits the Lax pair \cite{Xurhp}
\begin{equation}
\Phi_x = X \Phi,\hspace{0.5cm}\Phi_t =T \Phi, \label{lax0}
\end{equation}
where
\begin{equation}
	X=-\frac{k}{2}\sigma_3+\frac{i\lambda m(x,t)}{2}\sigma_2,\nonumber
\end{equation}
\begin{equation}
	T=\frac{k}{\lambda^2}\sigma_3+\frac{k}{2} \left(u^{2}-u_{x}^{2}\right)\sigma_3-i\left(\frac{u-k u_{x}}{\lambda}+\frac{\lambda}{2} \left(u^{2}-u_{x}^{2}\right) m \right) \sigma_2,
 \nonumber
\end{equation}
with $k=\sqrt{1-\lambda^2}$ and $\lambda\in \mathbb{C}$ being  a spectral parameter.
The   $X$ and $T$ of above Lax pair are traceless matrices, so it
implies  that the determinant of a matrix solution to (\ref{lax0})   is independent of $x$ and $t$.
To avoid multi-valued case  of   eigenvalue  $\lambda$,   we  introduce    a uniformization variable
\begin{equation}
	z= k+\lambda,
\end{equation}
and  obtain two single-valued functions
\begin{equation}
	k(z)=\frac{i}{2}(z-\frac{1}{z}),\hspace{0.5cm}\lambda(z)=\frac{1}{2}(z+\frac{1}{z}).\label{uniformization55}
\end{equation}
Usually we only use the $x$-part of Lax pair to analyze the initial value problem. For example in section \ref{secr}, we consider the case of $t=0$ to obtain the relationship of reflection coefficient and initial data. And the $t$-part   is often used
to determine the time evolution of the scattering data   by inverse scattering transform method.
Here  different from   NLS and  derivative NLS equations \cite{DandMNLS,SandRNLS,fNLS},     the Lax pair   (\ref{lax0}) for the mCH equation  has
singularities at $\lambda=0, \infty$ and branch cut points  $z=\pm i$  in the extended complex $\lambda$-plane, so the asymptotic  behavior of their eigenfunctions should   be controlled.
But the asymptotic behaviors of Lax pair  (\ref{lax0}) as $z \to \infty$  can't be directly obtained.
 This difficulty is also appear in other WKI-type equation,
and it is solved by appropriate transformation  due to   Boutet de Monvel and Shepelsky
\cite{CH,CH1}. This idea  was also applied in \cite{RHPsp,Xurhp}.  By using  her consequence directly,  we need to use   different transformations respectively to analyze these  singularities $z=i$, $z=0$ and $z=\infty$, and give a new scale to construct RH problem.  Now we first consider   the case $z=\infty$, which   is corresponding to  $\lambda=\infty$.

\noindent \textbf{Case I: $z=\infty$}
\\
In order to control asymptotic behavior of the Lax pair (\ref{lax0}) as $z\to \infty$, we define
 \begin{align}
 	F(x, t)=\sqrt{\frac{q+1}{2 q}}\left(\begin{array}{cc}
 		1 & \frac{-i m}{q+1} \\
 		\frac{-i m}{q+1} & 1
 	\end{array}\right),\label{F}
 \end{align}
and
\begin{equation}
	p(x,t,z)=x-\int_{x}^{\infty} (q-1) dy-\frac{2t}{\lambda(z)^2}, \ \ \ q=\sqrt{m^2+1}.
\end{equation}
Making  a transformation
\begin{equation}
	\Phi_\pm=F\mu_\pm e^{-\frac{i}{4}(z-\frac{1}{z})p\sigma_3}\label{transmu},
\end{equation}
then  we obtain  a new Lax pair
\begin{align}
	&(\mu_\pm)_x = -\frac{i}{4}(z-\frac{1}{z})p_x[\sigma_3,\mu_\pm]+P\mu_\pm,\label{lax1.1}\\
	&(\mu_\pm)_t =-\frac{i}{4}(z-\frac{1}{z})p_t[\sigma_3,\mu_\pm]+L\mu_\pm, \label{lax1.2}
\end{align}
where
\begin{align}
	&P=\frac{i m_{x}}{2q^2}\sigma_1+  \frac{m}{2 z q}\left(\begin{array}{cc}
		-i m & 1 \\
		-1 & i m
	\end{array}\right),\\
 	&L=\frac{i m_{t}}{2q^2}\sigma_1-\frac{m\left(u^{2}-u_{x}^{2}\right)}{2 z q}\left(\begin{array}{cc}
		-i m & 1 \\
		-1 & i m
	\end{array}\right)+\frac{\left(z^{2}-1\right) u_{x}}{z^{2}+1}\sigma_1 \nonumber\\
	&-\frac{2 z u}{\left(z^{2}+1\right) q}\left(\begin{array}{cc}
		-i m & 1 \\
		-1 & i m
	\end{array}\right)+\frac{2 i z\left(z^{2}-1\right)}{\left(z^{2}+1\right)^{2}}\left(\begin{array}{cc}
		\frac{1}{q}-1 & \frac{-i m}{q} \\
		\frac{i m}{q} & 1-\frac{1}{q}
	\end{array}\right).
\end{align}
Moreover
\begin{align}
	\mu_\pm \sim I, \hspace{0.5cm} x \rightarrow \pm\infty.
\end{align}
The Lax pair (\ref{lax1.1})-(\ref{lax1.2}) can be written in to a  total differential form
\begin{equation}
	d\left(e^{\frac{i}{4}(z-\frac{1}{z})p\hat{\sigma}_3}\mu_\pm \right)=e^{\frac{i}{4}(z-\frac{1}{z})p\hat{\sigma}_3}\left(Pdx+Ldt \right)\mu_\pm,
\end{equation}
which leads to two  Volterra type integrals
\begin{equation}
	\mu_\pm=I+\int_{x}^{\pm \infty}e^{-\frac{i}{4}(z-\frac{1}{z})(p(x)-p(y))\hat{\sigma}_3}P(y)\mu_\pm(y)dy\label{intmu}.
\end{equation}
Denote
$$\mu_\pm=\left(\left[ \mu_\pm\right]_1, \left[ \mu_\pm\right]_2 \right), $$
where  $\left[ \mu_\pm\right] _1$ and $\left[ \mu_\pm\right] _2$ are
the first and second columns of $\mu_\pm$ respectively.
Then  from  (\ref{intmu}),   we can show that  $\left[ \mu_-\right] _1$ and $\left[ \mu_+\right] _2$ are analysis in $\mathbb{C}^+$;  $\left[ \mu_+\right] _1$
and $\left[ \mu_-\right] _2$ are analysis in $\mathbb{C}^-$.

\begin{Proposition}\label{sym}
	Jost functions $ \mu_\pm$ admit three reduction conditions on
	the $z$-plane:
	
	The first symmetry reduction:
\begin{equation}
\mu_\pm(z)=\sigma_2\overline{\mu_\pm(\bar{z})}\sigma_2=\sigma_1\overline{\mu_\pm(-z)}\sigma_1.\label{symPhi1}
\end{equation}

The second symmetry reduction:
\begin{equation}
	\mu_\pm(z)=F^{-2}\sigma_2\mu_\pm(-z^{-1})\sigma_2,\label{symPhi2}
\end{equation}

\end{Proposition}

Since   $\Phi_\pm$ are two fundamental matrix solutions of the  Lax  pair (\ref{lax0}),  there exists a linear  relation between $\Phi_+$ and $\Phi_-$, namely
\begin{equation}
	\Phi_-(z;x,t)=\Phi_+(z;x,t)S(z),\hspace{0.5cm} z\in \mathbb{R}\hspace{0.5cm},\label{scattering}
\end{equation}
where $S(z)$ is called scattering matrix, and it is only depended on $z$
\begin{equation}
		S(z) =\left(\begin{array}{cc}
			a(z) &-\overline{b(\bar{z})}   \\[4pt]
			 b(z) & \overline{a(\bar{z})}
		\end{array}\right),\hspace{0.5cm}\det[S(z)]=1.
\end{equation}
Combing with (\ref{transmu}), above equation trans to
\begin{align}
	\mu_-(z)=\mu_+(z)e^{-\frac{i}{4}(z-\frac{1}{z})p\hat{\sigma}_3}S(z).
\end{align}
Then $S(z)$ has following symmetry reduction:
\begin{equation}
	S(z)=\overline{S(\bar{z}^{-1})}=\sigma_3S\left( -z^{-1}\right) \sigma_3.\label{symS}
\end{equation}
And   the reflection coefficients is defined by
\begin{equation}
	r(z)=\frac{b(z)}{a(z)},\label{symr}
\end{equation}
with symmetry reduction:
\begin{equation}
	r(z)=\overline{r(\bar{z}^{-1})}=r(-z^{-1})=-\overline{r(-\bar{z})}.
\end{equation}
From (\ref{scattering}), $a(z)$, $b(z)$ can be expressed by $\mu_\pm$ as
\begin{align}
	&a(z)=\mu_-^{11}\overline{\mu_+^{11}}+\mu_-^{21}\overline{\mu_+^{21}},\hspace{0.5cm}b(z)=\overline{\mu_-^{11}\mu_+^{21}}-\overline{\mu_-^{21}\mu_+^{11}}.\label{scatteringcoefficient2}
\end{align}
So $a(z)$ is analytic on $\mathbb{C}^+$.
In addition,   $\mu_\pm$  admit the  asymptotics
\begin{align}
	\mu_\pm=I+\dfrac{D_1}{z}+\mathcal{O}(z^{-2}),\hspace{0.5cm}z \rightarrow \infty,\label{asymu}
\end{align}
where the off-diagonal entries of the matrix $D_1(x, t)$ are
\begin{equation}
	D_{12}(x,t)=-D_{21}(x,t)=\dfrac{m_{x}}{(1+m^2)^{3/2}}.
\end{equation}
From (\ref{asymu}) and (\ref{a}),    we obtain  the asymptotic   of $a(z)$
\begin{align}
	a(z)=1+\mathcal{O}(z^{-1}),\hspace{0.5cm}z \rightarrow \infty.\label{asya}
\end{align}
The zeros of $a(z)$ on $\mathbb{R}$   are known to
occur and they correspond to spectral singularities.  They are excluded from our analysis in the this paper. To deal with our following work,
we assume our initial data satisfy this assumption.
\begin{Assumption}\label{initialdata}
	The initial data $u \in  H^{4,2}(\mathbb{R})$     and it generates generic scattering data which satisfy that
	
	\textbf{1. }a(z) has no zeros on $\mathbb{R}$.
	
	\textbf{2. }a(z) only has finite number of simple zeros.
\end{Assumption}
And the proof of following proposition wil be given is section \ref{secr}.
\begin{Proposition}\label{pror}
	If the initial data $u \in  H^{4,2}(\mathbb{R})$, then $r(z)$ belongs to $H^{1,1}(\mathbb{R})$.
\end{Proposition}

Suppose that $a(z)$ has $N_1$ simple zeros $z_1,...,z_{N_1}$ on $\{z\in\mathbb{C}^+:{\rm Im}z>0,|z|>1\}$, and $N_2$ simple
zeros $w_1,...,w_m$ on the circle $\{z=e^{i\varphi}:\frac{\pi}{2}<\varphi<\pi\}$.   The  symmetries  (\ref{symS}) imply that
\begin{equation}
	a( z_n)=0 \Leftrightarrow a(- \bar{z}_n)=0 \Leftrightarrow a\left(- \frac{1}{z_n}\right)=0 \Leftrightarrow a\left( \frac{1}{\bar{z}}\right)=0, \hspace{0.5cm}n=1,...,N_1,\nonumber
\end{equation}
and on the circle
\begin{equation}
	a( w_m)=0\Leftrightarrow a(- \bar{w}_m)=0, \hspace{0.5cm}m=1,...,N_2. \nonumber
\end{equation}
So the zeros of $a(z)$  come in pairs.
It is convenient to define zeros of $a(z)$ as $\zeta_n=z_n$, $\zeta_n+N_1=-\bar{z}_n$, $\zeta_{n+2N_1}=\bar{z}_n^{-1}$ and $\zeta_{n+3N_1}=-z_n^{-1}$ for $n=1,\cdot\cdot\cdot,N_1$;  $\zeta_{m+4N_1}=w_m$ and $\zeta_{m+4N_1+N_2}=-\bar{w}_m$ for $m=1,\cdot\cdot\cdot,N_2$. Then $\bar{\zeta}_n$ is the zeros of $\overline{a(\bar{z})}$. Therefore, the discrete spectrum is
\begin{equation}
	\mathcal{Z}=\left\{ \zeta_n, \  \bar{\zeta}_n\right\}_{n=1}^{4N_1+2N_2}, \label{spectrals}
\end{equation}
with $\zeta_n\in \mathbb{C}^+$ and $\bar{\zeta}_n\in \mathbb{C}^-$. And the distribution  of $	\mathcal{Z}$ on the $z$-plane   is shown  in Figure \ref{fig:figure1}.
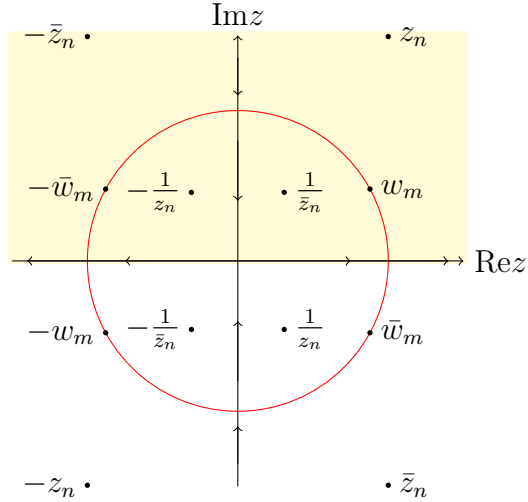
\begin{figure}[H]
	\centering
	\begin{tikzpicture}[node distance=2cm]
		\filldraw[yellow!20,line width=3] (3,0.01) rectangle (0.01,3);
		\filldraw[yellow!20,line width=3] (-3,0.01) rectangle (-0.01,3);
		\draw[->](-3,0)--(3,0)node[right]{Re$z$};
		\draw[->](0,-3)--(0,3)node[above]{Im$z$};
		\draw[red] (2,0) arc (0:360:2);
		\draw[->](0,0)--(-1.5,0);
		\draw[->](-1.5,0)--(-2.8,0);
		\draw[->](0,0)--(1.5,0);
		\draw[->](1.5,0)--(2.8,0);
		\draw[->](0,2.7)--(0,2.2);
		\draw[->](0,1.6)--(0,0.8);
		\draw[->](0,-2.7)--(0,-2.2);
		\draw[->](0,-1.6)--(0,-0.8);
	\coordinate (A) at (2,2.985);
\coordinate (B) at (2,-2.985);
\coordinate (C) at (-0.616996232,0.9120505887);
\coordinate (D) at (-0.616996232,-0.9120505887);
\coordinate (E) at (0.616996232,0.9120505887);
\coordinate (F) at (0.616996232,-0.9120505887);
\coordinate (G) at (-2,2.985);
\coordinate (H) at (-2,-2.985);
		\coordinate (J) at (1.7570508075688774,0.956);
		\coordinate (K) at (1.7570508075688774,-0.956);
		\coordinate (L) at (-1.7570508075688774,0.956);
		\coordinate (M) at (-1.7570508075688774,-0.956);
		\fill (A) circle (1pt) node[right] {$z_n$};
		\fill (B) circle (1pt) node[right] {$\bar{z}_n$};
		\fill (C) circle (1pt) node[left] {$-\frac{1}{z_n}$};
		\fill (D) circle (1pt) node[left] {$-\frac{1}{\bar{z}_n}$};
		\fill (E) circle (1pt) node[right] {$\frac{1}{\bar{z}_n}$};
		\fill (F) circle (1pt) node[right] {$\frac{1}{z_n}$};
		\fill (G) circle (1pt) node[left] {$-\bar{z}_n$};
		\fill (H) circle (1pt) node[left] {$-z_n$};
		\fill (J) circle (1pt) node[right] {$w_m$};
		\fill (K) circle (1pt) node[right] {$\bar{w}_m$};
		\fill (L) circle (1pt) node[left] {$-\bar{w}_m$};
		\fill (M) circle (1pt) node[left] {$-w_m$};
	\end{tikzpicture}
	\caption{Distribution of the discrete spectrum $\mathcal{Z}$. The red one is  unit circle.}
	\label{fig:figure1}
\end{figure}

Moreover, from trace formulae we have
\begin{equation}
	a(z)=\prod_{j=1}^{4N_1+2N_2}\frac{z-\zeta_j}{z-\bar{\zeta}_j}\exp\left\lbrace-\frac{1}{2\pi i}\int_{\mathbb{R}}\frac{\log (1+|r(s)|^2)}{s-z}ds \right\rbrace .\label{a}
\end{equation}
Then by taking $z\to\infty$, it implies
\begin{equation}
	0=\frac{1}{2\pi i}\int_{\mathbb{R}}\frac{\log (1+|r(s)|^2)}{s}ds.
\end{equation}
\noindent \textbf{Case II: $z=0$} (corresponding to  $\lambda\to\infty$).
\\
From the symmetry condition in Proposition \ref{sym}, we can obtain the property of $\mu(z)$ as $z\to 0$. In addition, (\ref{asya}) and (\ref{symS}) imply $a(0)=1$, which means $r(0)=0$.

\noindent \textbf{Case III:} $z=\pm i$ (corresponding to  $\lambda=0$).
\\
Consider the Jost solutions of  the Lax pair (\ref{lax0}), which are restricted by the boundary conditions
\begin{equation}
	\Phi_\pm \sim  e^{(-\frac{k}{2}x+\frac{k}{\lambda^2}t)\sigma_3}, \hspace{0.5cm}x\to \pm\infty.\label{asyx}
\end{equation}
Define a new transformation:
\begin{equation}
	\mu^0_\pm=\Phi_\pm e^{(\frac{k}{2}x-\frac{k}{\lambda^2}t)\sigma_3},\label{trans2}
\end{equation}
with
\begin{equation*}
	\mu^0_\pm \sim I, \hspace{0.5cm} x \rightarrow \pm\infty.
\end{equation*}
Then the Lax pair (\ref{lax0}) change to
\begin{align}
	&(\mu^0_\pm)_x = -\frac{k}{2}[\sigma_3,\mu^0_\pm]+L_0\mu^0_\pm,\label{lax0.1}\\
	&(\mu^0_\pm)_t = \frac{k}{\lambda^2}[\sigma_3,\mu^0_\pm]+M_0\mu^0_\pm, \label{lax0.2}
\end{align}
with
\begin{align}
	&L_0=\frac{\lambda mi}{2}\sigma_2,\\
	&M_{0}=\frac{\left(u^{2}-u_{x}^{2}\right)}{2}\left(\begin{array}{cc}
		k & -\lambda m \\
		\lambda m & -k
	\end{array}\right)+\frac{u}{\lambda}\left(\begin{array}{cc}
		0 & -1 \\
		1 & 0
	\end{array}\right)+\frac{k}{\lambda} u_{x} \sigma_{1}.
\end{align}
Similarly, we denote $\mu_\pm^0=\left(\left[ \mu_\pm^0\right] _1,\left[ \mu^0_\pm\right] _2 \right)$. To reconstruct $u(x,t)$, we analyze its asymptotic behavior as $z\to i$:
\begin{align}
	\mu^0=I+(z-i)\left(\begin{array}{cc}
		0 & -\frac{1}{2}(u+u_x) \\
		-\frac{1}{2}(u-u_x) & 0
	\end{array}\right)+\mathcal{O}\left( (z-i)^2\right) .\label{asymu0}
\end{align}
The relations  (\ref{trans2}) and (\ref{transmu}) is
\begin{equation}
	\mu_\pm(x,t,z)=F^{-1}(x,t)\mu^0_\pm e^{\frac{i}{4}(z-\frac{1}{z})c_\pm(x,t)\sigma_3},\label{mu0}
\end{equation}
where
\begin{align}
	c_\pm(x,t)=\int_{\pm\infty}^x (q-1)dy.\label{c+-}
\end{align}
Further,  taking $z \rightarrow i$ in (\ref{mu0}) and combining it to (\ref{asymu0}), we get  the asymptotic of $a(z)$ at $z \rightarrow i$:
\begin{align}
	a(z)=e^{\frac{1}{2}\int_{\mathbb{R}}(q-1)dx}\left(1+ \mathcal{O}\left( (z-i)^2\right)\right), \hspace{0.3cm}\text{as }z\to i.
\end{align}

\subsection{A RH problem}

\quad As shown in  \cite{Xurhp},   denote   norming constant    $c_n=b_n/a'(z_n)$.    Then we have  residue conditions as
\begin{align}
	\res_{z=z_n}\left[\frac{[\mu_{-}]^1(z)}{a(z)}\right]=c_ne^{-2k(z_n)p(z_n)}[\mu_{+}]^2(z_n),\label{resrelation1}
\end{align}
For $m=1,...,N_2$, there also have $c_{N_1+m}=b_{N_1+m}/a'(w_m)$ and
\begin{align}
	&\res_{z=w_m}\left[\frac{[\mu_{-}]^1(z)}{a(z)}\right]=c_{N_1+m}e^{-2 k (w_n) p(w_n)}[\mu_{+}]^2(w_m).
\end{align}
The symmetry of $a(z)$ and $\mu(z)$ in Proposition \ref{sym} leads to other norming constant for zeros of $a(z)$.
For brevity, we introduce a new constant $C_n$ as: for  $n=1,...,N_1$, $C_n=c_n$, $C_{n+N_1}=\bar{c}_n$ $C_{n+2N_1}=-\bar{z}_n^{-2}\bar{c}_n$ and $C_{n+3N_1}=-z^{-2}_nc_n$; for  $m=1,...,N_2$, $C_{m+4N_1}=\bar{C}_{m+4N_1+N_2}=c_{m+N_1}$,
and  the collection
$\sigma_d=  \left\lbrace \zeta_n,C_n\right\rbrace^{4N_1+2N_2}_{n=1}  $
is called the \emph{scattering data}.
define  a   sectionally meromorphic matrix
\begin{equation}
	N(z)\triangleq N(z;x,t)=\left\{ \begin{array}{ll}
		\left( \frac{\left[ \mu_-\right] _1}{a(z)}, \left[ \mu_+\right] _2\right),   &\text{as } z\in \mathbb{C}^+,\\[12pt]
		\left( \left[ \mu_+\right] _1,\frac{\left[ \mu_-\right] _2}{\overline{a(\bar{z})}}\right)  , &\text{as }z\in \mathbb{C}^-,\\
	\end{array}\right.
\end{equation}
which solves the following RHP.
\begin{RHP}\label{RHP1}
	 Find a matrix-valued function $	N(z)\triangleq N(z;x,t)$ which satisfies:
	
	$\blacktriangleright$ Analyticity: $N(z)$ is meromorphic in $\mathbb{C}\setminus \mathbb{R}$ and has single poles;
	
	$\blacktriangleright$ Symmetry: $N(z)=\sigma_3\overline{N(-\bar{z})}\sigma_3$=$\sigma_2\overline{N(\bar{z})}\sigma_2$=$F^{-2}\overline{N(-\bar{z}^{-1})}$;
	
	$\blacktriangleright$ Jump condition: $N$ has continuous boundary values $N_\pm(z)$ on $\mathbb{R}$ and
	\begin{equation}
		N_+(z)=N_-(z)\tilde{V}(z),\hspace{0.5cm}z \in \mathbb{R},
	\end{equation}
	where
	\begin{equation}
		\tilde{V}(z)=\left(\begin{array}{cc}
			1+|r(z)|^2 & e^{-kp}\overline{r(z)}\\
			e^{kp}r(z) & 1
		\end{array}\right);
	\end{equation}
	
	$\blacktriangleright$ Asymptotic behaviors:
	\begin{align}
		&N(z) = I+\mathcal{O}(z^{-1}),\hspace{0.5cm}z \rightarrow \infty,\\
		&N(z) =F^{-1} \left[ I+(z-i)\left(\begin{array}{cc}
			0 & -\frac{1}{2}(u+u_x) \\
			-\frac{1}{2}(u-u_x) & 0
		\end{array}\right)\right]e^{\frac{1}{2}c_+\sigma_3} +\mathcal{O}\left( (z-i)^2\right);
	\end{align}
	
	$\blacktriangleright$ Residue conditions: $N(z)$ has simple poles at each point in $ \mathcal{Z}\cup \bar{\mathcal{Z}}$ with:
	\begin{align}
		&\res_{z=\zeta_n}N(z)=\lim_{z\to \zeta_n}N(z)\left(\begin{array}{cc}
			0 & 0\\
			c_ne^{-\sqrt{1-\lambda(\zeta_n)^2}p(\zeta_n)} & 0
		\end{array}\right),\\
		&\res_{z=\bar{\zeta}_n}N(z)=\lim_{z\to \bar{\zeta}_n}N(z)\left(\begin{array}{cc}
			0 &-\bar{c}_ne^{\sqrt{1-\lambda(\bar{\zeta}_n)^2}p(\bar{\zeta}_n)}\\
			 0 & 0
		\end{array}\right).
	\end{align}
\end{RHP}

The solution of  mCH equation (\ref{mch}) is difficult to reconstruct , since  $p(x,t,z)$ is still unknown.  It has been a difficult problem when construct RHP of Camassa-Holm type equation until Boutet de Monvel and Shepelsky give the idea of changing the spatial variable in \cite{CH,CH1} which successfully applied to short-wave-type equations in \cite{RHPsp}.
So following \cite{Xurhp}, to make  the jump matrix become explicit, we introduce a  new   scale
\begin{equation}
	y(x,t)=x-\int_{x}^{+\infty} \left(\sqrt{m(k,t)^2+1}-1\right) dk=x-c_+(x,t).
\end{equation}

The price to pay for this is that the solution of the initial problem can be given only implicitly,
or parametrically: it will be given in terms of functions in the new scale, whereas the original scale will also be given in terms of functions in the new scale.
By the definition of the new scale $y(x, t)$, we define
\begin{equation}
	M(z)=M(z;y,t)\triangleq N(z;x(y,t),t),
\end{equation}
Denote the  phase function
\begin{equation}
	\theta(z)=-\frac{1}{4}(z-\frac{1}{z})\left[\frac{y}{t}-\frac{8}{(z+\frac{1}{z})^2} \right],\label{theta}
\end{equation}
and for convenience we denote $\theta_n=\theta(\zeta_n)$. Then, we can get the RH problem for the new variable $(y,t)$.
\begin{RHP}\label{RHP2}
	Find a matrix-valued function $M(z)=M(z;y,t)$ which satisfies:
	
	$\blacktriangleright$ Analyticity: $M(z)$ is meromorphic in $\mathbb{C}\setminus \mathbb{R}$ and has single poles;
	
	$\blacktriangleright$ Symmetry: $M(z)=\sigma_3\overline{M(-\bar{z})}\sigma_3$=$\sigma_2\overline{M(\bar{z})}\sigma_2$=$F^{-2}\overline{M(-\bar{z}^{-1})}$;
	
	$\blacktriangleright$ Jump condition: $M$ has continuous boundary values $M_\pm$ on $\mathbb{R}$ and
	\begin{equation}
		M_+(z)=M_-(z)V(z),\hspace{0.5cm}z \in \mathbb{R},
	\end{equation}
	where
	\begin{equation}
		V(z)=\left(\begin{array}{cc}
			1+|r(z)|^2 & e^{2it\theta}\overline{r(z)}\\
			e^{-2it\theta}r(z) & 1
		\end{array}\right);\label{jumpv}
	\end{equation}
	
	$\blacktriangleright$ Asymptotic behaviors:
	\begin{align}
		&M(z) = I+\mathcal{O}(z^{-1}),\hspace{0.5cm}z \rightarrow \infty,\\
		&M(z) =F^{-1} \left[ I+(z-i)\left(\begin{array}{cc}
			0 & -\frac{1}{2}(u+u_x) \\
			-\frac{1}{2}(u-u_x) & 0
		\end{array}\right)\right]e^{\frac{1}{2}c_+\sigma_3} +\mathcal{O}\left( (z-i)^2\right);\label{asyMi}
	\end{align}
	
	$\blacktriangleright$ Residue conditions: $M(z)$ has simple poles at each point in $ \mathcal{Z}$ with:
	\begin{align}
		&\res_{z=\zeta_n}M(z)=\lim_{z\to \zeta_n}M(z)\left(\begin{array}{cc}
			0 & 0\\
			c_ne^{-2it\theta_n} & 0
		\end{array}\right),\label{RES1}\\
		&\res_{z=\bar{\zeta}_n}M(z)=\lim_{z\to \bar{\zeta}_n}M(z)\left(\begin{array}{cc}
			0 &-\bar{c}_ne^{2it\theta_n}\\
			0 & 0
		\end{array}\right).\label{RES2}
	\end{align}
\end{RHP}
From the asymptotic behavior of the functions $\mu_\pm$ and (\ref{asyMi}), we arrive at following reconstruction formula of $u(x,t)=u(y(x,t),t)$:
\begin{equation}
	u(x,t)=u(y(x,t),t)=\lim_{z\to i}\frac{1}{z-i}\left(1- \dfrac{(M_{11}(z)+M_{21}(z))(M_{12}(z)+M_{22}(z)) }{(M_{11}(i)+M_{21}(i))(M_{12}(i)+M_{22}(i))}\right) ,\label{recons u}
\end{equation}
where
\begin{equation}
	x(y,t)=y+c_+(x,t)=y-\ln\left( \frac{M_{12}(i)+M_{22}(i)}{M_{11}(i)+M_{21}(i)}\right) .\label{recons x}
\end{equation}

\section{The reflection coefficient}\label{secr}
\quad We only consider the  $x$-part of Lax pair to give the proof of proposition (\ref{pror}) in this section. In fact,  taking account of $t$-part of Lax pair and though the standard direct scattering transform, then it deduce that $r(z)$ have linear time evolution:  $r(z,t)=e^{\frac{i}{4\lambda^2(z)}(z-\frac{1}{z})}r(z,0)$. So we can rewrite the steps as we shown in (Case I: $z=\infty$) at $t=0$.
Recall
\begin{align}
	F(x)=\sqrt{\frac{q+1}{2 q}}\left(\begin{array}{cc}
		1 & \frac{-i m}{q+1} \\
		\frac{-i m}{q+1} & 1
	\end{array}\right),\label{F}
\end{align}
and
\begin{equation}
	p(x)=x-\int_{x}^{\infty}\sqrt{m^2+1}-1dy.
\end{equation}
Making  a transformation
\begin{equation}
	\Phi_\pm(x,z)=F(x)\mu_\pm(x,z) e^{-\frac{i}{4}(z-\frac{1}{z})p(x)\sigma_3}\label{transmu0},
\end{equation}
then  we obtain  a new Lax pair
\begin{align}
	&(\mu_\pm)_x = -\frac{i}{4}(z-\frac{1}{z})p_x[\sigma_3,\mu_\pm]+P\mu_\pm,\label{lax3.1}
\end{align}
where
\begin{align}
	&P(x)=\frac{i m_{x}}{2\left(M^{(2)}+1\right)}\sigma_1+\frac{1}{2 z} \frac{m}{q}\left(\begin{array}{cc}
		-i m & 1 \\
		-1 & i m
	\end{array}\right).
\end{align}
Moreover
\begin{align}
	\mu_\pm \sim I, \hspace{0.5cm} x \rightarrow \pm\infty.
\end{align}
The standard AKNS method starts with the following two Volterra integral equations
\begin{equation}
	\mu_\pm(x,z)=I+\int_{x}^{\pm \infty}e^{-\frac{i}{4}(z-\frac{1}{z})(p(x)-p(y))\hat{\sigma}_3}P(y)\mu_\pm(y,z)dy\label{intmu0}.
\end{equation}

To obtain our result, we need estimates on the $L^2$-integral property of $\mu_\pm(z)$ and their derivatives. However, because of
the factor $1/z$ in the spectral problem (\ref{lax3.1}), we divided our approach into two cases: $|z|>1$ and $|z|<1$. And following  functional analysis results, namely estimates for Volterra-type integral equations (\ref{intmu0}) useful in the analysis of  direct scattering map.
\begin{lemma}\label{lemma1}
	For $f(x)\in L^1(\mathbb{R})$, following inequality hold.
	\begin{align}
		&\sup_{\pm x>0}\left(\int_{1}^{+\infty}|\int_{x}^{\pm \infty}\frac{f(y)}{z}dy|^2dz \right) ^{1/2}\lesssim \parallel f\parallel_1;\\
		&\left(\int_{0}^{\pm \infty}\int_{1}^{+\infty}|\int_{x}^{\pm \infty}\frac{f(y)}{z}dy|^2dzdx \right) ^{1/2}\lesssim \parallel f\parallel_{1,1/2}.
	\end{align}
\end{lemma}
\begin{lemma}\label{lemma2}
	$f(k)$ is a function on ($\mathbb{R}$). Denote $\eta=z-\frac{1}{z}$, with $dz=\left( \frac{1}{2}+\frac{\eta}{\sqrt{4+\eta^2}}\right) d\eta$.Let $f(z)=g(z-\frac{1}{z})=g(\eta)$, then
	\begin{align}
		\int_{1}^{+\infty}|f(z)|^2dz=\int_{0}^{+\infty}|g(\eta)|^2\left( \frac{1}{2}+\frac{\eta}{\sqrt{4+\eta^2}}\right)d\eta\lesssim \int_{\mathbb{R}}|g(\eta)|^2d\eta.
	\end{align}
\end{lemma}
\begin{lemma}\label{lemma3}
	For $\psi(\eta)\in L^2(\mathbb{R})$, $f(x)\in L^{2,1/2}(\mathbb{R})$, following inequality hold.
	\begin{align}
		|\int_{\mathbb{R}}\int_{x}^{\pm \infty}f(y)e^{-\frac{i}{2}\eta(p(x)-p(y))}\psi(\eta)dyd\eta |&=|\int_{x}^{\pm \infty}f(y)\psi(\frac{1}{2}(p(x)-p(y)))dy|\nonumber\\
		&\lesssim \left( \int_{x}^{\pm \infty}|f(y)|^2dy\right)^{1/2}  \parallel \psi\parallel_2;\\
		\int_{0}^{\pm \infty}\int_{\mathbb{R}}|\int_{x}^{\pm \infty}f(y)e^{-\frac{i}{2}\eta(p(x)-p(y))}&dy|^2d\eta dx  \leq \parallel f\parallel_{2,1/2}^2.
	\end{align}
\end{lemma}
The proof of above lemmas are trivial. In lemma 3, let $s=\frac{p(x)-p(y)}{2}$ with $ds=\sqrt{m^2+1}dy$, then
\begin{align}
	\left( \int_{x}^{\pm \infty}|\psi(\frac{p(x)-p(y)}{2})|^2dy\right)^{1/2}\leq\left( \int_{\mathbb{R}}\frac{|\psi(s)|^2}{\sqrt{m^2+1}}ds\right)^{1/2}\lesssim \parallel \psi\parallel_2.
\end{align}
And we omit the rest part of prove.

\subsection{Large-z Estimates}
\quad From the symmetry reduction (\ref{symPhi1}), we will only consider $[\mu_\pm]_1(x,z)$ for $z\in(1,+\infty)$. For the sake of brevity, denote
\begin{align}
	[\mu_\pm]_1(x,z)-e_1\triangleq n_\pm(x,z)=\left(\begin{array}{cc}
		n_\pm^1 \\
		n_\pm^2
	\end{array}\right),\label{n}
\end{align}
where $e_1$ is identity vector $(1,0)^T$. And we abbreviate
$C_B^0(\mathbb{R}^\pm\times (1,+\infty))$, $ C^0(\mathbb{R}^\pm,L^2(1,+\infty))$, $ L^2(\mathbb{R}^\pm\times (1,+\infty))$ to $C_B^0$, $ C^0$, $ L^2_{xz}$ respectively. Introduce the integral operator $T_\pm$:
\begin{align}\label{T1}
	T_\pm(f)(x,z)=\int_{x}^{\pm\infty}K_\pm(x,y,z)f(y,z)dy,
\end{align}
where integral kernel $K_\pm(x,y,z)$ is
\begin{align}\label{K}
	&K_\pm(x,y,z)=\left(\begin{array}{cc}
		1 & 0 \\
		0 & e^{-\frac{i}{2}(z-1/z)(p(x)-p(y))}
	\end{array}\right)P\nonumber\\
	&=\frac{i m_{x}}{2\left(M^{(2)}+1\right)}\left(\begin{array}{cc}
		0 & 1 \\
		-e^{-\frac{i}{2}(z-1/z)(p(x)-p(y))} & 0
	\end{array}\right)\nonumber\\
	&+\frac{1}{2 z} \frac{m}{q}\left(\begin{array}{cc}
		-i m & 1 \\
		-e^{-\frac{i}{2}(z-1/z)(p(x)-p(y))} & e^{-\frac{i}{2}(z-1/z)(p(x)-p(y))}i m
	\end{array}\right).
\end{align}
Then (\ref{lax3.1}) trans to
\begin{align}\label{eqn}
	&n_\pm=T_\pm(e_1)+T_\pm(n_\pm).	
\end{align}
To obtain the $z$-derivative property  of $n_\pm$, we take the $z$-derivative of above equation and get
\begin{align}\label{eqnk}
	&[n_\pm]_z=n_\pm^1+T_\pm([n_\pm]_z),\hspace{0.5cm}n_\pm^1=[T_\pm]_z(e_1)+[T_\pm]_z(n_\pm).
\end{align}
$[T_\pm]_z$ is also a integral operator with integral kernel $[K_\pm]_z(x,y,z)$:
\begin{align}
	[K_\pm]_z(x,y,z)=&-\frac{i(p(x)-p(y))}{2}\frac{i m_{x}}{2\left(M^{(2)}+1\right)}\left(\begin{array}{cc}
		0 & 0 \\
		-e^{-\frac{i}{2}(z-1/z)(p(x)-p(y))} & 0
	\end{array}\right)\nonumber\\
	&-\frac{1}{z}\frac{mi}{4q}(p(x)-p(y))e^{-\frac{i}{2}(z-1/z)(p(x)-p(y))}\left(\begin{array}{cc}
		0 & 0 \\
		-1 & i m
	\end{array}\right)\nonumber\\
	&-\frac{1}{z^2}\frac{i(p(x)-p(y))}{2}\frac{i m_{x}}{2\left(M^{(2)}+1\right)}\left(\begin{array}{cc}
		0 & 0 \\
		-e^{-\frac{i}{2}(z-1/z)(p(x)-p(y))} & 0
	\end{array}\right)\nonumber\\
	&-\frac{1}{z^2}\frac{m}{q}\left(\begin{array}{cc}
		-i m & 1 \\
		-e^{-\frac{i}{2}(z-1/z)(p(x)-p(y))} & e^{-\frac{i}{2}(z-1/z)(p(x)-p(y))}i m
	\end{array}\right)\nonumber\\
	&-\frac{1}{z^3}\frac{mi}{4q}(p(x)-p(y))e^{-\frac{i}{2}(z-1/z)(p(x)-p(y))}\left(\begin{array}{cc}
		0 & 0 \\
		-1 & i m
	\end{array}\right).\label{Kz}
\end{align}
\begin{lemma}\label{lemma4}
	$T_\pm$ and $[T_\pm]_z$ are integral operators defined above, then  $T_\pm(e_1)(x,z)\in C_B^0\cap C^0\cap L^2_{xz}$ and  $[T_\pm]_z(e_1)(x,z)\in  C^0\cap L^2_{xz}$.
\end{lemma}
\begin{proof}
	$T_\pm(e_1)(x,z)$ is given by
	\begin{align}
		T_\pm(e_1)(x,z)=&\int_{x}^{\pm\infty}\frac{i m_{x}}{2\left(M^{(2)}+1\right)}\left(\begin{array}{cc}
			0  \\
			-e^{-\frac{i}{2}(z-1/z)(p(x)-p(y))}
		\end{array}\right)dy\nonumber\\
		&+\frac{1}{2 z}\int_{x}^{\pm\infty} \frac{m}{q}\left(\begin{array}{cc}
			-i m  \\
			-e^{-\frac{i}{2}(z-1/z)(p(x)-p(y))}
		\end{array}\right)dy.\label{Te1}
	\end{align}
	It   immediately derives to
	\begin{align}
		|T_\pm(e_1)(x,z)|\lesssim \parallel m \parallel_2^2+\parallel m \parallel_1.
	\end{align}
	And from lemma \ref{lemma1}, we only need to estimate the first item of $T_\pm(e_1)(x,z)$. Let $\eta=z-\frac{1}{z}$, with $dk=\left( \frac{1}{2}+\frac{\eta}{\sqrt{4+\eta^2}}\right) d\eta$. Then from lemma \ref{lemma2}, we only need to prove the first integral of (\ref{Te1}): $$H_1(x,\eta)\triangleq\int_{x}^{\pm\infty}\frac{i m_{x}}{2\left(M^{(2)}+1\right)}e^{-\frac{i}{2}\eta(p(x)-p(y))}dy\in C^0(\mathbb{R}^\pm,L^2(\mathbb{R}))\cap L^2(\mathbb{R}^\pm\times \mathbb{R}).$$
	From lemma \ref{lemma3},
	\begin{align}
		\parallel H_1(x,\eta) \parallel_{C^0}\lesssim \parallel m_x\parallel_2,\hspace{0.5cm}
		\parallel H_1(x,\eta) \parallel_{L^2_{xz}}\lesssim \parallel m_x\parallel_{2,1/2}.
	\end{align}
	And  $[T_\pm]_z(e_1)(x,z)$ is  given by
	\begin{align}
		[T_\pm]_z(e_1)(x,z)=&-\int_{x}^{\pm\infty}\frac{i(p(x)-p(y))}{2}\frac{i m_{x}}{2\left(M^{(2)}+1\right)}\left(\begin{array}{cc}
			0 & 0 \\
			-e^{-\frac{i}{2}(z-1/z)(p(x)-p(y))} & 0
		\end{array}\right)dy\nonumber\\
		&-\int_{x}^{\pm\infty}\frac{1}{z}\frac{mi}{4q}(p(x)-p(y))e^{-\frac{i}{2}(z-1/z)(p(x)-p(y))}\left(\begin{array}{cc}
			0 & 0 \\
			-1 & i m
		\end{array}\right)\nonumber\\
		&-\frac{1}{z^2}\frac{i(p(x)-p(y))}{2}\frac{i m_{x}}{2\left(M^{(2)}+1\right)}\left(\begin{array}{cc}
			0 & 0 \\
			-e^{-\frac{i}{2}(z-1/z)(p(x)-p(y))} & 0
		\end{array}\right)\nonumber\\
		&-\frac{1}{z^2}\frac{m}{q}\left(\begin{array}{cc}
			-i m & 1 \\
			-e^{-\frac{i}{2}(z-1/z)(p(x)-p(y))} & e^{-\frac{i}{2}(z-1/z)(p(x)-p(y))}i m
		\end{array}\right)\nonumber\\
		&-\frac{1}{z^3}\frac{mi}{4q}(p(x)-p(y))e^{-\frac{i}{2}(z-1/z)(p(x)-p(y))}\left(\begin{array}{cc}
			0 & 0 \\
			-1 & i m
		\end{array}\right)dy.
	\end{align}
	Similarly from  lemma \ref{lemma1}, we only need to evaluate the first integral. Denote $\eta=z-\frac{1}{z}$ and
	$$H_2(x,\eta)=\int_{x}^{\pm\infty}\frac{i}{2}(p(x)-p(y))\frac{i m_{x}}{2\left(M^{(2)}+1\right)}e^{-\frac{i}{2}\eta(p(x)-p(y))}dy.$$ Note that $|p(x)-p(y)|\leq |x-y|+\parallel m\parallel_{1}.$ Then it can be bound  by analogy with $H_1$
	\begin{align}
		&\parallel H_2(x,\eta) \parallel_{C^0}\lesssim \parallel m_x\parallel_{2,1}+\parallel m_x\parallel_{2}\parallel m\parallel_{1},\\
		&\parallel H_2(x,\eta) \parallel_{L^2_{xz}}\lesssim \parallel m_x\parallel_{2,3/2}+\parallel m_x\parallel_{2,1/2}\parallel m\parallel_{1}.
	\end{align}
\end{proof}

The operator $T_\pm$,  and $[T_\pm]_k$ induce linear mappings, which proposition are given in next lemma.
\begin{lemma}
	The integral operator $T_\pm$ maps $C_B^0\cap C^0\cap L^2_{xz}$ to itself while its $z$-derivative $[T_\pm]_k$ is a integral operator on $ C^0\cap L^2_{xz}$. Moreover, $(I-T^\pm)^{-1} $exists as a bounded	operator on  $C_B^0\cap C^0\cap L^2_{xz}$.
\end{lemma}
\begin{proof}
	In fact,  (\ref{K}) leads to
	\begin{align}
		|K_\pm(x,y,z)|=|P(y,z)|.
	\end{align}
	Accordingly, for any $f(x,z) \in C_B^0(\mathbb{R}^\pm\times (1,+\infty))$,
	\begin{align}
		|T_\pm(f)(x,z)|\leq \int_{x}^{\pm\infty}|P(y,z)|dy \parallel f \parallel_{C_B^0}.
	\end{align}
	Denote $K^n_\pm$ is the integral kernel of Volterra operator $[T_\pm]^n$ as
	\begin{align}
		K^n_\pm(x,y,z)=\int_{x}^{y}\int_{y_1}^{y}...\int_{y_{n-2}}^{z}K_\pm(x,y_1,z)K_\pm(y_1,y_2,z)...K_\pm(y_{n-1},y,z)dy_{n-1}...dy_1,
	\end{align}
	with
	\begin{align}
		|K^n_\pm(x,y,z)|\leq \frac{1}{(n-1)!}\left( \int_{x}^{\pm\infty}|P(y,z)|dy \right) ^{n-1}|P(y,z)|.
	\end{align}
	Then  the standard Volterra theory gives the following operator norm:
	\begin{align}
		\parallel (I-T_\pm)^{-1} \parallel_{\mathcal{B}(C_B^0)}\leq e^{\int_{0}^{\pm\infty}|P(y,1)|dy}.
	\end{align}
	Analogously, $T_\pm$ is a bounded 	operator on $C^0$ with
	\begin{align}
		\parallel (I-T_\pm)^{-1} \parallel_{\mathcal{B}(C^0)}\leq e^{\int_{\mathbb{R}}|P(y,1)|dy}.
	\end{align}
	By some minute  modifications,  there is similar boundedness result for $T_\pm$ on $L^2_{xz}$ with
	\begin{align}
		&\parallel T_\pm \parallel_{\mathcal{B}(L^2_{xz})}\leq \int_{\mathbb{R}}|yP^2(y,1)|dy;\\
		&\parallel (I-T_\pm)^{-1} \parallel_{\mathcal{B}(L^2_{xz})}\leq e^{\int_{\mathbb{R}}|P(y,1)|dy}\int_{\mathbb{R}}|yP^2(y,1)|dy.
	\end{align}
\end{proof}
Then from above lemma, $$\parallel [T_\pm]_z(n_\pm)\parallel_{C^0}\leq\parallel [T_\pm]_z\parallel_{\mathcal{B}(C^0)}\parallel n_\pm \parallel_{L^2_{xz}}, \parallel [T_\pm]_z(n_\pm)\parallel_{L^2_{xz}}\leq\parallel [T_\pm]_z\parallel_{\mathcal{B}(L^2_{xz})}\parallel n_\pm \parallel_{C_B^0},$$
which implies $n_\pm^1\in C^0\cap L^2_{xz}$.
Since the operator $(I-T_\pm)^{-1}$
exist, the equations (\ref{eqn})-(\ref{eqnk})  are  solvable with:
\begin{align}
	&n_\pm(x,z)=(I-T_\pm)^{-1}(T_\pm(e_1))(x,z),\\
	&n_z^\pm(x,z)=(I-T_\pm)^{-1}([T_\pm]_z(e_1)+T_\pm([n_\pm]_z))(x,z).
\end{align}
Combining above Lemmas and the definition  of $n_\pm$ (\ref{n}), we immediately obtain the following property of $\mu_\pm(x,z)$.
\begin{Proposition}\label{mu1}
	Suppose that $u\in H^{4,2}(\mathbb{R})$, then $\mu_\pm(0,z)-I$ belongs in $C_B^0(1,+\infty)\cap L^2(1,+\infty)$, while its $z$-derivative $[\mu_\pm(0,z)]_z$ is in $ L^2(1,+\infty)$.
\end{Proposition}

\subsection{Small-z Estimates}
\quad  Analogously, we will only consider $n_\pm(x,z)$ for $z\in(0,1)$. Consider the change of variable: $z\in(0,1)\to \gamma=\frac{1}{z}\in(1,+\infty)$. From proposition \ref{sym},  we have
\begin{align}
	|[\mu_\pm]_1(x,z)|=|[\sigma_2F(x)^2\mu_\pm(x,\gamma)\sigma_2]_1|\lesssim|[\mu_\pm]_1(x,\gamma)|.\label{symu}
\end{align}
Then the boundedness  of $n_\pm(x,z)$ of $z\in(0,1)$ follows immediately.
Simple calculation gives that for $f(\gamma)\triangleq g(\frac{1}{z})$,
\begin{align}
	\int_{0}^1|g(\frac{1}{z})|^2dz=\int_{1}^{+\infty}\frac{|f(\gamma)|^2}{\gamma^2}d\gamma\leq\int_{1}^{+\infty}|f(\gamma)|^2d\gamma.
\end{align}
So $n_\pm(x,z)\in C^0(\mathbb{R}^\pm,L^2(0,1))\cap L^2(\mathbb{R}^\pm\times(0,1))$ is equivalent to $\mu_\pm(x,\gamma)$ in $ C^0(\mathbb{R}^\pm,L^2(1,+\infty)_\gamma)\cap L^2(\mathbb{R}^\pm\times(1,+\infty)_\gamma)$ which obtained from proposition  \ref{mu1}. However,
\begin{align}
	\int_{0}^1|g'(\frac{1}{z})|^2dz=\int_{1}^{+\infty}|f'(\gamma)|^2\gamma^2d\gamma.
\end{align}
Hence, our goal is to show that $[n_\pm]_z(x,\gamma)$ in $ C^0(\mathbb{R}^\pm,L^{2,1}(1,+\infty)_\gamma)\cap L^{2,1}(\mathbb{R}^\pm\times(1,+\infty)_\gamma)$. From (\ref{eqnk}), it is sufficiently  to show $\gamma[T_\pm]_\gamma(e_1)$ and $\gamma[T_\pm]_\gamma(n_\pm)$ in $ C^0(\mathbb{R}^\pm,L^{2}(1,+\infty)_\gamma)\cap L^{2}(\mathbb{R}^\pm\times(1,+\infty)_\gamma)$. For any bounded continuous function $f(x,\gamma)$ on $C_B^0(\mathbb{R}^\pm\times(1,+\infty)_\gamma)$,
\begin{align}
	|\gamma[T_\pm]_\gamma(f)(x,\gamma)|\leq |\int_{x}^{\pm\infty} \gamma[K_\pm]_\gamma(1,1)^T dy|\parallel f\parallel_{C_B^0(\mathbb{R}^\pm\times(1,+\infty)_\gamma)}.
\end{align}
Consequently,  we just need to prove the $L^2$-integrability of $|\int_{x}^{\pm\infty} \gamma[K_\pm]_\gamma dy|$ in following lemma.
\begin{lemma}
	Suppose that $u\in H^{4,2}(\mathbb{R})$,  then $\int_{x}^{\pm\infty} \gamma[K_\pm]_\gamma(1,1)^T dy$ belongs to $C^0(\mathbb{R}^\pm,L^{2}(1,+\infty)_\gamma)\cap L^{2}(\mathbb{R}^\pm\times(1,+\infty)_\gamma)$.
\end{lemma}
\begin{proof}
	Recall (\ref{Kz}) and have
	\begin{align}
		\gamma[K_\pm]_\gamma(x,y,\gamma)=&-\gamma\frac{i}{2}(p(x)-p(y))\frac{i m_{x}}{2\left(M^{(2)}+1\right)}\left(\begin{array}{cc}
			0 & 0 \\
			-e^{-\frac{i}{2}(\gamma-1/\gamma)(p(x)-p(y))} & 0
		\end{array}\right)\nonumber\\
		&-\frac{mi}{4q}(p(x)-p(y))e^{-\frac{i}{2}(\gamma-1/\gamma)(p(x)-p(y))}\left(\begin{array}{cc}
			0 & 0 \\
			-1 & i m
		\end{array}\right)\nonumber\\
		&-\frac{1}{\gamma}K^0_\pm(x,y,\gamma),
	\end{align}
	where
	\begin{align}
		K^0_\pm(x,y,\gamma)=
		&-\frac{i(p(x)-p(y))}{2}\frac{i m_{x}}{2\left(M^{(2)}+1\right)}\left(\begin{array}{cc}
			0 & 0 \\
			-e^{-\frac{i}{2}(\gamma-1/\gamma)(p(x)-p(y))} & 0
		\end{array}\right)\nonumber\\
		&-\frac{m}{q}\left(\begin{array}{cc}
			-i m & 1 \\
			-e^{-\frac{i}{2}(\gamma-1/\gamma)(p(x)-p(y))} & e^{-\frac{i}{2}(\gamma-1/\gamma)(p(x)-p(y))}i m
		\end{array}\right)\nonumber\\
		&-\frac{1}{\gamma}\frac{mi}{4q}(p(x)-p(y))e^{-\frac{i}{2}(\gamma-1/\gamma)(p(x)-p(y))}\left(\begin{array}{cc}
			0 & 0 \\
			-1 & i m
		\end{array}\right).
	\end{align}
	Then from lemma \ref{lemma1}, we have $\int_{x}^{\pm\infty}\frac{1}{\gamma} K^0_\pm(x,y,\gamma)(1,1)^T dy\in C^0(\mathbb{R}^\pm,L^{2}(1,+\infty)_\gamma)\cap L^{2}(\mathbb{R}^\pm\times(1,+\infty)_\gamma)$ right away. For the second item of $\gamma[K_\pm]_\gamma$, we multiply it by  $(1,1)^T$ and  obtain
	\begin{align}
		\gamma[K_\pm]_\gamma(1,1)^T=-\frac{mi}{4q}(p(x)-p(y))e^{-\frac{i}{2}(\gamma-1/\gamma)(p(x)-p(y))}\left(\begin{array}{cc}
			0  \\
			-1+im
		\end{array}\right).
	\end{align}
	Then we denote
	\begin{align}
		f_1(x,y)=\frac{mi}{4q}(p(x)-p(y))(1-im).
	\end{align}
	 Introduce a new variable $\eta=\gamma-1/\gamma$, with $d\gamma=\left( \frac{1}{2}+\frac{\eta}{\sqrt{4+\eta^2}}\right) d\eta$. Then from lemma \ref{lemma2}, our goal change to seek the $L^2$-integrability for $\eta\in\mathbb{R}$. For any $\psi(\eta)\in L^2(\mathbb{R})$,
	\begin{align}
		|\int_{\mathbb{R}}\int_{x}^{\pm \infty}f(x,y)e^{-\frac{i}{2}\eta(p(x)-p(y))}\psi(\eta)dyd\eta |&=|\int_{x}^{\pm \infty}f(x,y)\psi(\frac{p(x)-p(y)}{2})dy|\nonumber\\
		&\lesssim \left( \int_{x}^{\pm \infty}|f(x,y)|^2dy\right)^{1/2}  \parallel \psi\parallel_2,
	\end{align}
while
\begin{align}
	\left( \int_{x}^{\pm \infty}|f(x,y)|^2dy\right)^{1/2}  &\lesssim\left( \int_{x}^{\pm \infty} \left( |m|+|m|^2\right)^2 \left(|x-y|+\parallel m\parallel_{1} \right)^2dy\right) ^{1/2}\nonumber\\
	&\lesssim\parallel m\parallel_{2,1}+\parallel m\parallel_{2}.
\end{align}
And 	
\begin{align}
	\int_{0}^{\pm \infty}&\int_{\mathbb{R}}|\int_{x}^{\pm \infty}f(x,y)e^{-\frac{i}{2}\eta(p(x)-p(y))}dy|^2d\eta dx  \leq \int_{0}^{\pm \infty}\int_{x}^{\pm \infty}|f(x,y)|^2dydx\nonumber\\
	&\lesssim \int_{0}^{\pm \infty}\int_{x}^{\pm \infty} \left( |m|+|m|^2\right)^2 \left(|x-y|+\parallel m\parallel_{1} \right)^2dydx\nonumber\\
	&\lesssim \int_{0}^{\pm \infty}\int_{0}^{y} (y-x)^2|m|^2+|m|^2dxdy\leq\parallel m\parallel_{2,3/2}+\parallel m\parallel_{2,1/2}.
\end{align}
Finally, we deal with
\begin{align}
	H_3(x,\gamma)\triangleq\int_{x}^{\pm \infty}\gamma\frac{i}{2}(p(x)-p(y))\frac{i m_{x}}{2\left(M^{(2)}+1\right)}e^{-\frac{i}{2}(\gamma-1/\gamma)(p(x)-p(y))}dy.
\end{align}
By partial integration,
\begin{align}
	H_3(x,\gamma)=&\int_{x}^{\pm \infty}\frac{\partial}{\partial y}\left(\frac{i m_{x}(p(x)-p(y))}{2\left(M^{(2)}+1\right)^{3/2}} \right) e^{-\frac{i}{2}(\gamma-1/\gamma)(p(x)-p(y))}dy\nonumber\\
	&-\frac{1}{\gamma}\int_{x}^{\pm \infty}\frac{i m_{x}(p(x)-p(y))}{2\left(M^{(2)}+1\right)}  e^{-\frac{i}{2}(\gamma-1/\gamma)(p(x)-p(y))}dy\nonumber\\
	&=H_{31}(x,\gamma)+\frac{1}{\gamma}H_{32}(x,y,\gamma).
\end{align}
From lemma \ref{lemma1}, only $H_{31}(x,\gamma)$ need to be control. Rewrite it as
\begin{align}
	\int_{x}^{\pm \infty}&\left(\frac{i m_{xx}(p(x)-p(y))}{2\left(M^{(2)}+1\right)^{3/2}}-\frac{i m_{x}}{2\left(M^{(2)}+1\right)}-\frac{3i m_{x}^2m(p(x)-p(y))}{2\left(M^{(2)}+1\right)^{5/2}} \right)\nonumber\\
	& e^{-\frac{i}{2}(\gamma-1/\gamma)(p(x)-p(y))}dy.
\end{align}
Analogously,  it admit
\begin{align}
	\parallel &H_{31}\parallel_{C^0(\mathbb{R}^\pm,L^{2}(1,+\infty)_\gamma)}\lesssim\parallel m_{xx}\parallel_{2,1}+\parallel m_{x}\parallel_{2}+\parallel m\parallel_{2,1}\nonumber;\\
	&\parallel H_{31}\parallel_{L^{2}(\mathbb{R}^\pm\times(1,+\infty)_\gamma)}\lesssim\parallel m_{xx}\parallel_{2,3/2}+\parallel m_{x}\parallel_{2,1/2}+\parallel m\parallel_{2,3/2}.
\end{align}
Combining above proof we obtain the result.
\end{proof}
\begin{Proposition}\label{mu2}
	Suppose that $u\in H^{4,2}(\mathbb{R})$, then $\mu_\pm(0,z)-I$ belongs in $C_B^0(0,1)\cap L^2(0,1)\cap L^{2,1}(1,\infty)$, while its $z$-derivative $[\mu_\pm(0,z)]_z$ is in $ L^2(0,1)$.
\end{Proposition}
Then we begin to prove proposition \ref{pror}.  From (\ref{symr}), it is requisite to shown
$$r(z)=\frac{b(z)}{a(z)},\hspace{0.3cm}r'(z)=\frac{b'(z)}{a(z)}-\frac{b(z)a'(z)}{a^2(z)},\hspace{0.3cm}zr(z)=\frac{zb(z)}{a(z)}\text{ in }L^2(\mathbb{R}).$$
Rewrite (\ref{scatteringcoefficient2}) as
\begin{align}
	&a(z)=(n_-^{1}(0,z)+1)\overline{(n_+^{1}(0,z)+1)}+n_-^{2}(0,z)\overline{n_+^{2}(0,z)},\nonumber\\
	&\overline{b(z)}=n_+^{2}(0,z)-n_-^{2}(0,z)+n_+^{2}(0,z)n_-^{1}(0,z)-n_-^{2}(0,z)n_+^{1}(0,z).
\end{align}
Then proposition \ref{mu1} and \ref{mu2} give the boundedness of $a(z)$, $a'(z)$, $b(z)$, $b'(z)$ and the $L^2$-integrability of $b(z)$, $b'(z)$. So we just need to show $zb(z)\in L^2(\mathbb{R})$. For $|z|>1$, proposition \ref{mu1} and \ref{mu2}  provide $zb(z)\in L^2(1,+\infty)$.  For $|z|<1$,  by change of variable: $z\in(0,1)\to \gamma=\frac{1}{z}\in(1,+\infty)$, simple calculation gives that for $f(\gamma)\triangleq g(\frac{1}{z})$
\begin{align}
	\int_{0}^1|zg(\frac{1}{z})|^2dz=\int_{1}^{+\infty}\frac{|f(\gamma)|^2}{\gamma^4}d\gamma\leq\int_{1}^{+\infty}|f(\gamma)|^2d\gamma.
\end{align}
Together with (\ref{symu}), we have $zn_\pm(0,z)\in L^2(0,1)$. Then from the symmetry (\ref{symPhi1}), we conclude that $zb(z)\in L^2(\mathbb{R})$ and finally obtain proposition \ref{pror}.

\section{Deformation of  the RH problem}\label{sec3}

\quad
The long-time asymptotic  of RHP \ref{RHP2}  is affected by the growth and decay of the exponential function $$e^{\pm2it\theta},\ with \ \theta(z)=-\frac{1}{4}(z-\frac{1}{z})\left[\frac{y}{t}-\frac{8}{(z+\frac{1}{z})^2} \right],$$
which is appearing in both the jump relation and the residue conditions. So we need control the real part of $\pm2it\theta$.
Therefore, in this section, we introduce  a new transform  $M(z)\to M^{(1)}(z)$,  which  make that the  $M^{(1)}(z)$ is well behaved as $t\to \infty$ along any characteristic line.
Let $\xi=\frac{y}{t}$. To obtain asymptotic behavior  of $e^{2it\theta}$ as $t\to \infty$, we consider the real part of $2it\theta$:
\begin{align}
&\text{Re}(2it\theta)=-2t\text{Im}\theta\nonumber\\
&=-2t\text{Im}z\left[ -\frac{\xi}{4}\left(1+|z|^{-2} \right) +2\dfrac{-|z|^6+2|z|^4+(3\text{Re}^2z-\text{Im}^2z)(1+|z|^2)+2|z|^2-1}{\left((\text{Re}^2z-\text{Im}^2z+1)^2+4\text{Re}^2z\text{Im}^2z \right)^2 }  \right]  .\label{Reitheta}
\end{align}
The property of $\text{Im}\theta$ are shown in Figure \ref{figtheta}.

\begin{figure}[p]
	\centering
	\subfigure[]{\includegraphics[width=0.3\linewidth]{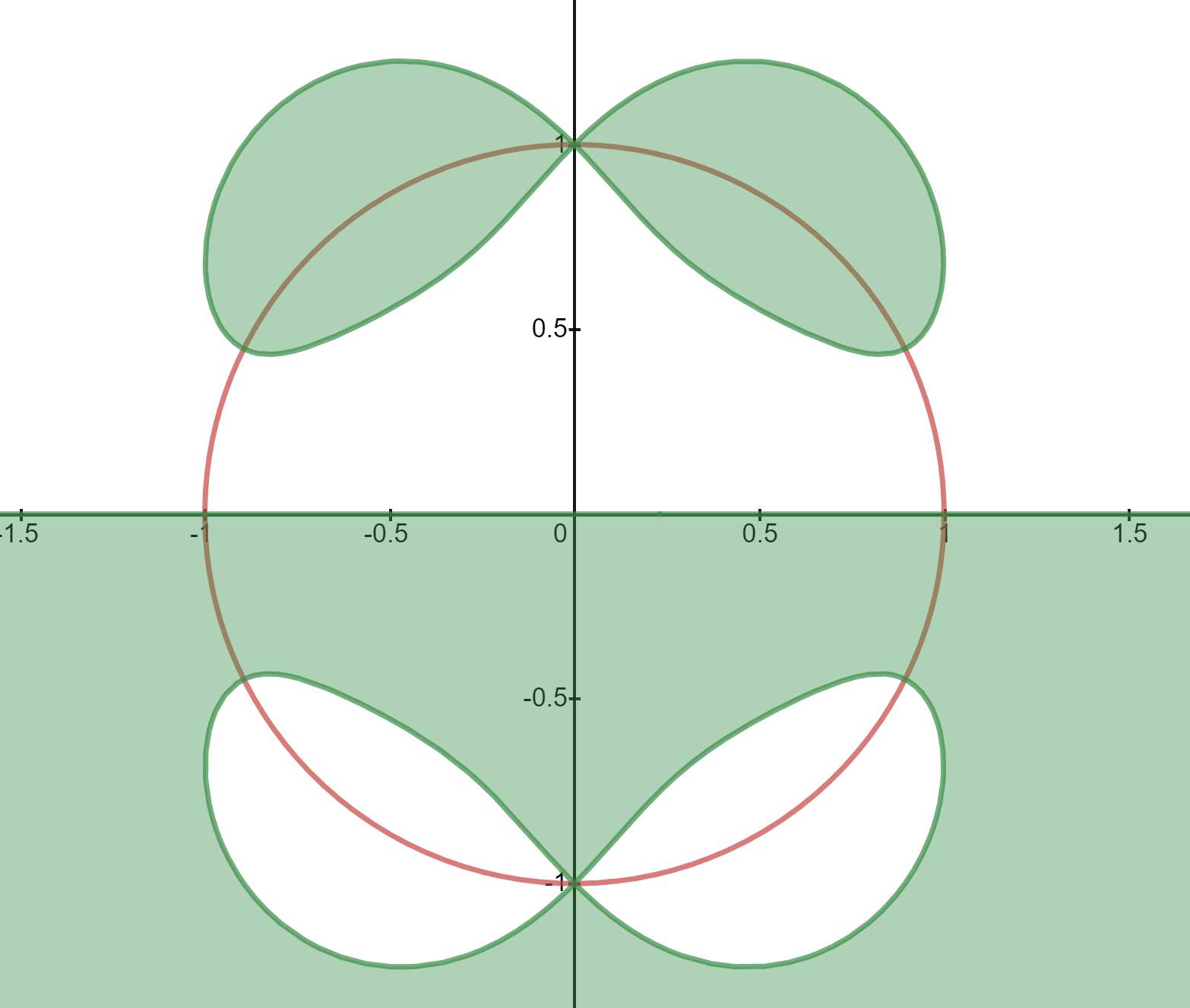}
	\label{fig:desmos-graph}}
	\subfigure[]{\includegraphics[width=0.3\linewidth]{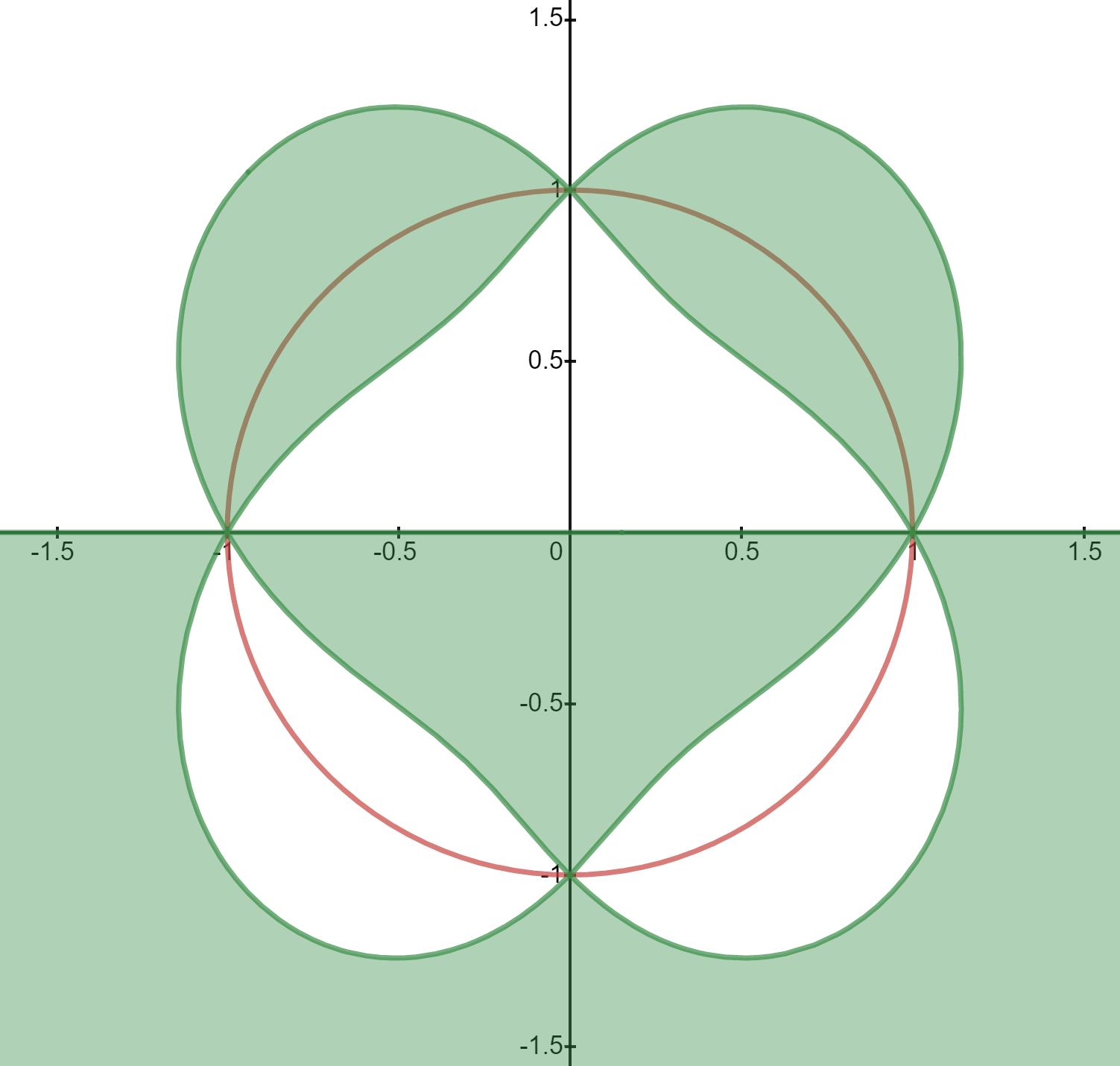}
	\label{fig:desmos-graph-1}}
	\subfigure[]{\includegraphics[width=0.3\linewidth]{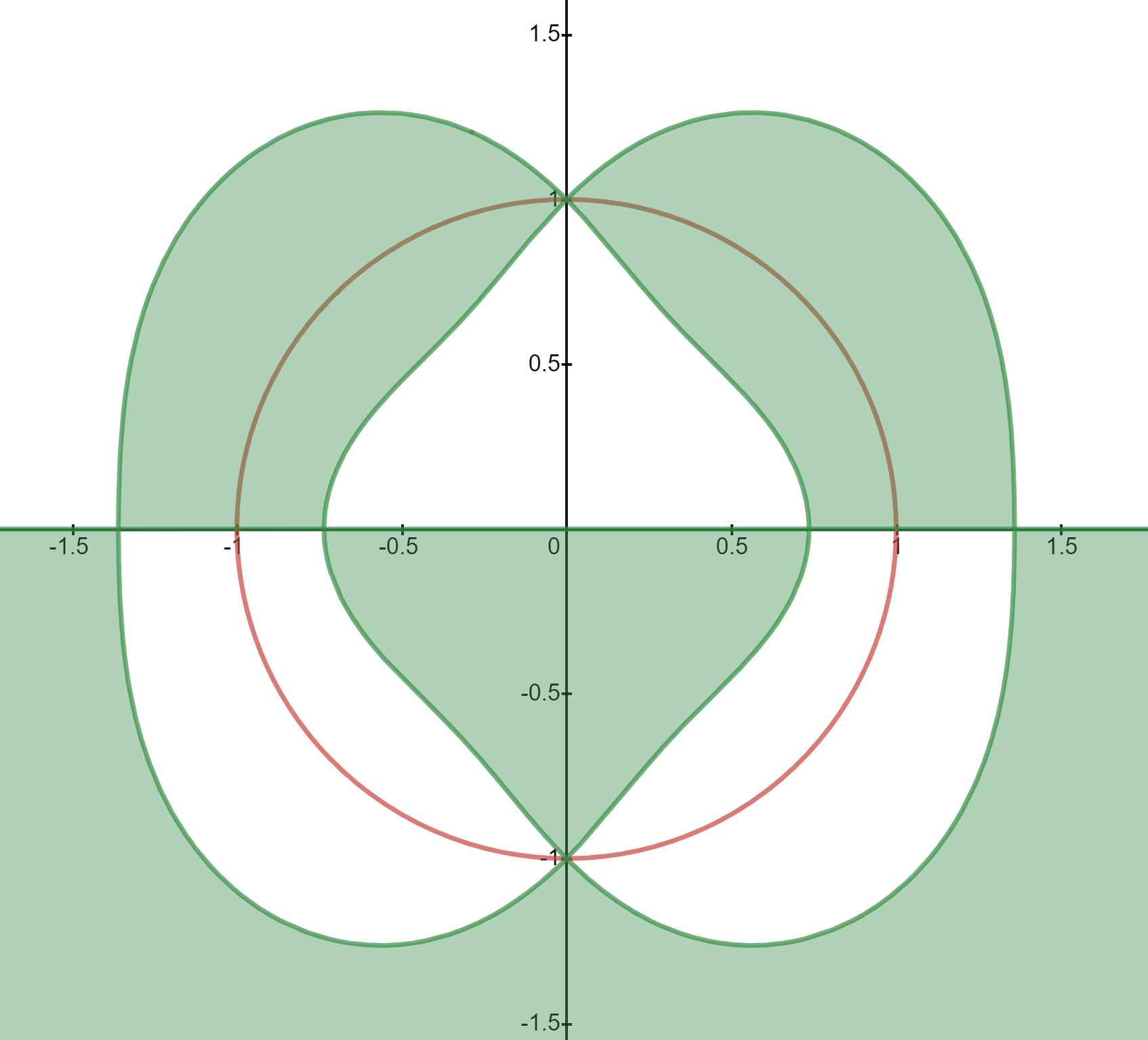}
	\label{fig:desmos-graph-5}}	\subfigure[]{\includegraphics[width=0.3\linewidth]{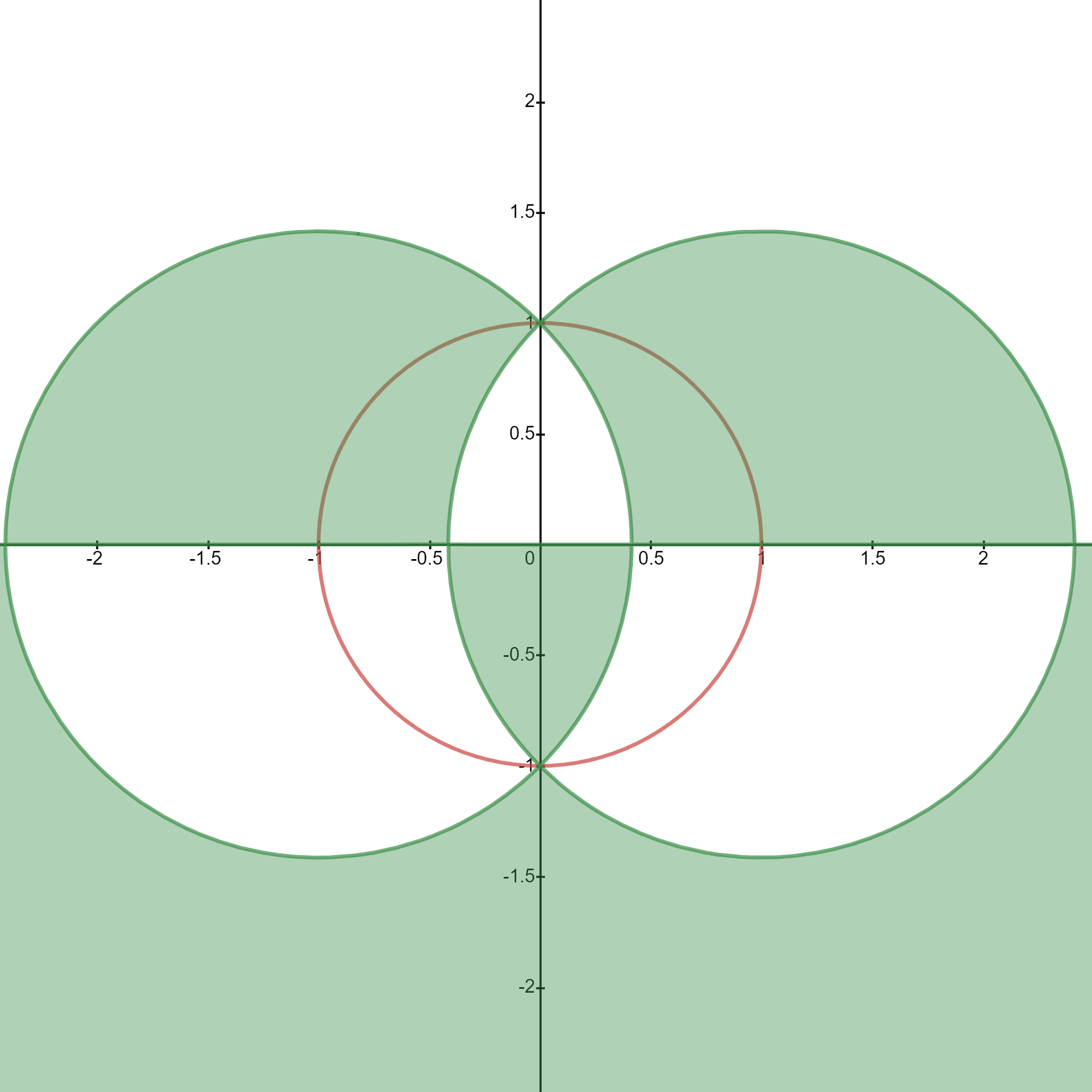}
	\label{fig:desmos-graph-2}}
	\subfigure[]{\includegraphics[width=0.3\linewidth]{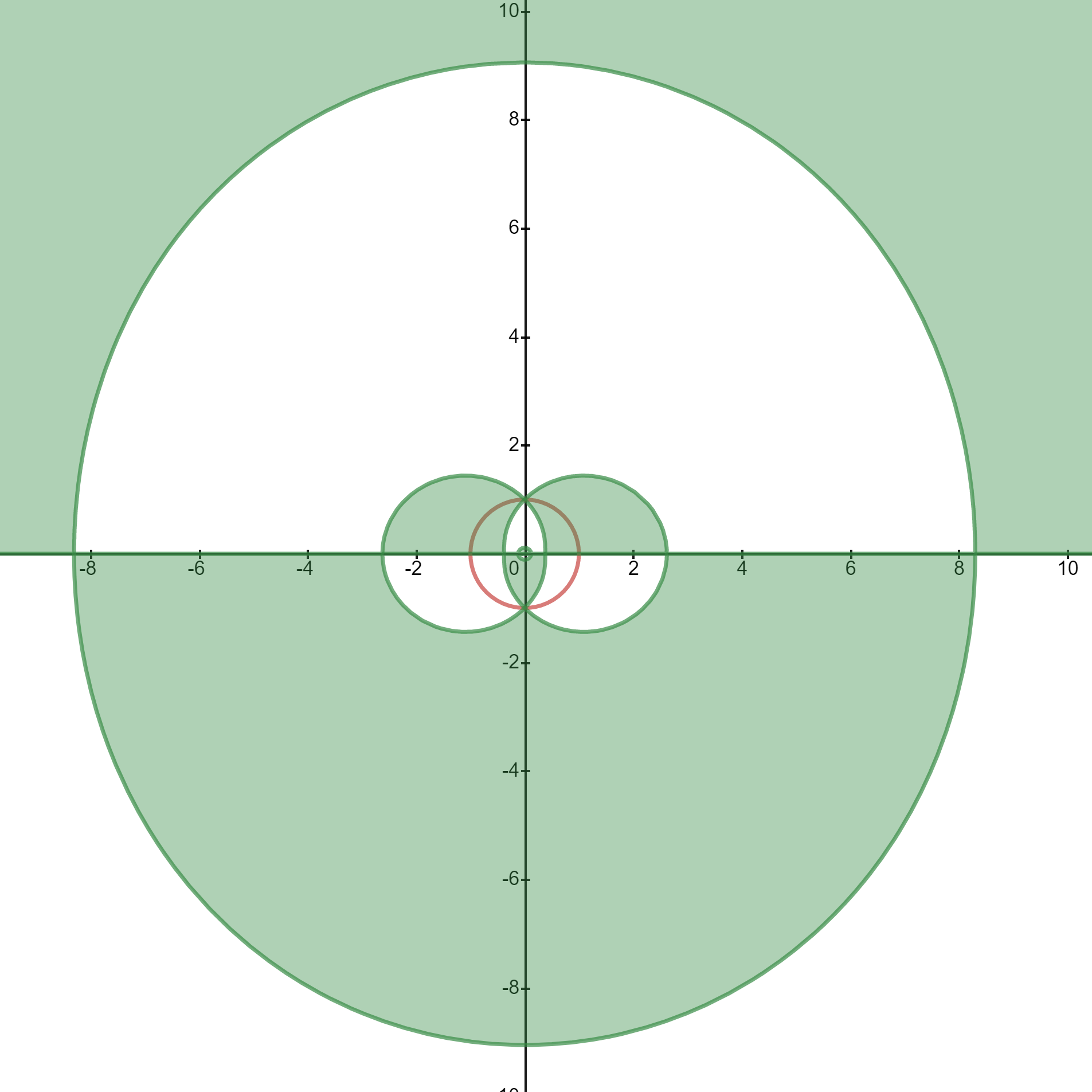}
	\label{fig:desmos-graph-3}}
	\subfigure[]{\includegraphics[width=0.3\linewidth]{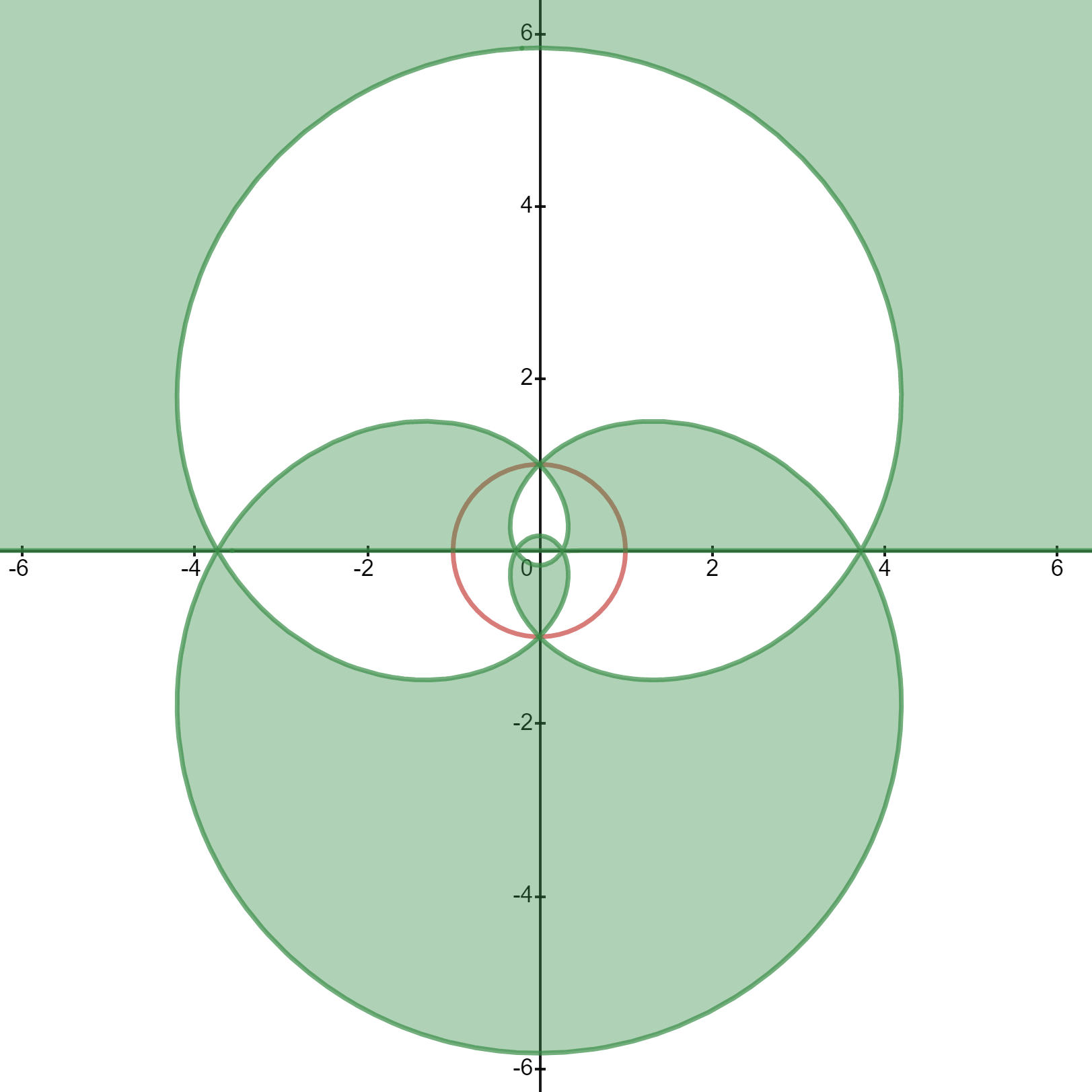}
	\label{fig:desmos-graph-4}}
	\subfigure[]{\includegraphics[width=0.3\linewidth]{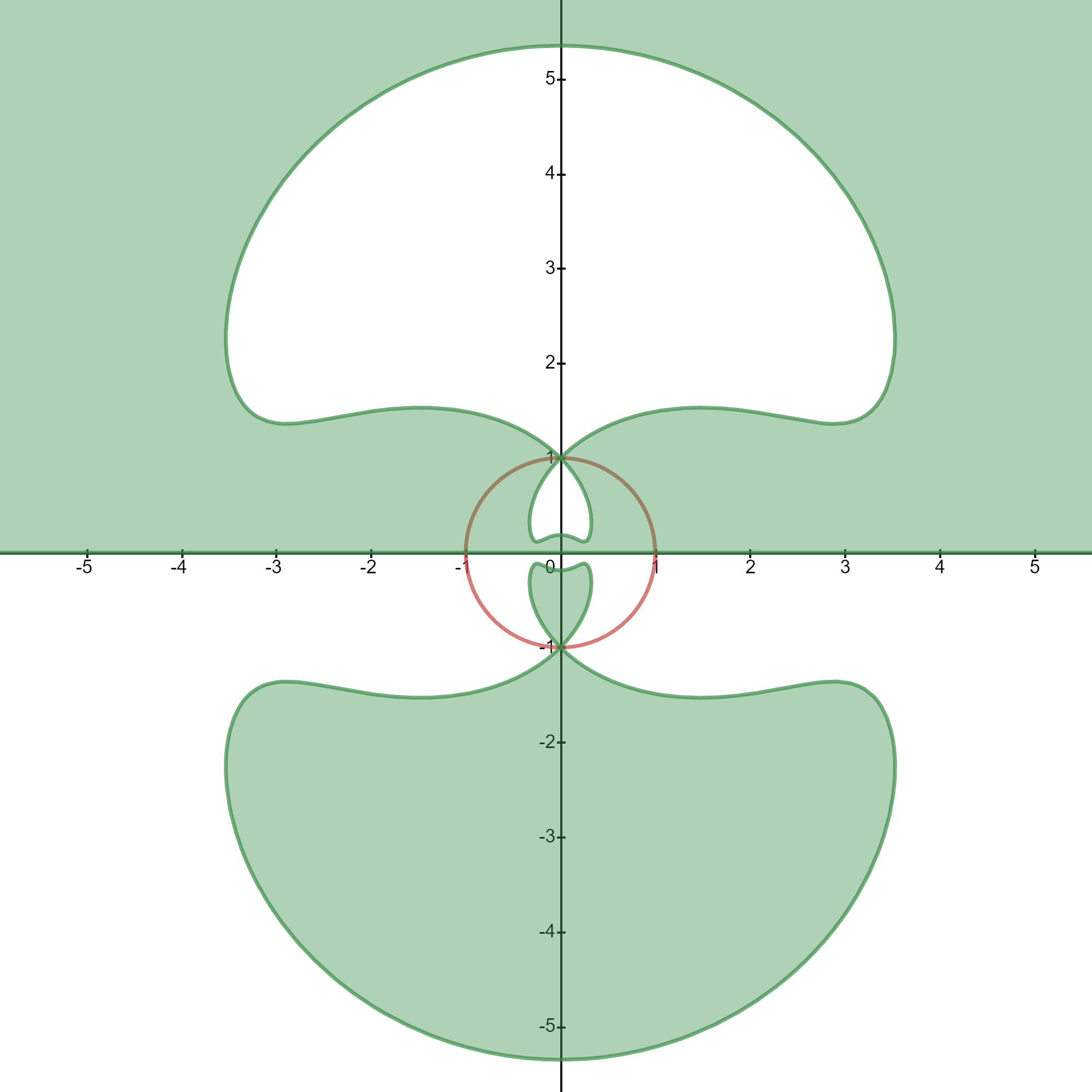}
	\label{fig:desmos-graph-6}}
	\caption{In these figure we take $\xi=2.5,2,1.5,0,-0.1,-0.25,-0.3$ respectively to show all type of $\text{Im}\theta$. The red curve is  unit circle. In the green region, $\text{Im}\theta>0$. It implies that $|e^{2it\theta}|\to 0$ as $t\to\infty$. And $\text{Im}\theta<0$ in the white region, which implies  $|e^{-2it\theta}|\to 0$ as $t\to\infty$.  Moreover,  $\text{Im}\theta=0$ on the green  curve.  }
	\label{figtheta}
\end{figure}
According to the figure, in our paper, we  divide $\xi$ in four case:

Case I: $\xi>2$ (Figure \ref{figtheta} (a)),\hspace{0.7cm} Case II: $0\leq\xi<2$ (Figure \ref{figtheta} (c) and (d)),

Case III: $-\frac{1}{4}<\xi<0$ (Figure \ref{figtheta} (e)),\hspace{0.5cm} Case IV: $\xi<-\frac{1}{4}$ (Figure \ref{figtheta} (g)).\\
In Case I $\xi>2$ and Case IV $\xi<-\frac{1}{4}$, the stationary phase point absents, while  in Case II $0\leq\xi<2$ and Case III $-\frac{1}{4}<\xi<0$, there exist four and eight  stationary phase points denoted as $\xi_1>...>\xi_4$ and $\xi_1>...>\xi_8$ respectively (see figure \ref{phase}). Moreover,  denote $\xi_0=-\infty$, $\xi_{n(\xi)+1}=+\infty$, and introduce some  intervals  when $j=1,...,n(\xi)$, for $0\leq\xi<2$
\begin{align}
I_{j1}=I_{j2}=\left\{ \begin{array}{ll}
\left( \frac{\xi_j+\xi_{j+1}}{2},\xi_j\right) ,\    & j\text{ is  odd number} ,\\[10pt]
\left(\xi_j ,\frac{\xi_j+\xi_{j-1}}{2}\right),   &j\text{ is  even number},
\end{array}\right.\label{In1}\\
I_{j3}=I_{j4}=\left\{ \begin{array}{ll}
\left(\xi_j ,\frac{\xi_j+\xi_{j-1}}{2}\right),\    & j\text{ is  odd number} ,\\[10pt]
\left( \frac{\xi_j+\xi_{j+1}}{2},\xi_j\right) ,   &j\text{ is  even number},
\end{array}\right.
\end{align}
and for $-\frac{1}{4}<\xi<0$,
\begin{align}
I_{j1}=I_{j2}=\left\{ \begin{array}{ll}
\left(\xi_j ,\frac{\xi_j+\xi_{j-1}}{2}\right),\    & j\text{ is  odd number} ,\\[10pt]
\left( \frac{\xi_j+\xi_{j+1}}{2},\xi_j\right) ,   &j\text{ is  even number},
\end{array}\right.\\
I_{j3}=I_{j4}=\left\{ \begin{array}{ll}
\left( \frac{\xi_j+\xi_{j+1}}{2},\xi_j\right) ,\    & j\text{ is  odd number} ,\\[10pt]
\left(\xi_j ,\frac{\xi_j+\xi_{j-1}}{2}\right),   &j\text{ is  even number},
\end{array}\right..\label{In2}
\end{align}
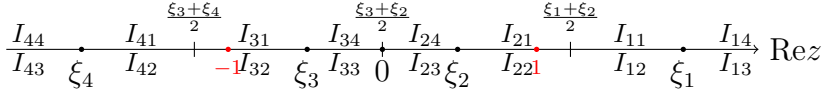
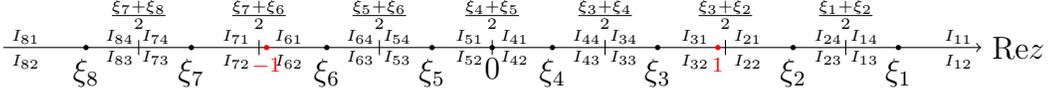
\begin{figure}[h]
	\subfigure[]{
		\begin{tikzpicture}
			\draw[->](-5,0)--(5,0)node[right]{ Re$z$};
			\draw(2.5,0)--(2.5,0.1)node[above]{\scriptsize$\frac{\xi_1+\xi_2}{2}$};
			\draw(2.5,0)--(2.5,-0.1);
			\draw(-2.5,0)--(-2.5,0.1)node[above]{\scriptsize$\frac{\xi_3+\xi_4}{2}$};
			\draw(-2.5,0)--(-2.5,-0.1);
			\draw(0,0)--(0,0.1)node[above]{\scriptsize$\frac{\xi_3+\xi_2}{2}$};
			\draw(0,0)--(0,-0.1);
			\coordinate (I) at (0,0);
			\fill (I) circle (1pt) node[below] {$0$};
			\coordinate (A) at (-4,0);
			\fill (A) circle (1pt) node[below] {$\xi_4$};
			\coordinate (b) at (-1,0);
			\fill (b) circle (1pt) node[below] {$\xi_3$};
			\coordinate (e) at (4,0);
			\fill (e) circle (1pt) node[below] {$\xi_1$};
			\coordinate (f) at (1,0);
			\fill (f) circle (1pt) node[below] {$\xi_2$};
			\coordinate (ke) at (4.7,0.1);
			\fill (ke) circle (0pt) node[below] {\footnotesize$I_{13}$};
			\coordinate (k1e) at (4.7,-0.1);
			\fill (k1e) circle (0pt) node[above] {\footnotesize$I_{14}$};
			\coordinate (le) at (3.3,0.1);
			\fill (le) circle (0pt) node[below] {\footnotesize$I_{12}$};
			\coordinate (l1e) at (3.3,-0.1);
			\fill (l1e) circle (0pt) node[above] {\footnotesize$I_{11}$};
			\coordinate (n2) at (0.57,0.1);
			\fill (n2) circle (0pt) node[below] {\footnotesize$I_{23}$};
			\coordinate (n12) at (0.57,-0.1);
			\fill (n12) circle (0pt) node[above] {\footnotesize$I_{24}$};
			\coordinate (m2) at (1.8,0.1);
			\fill (m2) circle (0pt) node[below] {\footnotesize$I_{22}$};
			\coordinate (m12) at (1.8,-0.1);
			\fill (m12) circle (0pt) node[above] {\footnotesize$I_{21}$};
			\coordinate (k) at (-4.7,0.1);
			\fill (k) circle (0pt) node[below] {\footnotesize$I_{43}$};
			\coordinate (k1) at (-4.7,-0.1);
			\fill (k1) circle (0pt) node[above] {\footnotesize$I_{44}$};
			\coordinate (l) at (-3.2,0.1);
			\fill (l) circle (0pt) node[below] {\footnotesize$I_{42}$};
			\coordinate (l1) at (-3.2,-0.1);
			\fill (l1) circle (0pt) node[above] {\footnotesize$I_{41}$};
			\coordinate (n) at (-0.5,0.1);
			\fill (n) circle (0pt) node[below] {\footnotesize$I_{33}$};
			\coordinate (n1) at (-0.5,-0.1);
			\fill (n1) circle (0pt) node[above] {\footnotesize$I_{34}$};
			\coordinate (m) at (-1.7,0.1);
			\fill (m) circle (0pt) node[below] {\footnotesize$I_{32}$};
			\coordinate (m1) at (-1.7,-0.1);
			\fill (m1) circle (0pt) node[above] {\footnotesize$I_{31}$};
			\coordinate (c) at (-2.05,0);
			\fill[red] (c) circle (1pt) node[below] {\scriptsize$-1$};
			\coordinate (d) at (2.05,0);
			\fill[red] (d) circle (1pt) node[below] {\scriptsize$1$};
			\end{tikzpicture}
			\label{pcase1}}
		\subfigure[]{
		\begin{tikzpicture}
		\draw[->](-6.5,0)--(6.5,0)node[right]{ Re$z$};
		\coordinate (I) at (0,0);
		\fill (I) circle (1pt) node[below] {$0$};
		\coordinate (c) at (-3,0);
		\fill[red] (c) circle (1pt) node[below] {\scriptsize$-1$};
		\coordinate (D) at (3,0);
		\fill[red] (D) circle (1pt) node[below] {\scriptsize$1$};
		\draw(-1.5,0)--(-1.5,0.1)node[above]{\scriptsize$\frac{\xi_5+\xi_6}{2}$};
		\draw(-1.5,0)--(-1.5,-0.1);
		\draw(-4.7,0)--(-4.7,0.1)node[above]{\scriptsize$\frac{\xi_7+\xi_8}{2}$};
		\draw(-4.7,0)--(-4.7,-0.1);
		\draw(-3.1,0)--(-3.1,0.1)node[above]{\scriptsize$\frac{\xi_7+\xi_6}{2}$};
		\draw(-3.1,0)--(-3.1,-0.1);
		\draw(1.5,0)--(1.5,0.1)node[above]{\scriptsize$\frac{\xi_3+\xi_4}{2}$};
		\draw(1.5,0)--(1.5,-0.1);
		\draw(4.7,0)--(4.7,0.1)node[above]{\scriptsize$\frac{\xi_1+\xi_2}{2}$};
		\draw(4.7,0)--(4.7,-0.1);
		\draw(3.1,0)--(3.1,0.1)node[above]{\scriptsize$\frac{\xi_3+\xi_2}{2}$};
		\draw(3.1,0)--(3.1,-0.1);
		\draw(0,0)--(0,0.1)node[above]{\scriptsize$\frac{\xi_4+\xi_5}{2}$};
		\draw(0,0)--(0,-0.1);
		\coordinate (A) at (-5.4,0);
		\fill (A) circle (1pt) node[below] {$\xi_8$};
		\coordinate (b) at (-4,0);
		\fill (b) circle (1pt) node[below] {$\xi_7$};
		\coordinate (C) at (-0.8,0);
		\fill (C) circle (1pt) node[below] {$\xi_5$};
		\coordinate (d) at (-2.2,0);
		\fill (d) circle (1pt) node[below] {$\xi_6$};
		\coordinate (E) at (5.4,0);
		\fill (E) circle (1pt) node[below] {$\xi_1$};
		\coordinate (R) at (4,0);
		\fill (R) circle (1pt) node[below] {$\xi_2$};
		\coordinate (T) at (0.8,0);
		\fill (T) circle (1pt) node[below] {$\xi_4$};
		\coordinate (Y) at (2.2,0);
		\fill (Y) circle (1pt) node[below] {$\xi_3$};
		\coordinate (q) at (6.2,-0.1);
		\fill (q) circle (0pt) node[above] {\tiny$I_{11}$};
		\coordinate (q1) at (6.2,0.05);
		\fill (q1) circle (0pt) node[below] {\tiny$I_{12}$};
		\coordinate (w) at (4.95,-0.1);
		\fill (w) circle (0pt) node[above] {\tiny$I_{14}$};
		\coordinate (w1) at (4.95,0.1);
		\fill (w1) circle (0pt) node[below] {\tiny$I_{13}$};
		\coordinate (e) at (4.47,-0.1);
		\fill (e) circle (0pt) node[above] {\tiny$I_{24}$};
		\coordinate (e1) at (4.47,0.1);
		\fill (e1) circle (0pt) node[below] {\tiny$I_{23}$};
		\coordinate (r) at (3.4,-0.1);
		\fill (r) circle (0pt) node[above] {\tiny$I_{21}$};
		\coordinate (r1) at (3.4,0.05);
		\fill (r1) circle (0pt) node[below] {\tiny$I_{22}$};
		\coordinate (t) at (2.7,-0.1);
		\fill (t) circle (0pt) node[above] {\tiny$I_{31}$};
		\coordinate (t1) at (2.7,0.05);
		\fill (t1) circle (0pt) node[below] {\tiny$I_{32}$};
		\coordinate (y) at (1.75,-0.1);
		\fill (y) circle (0pt) node[above] {\tiny$I_{34}$};
		\coordinate (y1) at (1.75,0.1);
		\fill (y1) circle (0pt) node[below] {\tiny$I_{33}$};
		\coordinate (l) at (1.26,-0.1);
		\fill (l) circle (0pt) node[above] {\tiny$I_{44}$};
		\coordinate (l1) at (1.26,0.1);
		\fill (l1) circle (0pt) node[below] {\tiny$I_{43}$};
		\coordinate (k) at (0.3,-0.1);
		\fill (k) circle (0pt) node[above] {\tiny$I_{41}$};
		\coordinate (k1) at (0.3,0.1);
		\fill (k1) circle (0pt) node[below] {\tiny$I_{42}$};
		\coordinate (q8) at (-6.2,-0.1);
		\fill (q8) circle (0pt) node[above] {\tiny$I_{81}$};
		\coordinate (q18) at (-6.2,0.05);
		\fill (q18) circle (0pt) node[below] {\tiny$I_{82}$};
		\coordinate (w8) at (-4.95,-0.1);
		\fill (w8) circle (0pt) node[above] {\tiny$I_{84}$};
		\coordinate (w18) at (-4.95,0.1);
		\fill (w18) circle (0pt) node[below] {\tiny$I_{83}$};
		\coordinate (e7) at (-4.47,-0.1);
		\fill (e7) circle (0pt) node[above] {\tiny$I_{74}$};
		\coordinate (e17) at (-4.47,0.1);
		\fill (e17) circle (0pt) node[below] {\tiny$I_{73}$};
		\coordinate (7r) at (-3.4,-0.1);
		\fill (7r) circle (0pt) node[above] {\tiny$I_{71}$};
		\coordinate (r17) at (-3.4,0.05);
		\fill (r17) circle (0pt) node[below] {\tiny$I_{72}$};
		\coordinate (t6) at (-2.7,-0.1);
		\fill (t6) circle (0pt) node[above] {\tiny$I_{61}$};
		\coordinate (t16) at (-2.7,0.05);
		\fill (t16) circle (0pt) node[below] {\tiny$I_{62}$};
		\coordinate (y6) at (-1.75,-0.1);
		\fill (y6) circle (0pt) node[above] {\tiny$I_{64}$};
		\coordinate (y16) at (-1.75,0.1);
		\fill (y16) circle (0pt) node[below] {\tiny$I_{63}$};
		\coordinate (l5) at (-1.26,-0.1);
		\fill (l5) circle (0pt) node[above] {\tiny$I_{54}$};
		\coordinate (l15) at (-1.26,0.1);
		\fill (l15) circle (0pt) node[below] {\tiny$I_{53}$};
		\coordinate (k5) at (-0.3,-0.1);
		\fill (k5) circle (0pt) node[above] {\tiny$I_{51}$};
		\coordinate (k15) at (-0.3,0.1);
		\fill (k15) circle (0pt) node[below] {\tiny$I_{52}$};
		\end{tikzpicture}
		\label{phase2}}
	\caption{Figure (a) and (b) are corresponding to the  $0\leq\xi<2$ and  $-\frac{1}{4}<\xi<0$ respectively. In (a), there are four stationary phase points $\xi_1,...\xi_4$ with $\xi_1=-\xi_4=1/\xi_2=-1/\xi_3$. And in (b), there are eight stationary phase points $\xi_1,...\xi_8$ with $\xi_1=-\xi_8=1/\xi_4=-1/\xi_5$ and $\xi_2=-\xi_7=1/\xi_3=-1/\xi_6$.}
	\label{phase}
\end{figure}

For brevity, we denote
\begin{align}
	n(\xi)=\left\{ \begin{array}{ll}
		0,   &\text{as } \xi>2 \text{ and } \xi<-\frac{1}{4},\\[10pt]
		4 , &\text{as } 0\leq\xi<2,\\[10pt]
		8,   &\text{as } -\frac{1}{4}<\xi<0,
	\end{array}\right.
\end{align}
as the number of stationary phase points, and $\mathcal{N}\triangleq\left\lbrace 1,...,4N_1+2N_2\right\rbrace $. Moreover, we introduce a small positive constant $\delta_0$ to give the  partitions $\Delta,\nabla$ and $\Lambda$  of $\mathcal{N}$   as follow:
\begin{align}
&\nabla=\left\lbrace n \in  \mathcal{N}  ;\text{Im}\theta_n< 0\right\rbrace,
\Delta=\left\lbrace n \in  \mathcal{N} ;\text{Im}\theta_n> 0\right\rbrace,\Lambda=\left\lbrace n \in  \mathcal{N}  ;|\text{Im}\theta_n|\leq \delta_0\right\rbrace.\label{devide}
\end{align}
For $\zeta_n$ with $n\in\Delta$, the residue of $M(z)$ at $\zeta_n$ in (\ref{RES1}) grows without bound as $t\to\infty$. Similarly, for $\zeta_n$ with $n\in\nabla$, the residue are  approaching to  $0$. Denote two constants $\mathcal{N}(\Lambda)=|\Lambda|$ and
\begin{equation}
	\rho_0=\min_{n\in\mathcal{N}\setminus \Lambda}\left\lbrace |\text{Im}\theta_n|\right\rbrace >\delta_0.\label{rho0}
\end{equation}
To distinguish different  type of zeros, we further give
\begin{align}
	&\nabla_1=\left\lbrace j \in \left\lbrace 1,...,N_1\right\rbrace  ;\text{Im}\theta(z_j)< 0\right\rbrace,
	\Delta_1=\left\lbrace j \in \left\lbrace 1,...,N_1\right\rbrace  ;\text{Im}\theta(z_j)> 0\right\rbrace,\\
	&\nabla_2=\left\lbrace i \in \left\lbrace 1,...,N_2\right\rbrace  ;\text{Im}\theta(w_i)< 0\right\rbrace,
	\Delta_2=\left\lbrace i \in \left\lbrace 1,...,N_2\right\rbrace  ;\text{Im}\theta(w_i)> 0\right\rbrace,\\
	&\Lambda_1=\left\lbrace j_0 \in \left\lbrace 1,...,N_1\right\rbrace  ;|\text{Im}\theta(z_{j_0})|\leq\delta_0\right\rbrace,\Lambda_2=\left\lbrace i_0 \in \left\lbrace 1,...,N_2\right\rbrace  ;|\text{Im}\theta(w_{i_0)}|\leq\delta_0\right\rbrace.
\end{align}
For the poles $\zeta_n$ with $n\notin\Lambda$, we want to trap them for jumps along small closed loops enclosing themselves respectively. And the jump matrix $V(z)$ (\ref{jumpv})  also needs to  be restricted. Recall the  well known factorizations of  $V(z)$:
\begin{align}
	V(z)&=\left(\begin{array}{cc}
	1 & \bar{r} e^{2it\theta}\\
	0 & 1
\end{array}\right)\left(\begin{array}{cc}
1 & 0 \\
r e^{-2it\theta} & 1
\end{array}\right)\\
&=\left(\begin{array}{cc}
	1 & \\
	\frac{r e^{-2it\theta}}{1+|r|^2} & 1
\end{array}\right)(1+|r|^2)^{\sigma_3}\left(\begin{array}{cc}
1 & \frac{\bar{r} e^{2it\theta}}{1+|r|^2} \\
0 & 1
\end{array}\right).
\end{align}
We will utilize  these factorizations to deform the jump contours so that the oscillating factor $e^{\pm2it\theta}$ are decaying in corresponding region respectively. Note that, $\text{Im}\theta$ has different identities for different case. Namely,  the functions which  will be used  following depend on $\xi$. Denote
\begin{align}
	I(\xi)=\left\{ \begin{array}{ll}
		\emptyset,   &\text{as } \xi>2,\\[4pt]
		(\xi_4,\xi_3)\cup(\xi_2,\xi_1),   &\text{as } 0\geq\xi<2,\\[4pt]
		(-\infty,\xi_8)\cup_{j=1}^3(\xi_{2j+1},\xi_{2j})\cup(\xi_1,+\infty),   &\text{as } -\frac{1}{4}<\xi<0,\\[4pt]
		\mathbb{R} , &\text{as }\xi<-\frac{1}{4}.\\
	\end{array}\right.
\end{align}
Define  functions
\begin{align}
\nu(z)&=-\frac{1}{2\pi }\log (1+|r(z)|^2),\delta (z)=\delta (z,\xi)=\exp\left(-i\int _{I(\xi)}\dfrac{\nu(s) ds}{s-z}\right);\\
T(z)&=T(z,\xi)=\prod_{n\in \Delta}\dfrac{z-\zeta_n}{\bar{\zeta}_n^{-1}z-1}\delta (z,\xi)\nonumber\\
&=\prod_{j\in \Delta_1}\dfrac{z-z_j}{\bar{z}_j^{-1}z-1}\dfrac{z+\bar{z}_j}{z_j^{-1}z+1}\dfrac{z-\bar{z}_j^{-1}}{z_jz-1}\dfrac{z+z_j^{-1}}{\bar{z}_jz+1}\prod_{i\in \Delta_2}\dfrac{z-w_i}{w_iz-1}\dfrac{z+\bar{w}_i}{\bar{w}_iz+1}\delta (z,\xi) \label{T}.
\end{align}
In  the above formulas, we choose the principal branch of power and logarithm functions.

\begin{Proposition}\label{proT}
	The function defined by (\ref{T}) has following properties:\\
	(a) $T$ is meromorphic in $\mathbb{C}\setminus \mathbb{R}$, and for each $n\in\Delta$, $T(z)$ has a simple pole at $\zeta_n$ and a simple zero at $\bar{\zeta}_n$;\\
	(b) $T(z)=\overline{T(-\bar{z})}=T(-z^{-1})=\overline{T^{-1}(\bar{z})}$;\\
	(c) For $z\in I(\xi)$, as z approaching the real axis from above and below, $T$ has boundary values $T_\pm$, which satisfy:
	\begin{equation}
	T_-(z)=(1+|r(z)|^2)T_+(z),\hspace{0.5cm}z\in I(\xi);
	\end{equation}
	(d) $\lim_{z\to \infty}T(z)\triangleq T(\infty)$ with $T(\infty)=1$;\\
	(e) for $z=i$,
	\begin{equation}
	T(i)=\prod_{j\in \Delta_1}\left( \dfrac{i-z_j}{i-\bar{z}_j}\dfrac{i+\bar{z}_j}{i+z_j}\right) \prod_{h\in \Delta_2}\dfrac{i-w_h}{i+w_h}\dfrac{i+\bar{w}_h}{i-\bar{w}_h}\delta (i,\xi)\label{Ti},
	\end{equation}
	and as $z\to i$, $T(z)$ has asymptotic expansion  as
	\begin{equation}
	T(z)=T(i)\left( 1-I_0(\xi)(z-i)\right) +	\mathcal{O}((z-i)^2) \label{expT0},
	\end{equation}
	with
	\begin{equation}
		I_0(\xi)=\frac{1}{2\pi i}\int _{I(\xi)}\dfrac{ \log (1+|r(s)|^2)}{(s-i)^2}ds;
	\end{equation}
	(f) $T(z)$ is continuous at $z=0$, and
	\begin{equation}
	T(0)=T(\infty)=1 \label{T0};
	\end{equation}
	(g) As $z\to \xi_j$ along any ray $\xi_j+e^{i\phi}\mathbb{R}^+$ with $|\phi|<\pi$ ,
	\begin{align}
	|T(z,\xi)-T_j(\xi)(z-\xi_j)^{i\nu(\xi_j)}|\lesssim \parallel r\parallel_{H^{1,1}(\mathbb{R})}|z-\xi_j|^{1/2},\label{T-TJ}
	\end{align}
	where $T_j(\xi) $ is the complex unit
	\begin{align}
	T_j(\xi)=\prod_{n\in \Delta}\dfrac{z-\zeta_n}{\bar{\zeta}_n^{-1}z-1}e^{i\beta(\xi_j,\xi)},
	\end{align}
	for  $j=1,...,n(\xi)$. In above function,
	\begin{align}
	\beta_j(z,\xi)=\int_{I(\xi)}\frac{\nu(s)}{s-z}ds-\log(z-\xi_j)\nu(\xi_j).
	\end{align}
\end{Proposition}
\begin{proof}
	 Properties (a), (b), (d) and (f)  can be obtain by simple calculation from the definition of $T(z)$ in (\ref{T}). And (c)  follows from the Plemelj formula. By the Laurent expansion (e) can be obtained immediately. And for (g), analogously to  \cite{fNLS}, rewrite
	 \begin{align}
	 \delta (z,\xi)=\exp\left(i\beta_j(z,\xi)+i\log(z-\xi_j)\nu(\xi_j) \right) ,
	 \end{align}
	 and note the fact that
	 \begin{align}
	 |(z-\xi_j)^{i\nu(\xi_j)}|\leq e^{-\pi\nu(\xi_j)=\sqrt{1+|r(\xi_j)|^2}}\label{key},
	 \end{align}
	 and
	 $$|\beta_j(z,\xi)-\beta_j(\xi_j,\xi)|\lesssim \parallel r\parallel_{H^{1,0}(\mathbb{R})}|z-\xi_j|^{1/2}.$$
	 The result then follows promptly.
	 For brevity, we  omit computation.
\end{proof}

Additionally, Introduce   a positive constant $\varrho$:
\begin{equation}
	\varrho=\frac{1}{2}\min\left\lbrace \min_{j\in \mathcal{N}}\left\lbrace |\text{Im}\zeta_j|, \right\rbrace ,\min_{j\in \mathcal{N}\setminus\Lambda,\text{Im}\theta(z)=0}|\zeta_j-z|, \min_{  j\in \mathcal{N}}|\zeta_j-i|\right\rbrace .
\end{equation}
By above definition, for every $n\in  \mathcal{N}  $,  $\mathbb{D}_n\triangleq\mathbb{D}(\zeta_n,\varrho)$  are pairwise disjoint and are disjoint  with $\left\lbrace z\in \mathbb{C}|\text{Im} \theta(z)=0 \right\rbrace $ and $\mathbb{R}$. Moreover, $i\notin \mathbb{D}_n$. Further, from the symmetry of poles and $\theta$, this definition guarantee   $\overline{\mathbb{D}}_n\triangleq\mathbb{D}(\bar{\zeta}_n,\varrho)$ have same property. Denote
 a piecewise matrix function
\begin{equation}
	G(z)=\left\{ \begin{array}{ll}
		\left(\begin{array}{cc}
			1 & 0\\
			-C_n(z-\zeta_n)^{-1}e^{-2it\theta_n} & 1
		\end{array}\right),   &\text{as } z\in\mathbb{D}_n,n\in\nabla;\\[12pt]
		\left(\begin{array}{cc}
			1 & -C_n^{-1}(z-\zeta_n)e^{2it\theta_n}\\
			0 & 1
		\end{array}\right),   &\text{as } z\in\mathbb{D}_n,n\in\Delta;\\
		\left(\begin{array}{cc}
		1 & \bar{C}_n(z-\bar{\zeta}_n)^{-1}e^{2it\bar{\theta}_n}\\
		0 & 1
		\end{array}\right),   &\text{as } 	z\in\overline{\mathbb{D}}_n,n\in\nabla;\\
		\left(\begin{array}{cc}
		1 & 0	\\
		\bar{C}_n^{-1}(z-\bar{\zeta}_n)e^{-2it\bar{\theta}_n} & 1
		\end{array}\right),   &\text{as } 	z\in\overline{\mathbb{D}}_n,n\in\Delta;\\
	I &\text{as } 	z \text{ in elsewhere};
	\end{array}\right..\label{funcG}
\end{equation}
Then by using $T(z)$ and $G(z)$, the new  matrix-valued   function $M^{(1)}(z)$  is defined as
\begin{equation}
M^{(1)}(z;y,t)\triangleq M^{(1)}(z)=M(z)G(z)T(z)^{\sigma_3},\label{transm1}
\end{equation}
which then satisfies the following RH problem.

\begin{RHP}\label{RHP3}
	Find a matrix-valued function  $  M^{(1)}(z )$ which satisfies:
	
	$\blacktriangleright$ Analyticity: $M^{(1)}(z)$ is meromorphic in $\mathbb{C}\setminus \Sigma^{(1)}$, where
	\begin{equation}
		\Sigma^{(1)}=\mathbb{R}\cup\left[\underset{n\in\mathcal{N}\setminus\Lambda}{\cup}\left( \overline{\mathbb{D}}_n\cup\mathbb{D}_n\right)  \right] ,
	\end{equation}
	is shown in Figure \ref{fig:zero};
	
	$\blacktriangleright$ Symmetry: $M^{(1)}(z)=\sigma_3\overline{M^{(1)}(-\bar{z})}\sigma_3=\sigma_2\overline{M^{(1)}(\bar{z})}\sigma_2$=$F^{-2}\overline{M^{(1)}(-\bar{z}^{-1})}$;
	
	$\blacktriangleright$ Jump condition: $M^{(1)}$ has continuous boundary values $M^{(1)}_\pm$ on $\Sigma^{(1)}$ and
	\begin{equation}
		M^{(1)}_+(z)=M^{(1)}_-(z)V^{(1)}(z),\hspace{0.5cm}z \in \Sigma^{(1)},
	\end{equation}
	where
	\begin{equation}
		V^{(1)}(z)=\left\{\begin{array}{ll}\left(\begin{array}{cc}
				1 & e^{2it\theta}\bar{r}(z)T^{-2}(z) \\
				0 & 1
			\end{array}\right)
			\left(\begin{array}{cc}
				1 & 0\\
				e^{-2it\theta}r(z)T^2(z) & 1
			\end{array}\right),   &\text{as } z\in 	\mathbb{R}\setminus I(\xi);\\[12pt]
			\left(\begin{array}{cc}
				1 & 0\\
				\frac{e^{-2it\theta}r(z)T_-^{2}(z)}{1+|r(z)|^2} & 1
			\end{array}\right)\left(\begin{array}{cc}
				1 & \frac{e^{2it\theta}\bar{r}(z)T_+^{-2}(z)}{1+|r(z)|^2}\\
				0 & 1
			\end{array}\right),   &\text{as } z\in I(\xi);\\[12pt]
			\left(\begin{array}{cc}
				1 & 0\\
				-C_n(z-\zeta_n)^{-1}T^2(z)e^{-2it\theta_n} & 1
			\end{array}\right),   &\text{as } 	z\in\partial\mathbb{D}_n,n\in\nabla;\\[12pt]
			\left(\begin{array}{cc}
				1 & -C_n^{-1}(z-\zeta_n)T^{-2}(z)e^{2it\theta_n}\\
				0 & 1
			\end{array}\right),   &\text{as } z\in\partial\mathbb{D}_n,n\in\Delta;\\
			\left(\begin{array}{cc}
				1 & \bar{C}_n(z-\bar{\zeta}_n)^{-1}T^{-2}(z)e^{2it\bar{\theta}_n}\\
				0 & 1
			\end{array}\right),   &\text{as } 	z\in\partial\overline{\mathbb{D}}_n,n\in\nabla;\\
			\left(\begin{array}{cc}
				1 & 0	\\
				\bar{C}_n^{-1}(z-\bar{\zeta}_n)e^{-2it\bar{\theta}_n}T^2(z) & 1
			\end{array}\right),   &\text{as } 	z\in\partial\overline{\mathbb{D}}_n,n\in\Delta;\\
		\end{array}\right.;\label{jumpv1}
	\end{equation}
	
	$\blacktriangleright$ Asymptotic behaviors:
	\begin{align}
		M^{(1)}(z;y,t) =& I+\mathcal{O}(z^{-1}),\hspace{0.5cm}z \rightarrow \infty,\\
	M^{(1)}(z;y,t) =&F^{-1} \left[ I+(z-i)\left(\begin{array}{cc}
		0 & -\frac{1}{2}(u+u_x) \\
		-\frac{1}{2}(u-u_x) & 0
	\end{array}\right)\right]\nonumber\\ &e^{\frac{1}{2}c_+\sigma_3}T(i)^{\sigma_3}\big(I-I_0\sigma_3(z-i) \big) +\mathcal{O}\left( (z-i)^2\right);\label{asyM1i}
\end{align}

	$\blacktriangleright$ Residue conditions: $M^{(1)}$ has simple poles at each point $\zeta_n$ and $\bar{\zeta}_n$ for $n\in\Lambda$ with:
	\begin{align}
		&\res_{z=\zeta_n}M^{(1)}(z)=\lim_{z\to \zeta_n}M^{(1)}(z)\left(\begin{array}{cc}
			0 & 0\\
			C_ne^{-2it\theta_n}T^2(\zeta_n) & 0
		\end{array}\right),\\
		&\res_{z=\bar{\zeta}_n}M^{(1)}(z)=\lim_{z\to \bar{\zeta}_n}M^{(1)}(z)\left(\begin{array}{cc}
			0 & -\bar{C}_nT^{-2}(\bar{\zeta}_n)e^{2it\bar{\theta}_n}\\
			0 & 0
		\end{array}\right).
	\end{align}
\end{RHP}

\begin{proof}
	The triangular factors  (\ref{funcG})  trades 	poles $\zeta_n$  and $\bar{\zeta}_n$  to jumps on the disk boundaries $\partial \mathbb{D}_n$ and $\partial \overline{\mathbb{D}}_n$ respectively  for $n\in\mathcal{N}\setminus\Lambda$. Then by simple calculation we can obtain the residues condition and jump condition from (\ref{RES1}), (\ref{RES2}) (\ref{jumpv}), (\ref{funcG}) and (\ref{transm1}). The   analyticity and symmetry of $M^{(1)}(z)$ is directly from its definition, the Proposition \ref{proT}, (\ref{funcG}) and the identities of $M$. As for asymptotic behaviors, from $\lim_{z\to i}G(z)=\lim_{z\to \infty}G(z)=I$ and Proposition \ref{proT} (e), we  obtain the asymptotic behaviors of $M^{(1)}(z)$.
\end{proof}

\begin{figure}[p]
	\centering	
		\subfigure[]{\includegraphics[width=0.4\linewidth]{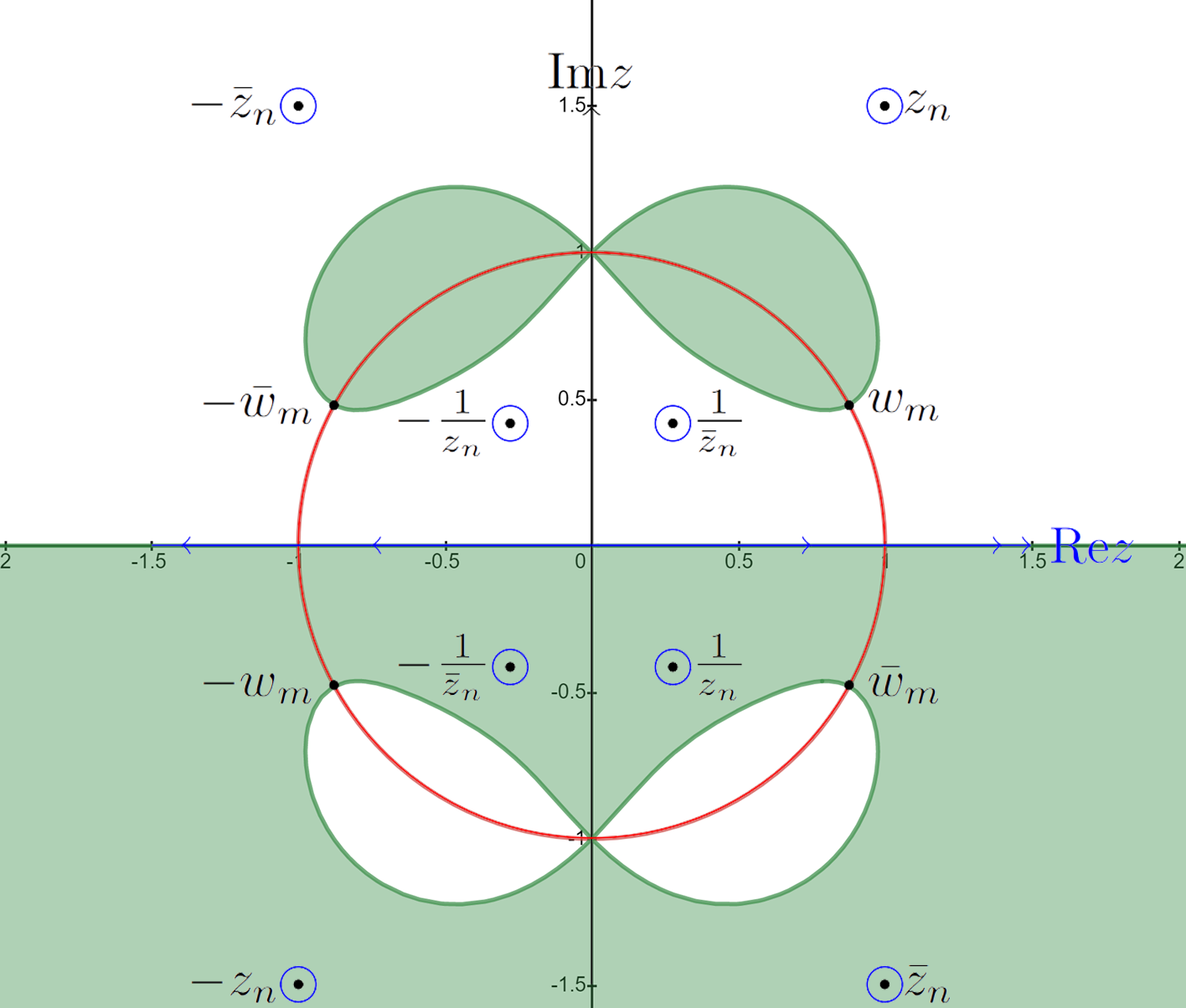}
		\label{zero1}}
	\subfigure[]{\includegraphics[width=0.5\linewidth]{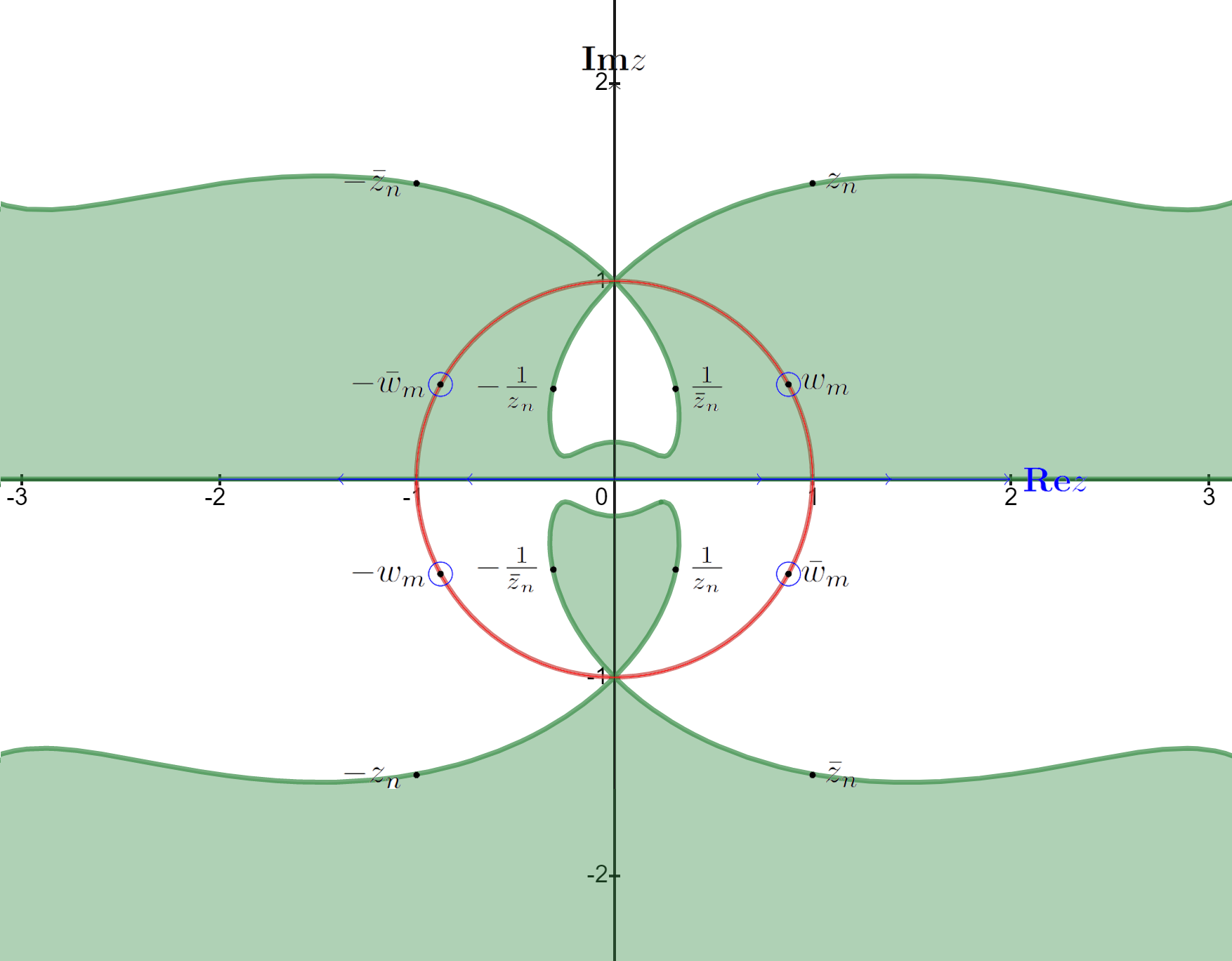}
		\label{zero2}}
	\caption{Subfigure (a) and (b) are respectively corresponding to $\xi=2.5$ and $\xi=-0.3$.  $\mathbb{R}$ and  the small circles  constitute $\Sigma^{(1)}$. And the other cases of $\xi$ are similar. For (a), because Im$\theta(w_m)=0$, it remain the pole of $M^{(1)}$. And Im$\theta(z_n)\neq0$, so we change it to  jump on $\mathbb{D}_n$. As for (b), Im$\theta(z_n)=0$ while Im$\theta(w_m)\neq0$. so we keep $w_m$ as a pole and trad $z_n$ for jumps.    }	\label{fig:zero}
\end{figure}

\section{Mixed $\bar{\partial}$-RH Problem }\label{sec4}

\quad  In this section,  we make continuous extension to the jump matrix $V^{1}$  to remove the jump from $\mathbb{R}$. Besides, the new problem is hoped to  takes advantage of the decay/growth of $e^{2it\theta(z)}$ for $z\notin\mathbb{R}$. For this purpose, we   introduce some new  regions and contours relyed on $\xi$:

1. for the case  $\xi<-\frac{1}{4}$ and  $\xi>2$,
\begin{align}
	&\Omega_{2n+1}=\left\lbrace z\in\mathbb{C}|n\pi \leq\arg z \leq n\pi+\varphi \right\rbrace ,\\
	&\Omega_{2n+2}=\left\lbrace z\in\mathbb{C}|(n+1)\pi -\varphi\leq\arg z \leq (n+1)\pi \right\rbrace,
\end{align}
where $n=0,1$. And
\begin{align}
&\Sigma_k=e^{(k-1)i\pi/2+i\varphi}R_+,\hspace{0.5cm}k=1,3,\\
&\Sigma_k=e^{ki\pi/2-i\varphi}R_+,\hspace{0.5cm}k=2,4,
\end{align}
which is the boundary of $\Omega_k$ respectively. In addition, for these cases, let
\begin{align}
&\Omega(\xi)=\underset{k=1,...,4}{\cup}\Omega_k,\\
&\Sigma^{(2)}(\xi)=\underset{n\in\mathcal{N}\setminus\Lambda}{\cup}\left( \partial\overline{\mathbb{D}}_n\cup\partial\mathbb{D}_n\right)  ,\
\tilde{\Sigma}(\xi)=\Sigma_1\cup\Sigma_2\cup\Sigma_{3}\cup\Sigma_{4},
\end{align}
 which are shown in Figure \ref{figR2}.
And $\frac{\pi}{8}>\varphi>0$ is an fixed sufficiently small angle  achieving following conditions:\\
a.  $\cos 2 \varphi>\frac{4}{\xi}-1$ for $\xi>2$;\\
b.  $\cos 2 \varphi>-\frac{1}{2\xi}-1$ for $\xi<-\frac{1}{4}$;\\
c.  each $\Omega_i$ doesn't intersect $\left\lbrace z\in\mathbb{C}|\text{Im }\theta(z)=0\right\rbrace  $ and  any of $\mathbb{D}_n$ or $\overline{\mathbb{D}}_n$.

2.  for the case $-\frac{1}{4}<\xi<2$, $l\in(0,\frac{|\xi_{j+(-1)^j}-\xi_j|}{2\cos\varphi})$
\begin{align}
&\Sigma_{jk}(\xi)=\left\{\begin{array}{lll}
\xi_j+e^{i[(k/2+1/2+j)\pi+(-1)^{j+1}\varphi]}l,&\ 0>\xi>-0.25\\
\xi_j+e^{i[(k/2+j)\pi+(-1)^j\varphi]}l,&\ 0\leq\xi< 2\end{array}\right.,\hspace{0.2cm}k=1,3,\\
&\Sigma_{jk}(\xi)=\left\{\begin{array}{lll}
\xi_j+e^{i[(k/2+j)\pi+(-1)^j\varphi]}l,&\ 0>\xi>-0.25\\
\xi_j+e^{i[(k/2+1/2+j)\pi+(-1)^{j+1}\varphi]}l,&\ 0\leq\xi< 2\end{array}\right.,\hspace{0.2cm}k=2,4,
\end{align}
where $j=2,...,n(\xi)-1$ and for $j=1,n(\xi)$
\begin{align}
&\Sigma_{j1}(\xi)=\left\{\begin{array}{lll}
\xi_j+e^{(1+j)\pi i+(-1)^{j+1}i\varphi}\mathbb{R}^+,&\ 0>\xi>-0.25\\
\xi_j+e^{j\pi i+(-1)^ji\varphi}l,&\ 0\leq\xi< 2\end{array}\right.,\\
&\Sigma_{j2}(\xi)=\left\{\begin{array}{lll}
\xi_j+e^{(1+j)\pi i+(-1)^ji\varphi}\mathbb{R}^+,&\ 0>\xi>-0.25\\
\xi_j+e^{j\pi i+(-1)^{j+1}i\varphi }l,&\ 0\leq\xi< 2\end{array}\right.,\\
&\Sigma_{j3}(\xi)=\left\{\begin{array}{lll}
\xi_j+e^{j\pi i+(-1)^{j+1}i\varphi }l,&\ 0>\xi>-0.25\\
\xi_j+e^{(1+j)\pi i+(-1)^ji\varphi}\mathbb{R}^+,&\ 0\leq\xi< 2\end{array}\right.,\\
&\Sigma_{j4}(\xi)=\left\{\begin{array}{lll}
\xi_j+e^{j\pi i+(-1)^ji\varphi}l,&\ 0>\xi>-0.25\\
\xi_j+e^{(1+j)\pi i+(-1)^{j+1}i\varphi}\mathbb{R}^+,&\ 0\leq\xi< 2\end{array}\right..
\end{align}
Moreover, for $ l\in(0,\frac{|\xi_{j+(-1)^j}-\xi_j|}{2\cos\varphi})$,
\begin{align}
&\Sigma_{j\pm}'=\left\{\begin{array}{lll}
\frac{\xi_{j+1}+\xi_j}{2}+e^{i\pi }l,&\ 0>\xi>-0.25,\ j=1,...,n(\xi)-1\\
\frac{\xi_{j-1}+\xi_j}{2}-e^{i\pi }l,&\ 0\leq\xi< 2,\ j=2,...,n(\xi)\end{array}\right..
\end{align}
For convenience, denote $\Sigma_{n(\xi)\pm}'=\emptyset$ when $0>\xi>-0.25$ and $\Sigma_{1\pm}'=\emptyset$ when $ 0\leq\xi< 2$.
And $\frac{\pi}{8}>\varphi>0$ is an fixed sufficiently small angle  achieving following conditions:\\
1.  each $\Omega_i$ doesn't intersect $\left\lbrace z\in\mathbb{C};\text{Im }\theta(z)=0\right\rbrace  $ and  any of $\mathbb{D}_n$ or $\overline{\mathbb{D}}_n$,\\
2. $2\tan\varphi>\xi_{n(\xi)/2}-\xi_{n(\xi)/2+1}$.\\
This contours separate complex plane $\mathbb{C}$ into  sectors  shown in Figure \ref{FigOmig}.
\begin{figure}[p]
	\subfigure[]{
		\begin{tikzpicture}
		\draw[cyan!20, fill=cyan!20] (-5,0.5)--(-5,-0.5)--(-4,0)--(-2.5,-0.6)--(-1,0)--(0,-0.5)--(1,0)--(2.5,-0.6)--(4,0)--(5,-0.5)--(5,0.5)--(4,0)--(2.5,0.6)--(1,0)--(0,0.5)--(-1,0)--(-2.5,0.6)--(-4,0)--(-5,0.5);
		\draw(-4,0)--(-5,0.5)node[above]{\scriptsize$\Sigma_{44}$};
		\draw[-<](-4,0)--(-4.5,0.25);
		\draw(-4,0)--(-2.5,0.6);
		\draw[->](-4,0)--(-3.25,-0.3)node[below]{\scriptsize$\Sigma_{42}$};
		\draw(-4,0)--(-5,-0.5)node[below]{\scriptsize$\Sigma_{43}$};
		\draw[->](-4,0)--(-3.25,0.3)node[above]{\scriptsize$\Sigma_{41}$};
		\draw(-4,0)--(-2.5,-0.6);
		\draw[-<](-4,0)--(-4.5,-0.25);
		\draw(-1,0)--(0,0.5);
		\draw[->](-1,0)--(-0.5,0.25)node[above]{\scriptsize$\Sigma_{34}$};
		\draw(-1,0)--(-2.5,0.6);
		\draw[-<](-1,0)--(-1.75,-0.3)node[below]{\scriptsize$\Sigma_{32}$};
		\draw(-1,0)--(0,-0.5);
		\draw[-<](-1,0)--(-1.75,0.3)node[above]{\scriptsize$\Sigma_{31}$};
		\draw(-1,0)--(-2.5,-0.6);
		\draw[->](-1,0)--(-0.5,-0.25)node[below]{\scriptsize$\Sigma_{33}$};
		\draw[dashed](-5,0)--(5,0)node[right]{ Re$z$};
		\draw(1,0)--(0,0.5);
		\draw[-<](1,0)--(0.5,0.25)node[above]{\scriptsize$\Sigma_{24}$};
		\draw(1,0)--(2.5,0.6);
		\draw[->](1,0)--(1.75,-0.3)node[below]{\scriptsize$\Sigma_{22}$};
		\draw(1,0)--(0,-0.5);
		\draw[->](1,0)--(1.75,0.3)node[above]{\scriptsize$\Sigma_{21}$};
		\draw(1,0)--(2.5,-0.6);
		\draw[-<](1,0)--(0.5,-0.25)node[below]{\scriptsize$\Sigma_{23}$};
		\draw(4,0)--(5,0.5)node[above]{\scriptsize$\Sigma_{14}$};
		\draw[->](4,0)--(4.5,0.25);
		\draw(4,0)--(2.5,0.6);
		\draw[-<](4,0)--(3.25,-0.3)node[below]{\scriptsize$\Sigma_{12}$};
		\draw(4,0)--(5,-0.5)node[below]{\scriptsize$\Sigma_{13}$};
		\draw[-<](4,0)--(3.25,0.3)node[above]{\scriptsize$\Sigma_{11}$};
		\draw(4,0)--(2.5,-0.6);
		\draw[->](4,0)--(4.5,-0.25);
		\draw[->](2.5,0)--(2.5,0.6)node[above]{\scriptsize$\Sigma_{2+}'$};
		\draw[->](2.5,0)--(2.5,-0.6)node[below]{\scriptsize$\Sigma_{2-}'$};
		\draw[->](-2.5,0)--(-2.5,0.6)node[above]{\scriptsize$\Sigma_{4+}'$};
		\draw[->](-2.5,0)--(-2.5,-0.6)node[below]{\scriptsize$\Sigma_{4-}'$};
		\draw[->](0,0)--(0,0.5)node[above]{\scriptsize$\Sigma_{3+}'$};
		\draw[->](0,0)--(0,-0.5)node[below]{\scriptsize$\Sigma_{3-}'$};
		\coordinate (I) at (0,0);
		\fill (I) circle (1pt) node[below] {$0$};
		\coordinate (A) at (-4,0);
		\fill (A) circle (1pt) node[below] {$\xi_4$};
		\coordinate (b) at (-1,0);
		\fill (b) circle (1pt) node[below] {$\xi_3$};
		\coordinate (e) at (4,0);
		\fill (e) circle (1pt) node[below] {$\xi_1$};
		\coordinate (f) at (1,0);
		\fill (f) circle (1pt) node[below] {$\xi_2$};
		\coordinate (ke) at (4.7,0.1);
		\fill (ke) circle (0pt) node[below] {\tiny$\Omega_{13}$};
		\coordinate (k1e) at (4.7,-0.1);
		\fill (k1e) circle (0pt) node[above] {\tiny$\Omega_{14}$};
		\coordinate (le) at (3,0.1);
		\fill (le) circle (0pt) node[below] {\tiny$\Omega_{12}$};
		\coordinate (l1e) at (3,-0.1);
		\fill (l1e) circle (0pt) node[above] {\tiny$\Omega_{11}$};
		\coordinate (n2) at (0.27,0.1);
		\fill (n2) circle (0pt) node[below] {\tiny$\Omega_{23}$};
		\coordinate (n12) at (0.27,-0.1);
		\fill (n12) circle (0pt) node[above] {\tiny$\Omega_{24}$};
		\coordinate (m2) at (2.25,0.1);
		\fill (m2) circle (0pt) node[below] {\tiny$\Omega_{22}$};
		\coordinate (m12) at (2.25,-0.1);
		\fill (m12) circle (0pt) node[above] {\tiny$\Omega_{21}$};
		\coordinate (k) at (-4.7,0.1);
		\fill (k) circle (0pt) node[below] {\tiny$\Omega_{43}$};
		\coordinate (k1) at (-4.7,-0.1);
		\fill (k1) circle (0pt) node[above] {\tiny$\Omega_{44}$};
		\coordinate (l) at (-3,0.1);
		\fill (l) circle (0pt) node[below] {\tiny$\Omega_{42}$};
		\coordinate (l1) at (-3,-0.1);
		\fill (l1) circle (0pt) node[above] {\tiny$\Omega_{41}$};
			\coordinate (n) at (-0.27,0.1);
		\fill (n) circle (0pt) node[below] {\tiny$\Omega_{33}$};
		\coordinate (n1) at (-0.27,-0.1);
		\fill (n1) circle (0pt) node[above] {\tiny$\Omega_{34}$};
		\coordinate (m) at (-2.2,0.1);
		\fill (m) circle (0pt) node[below] {\tiny$\Omega_{32}$};
		\coordinate (m1) at (-2.2,-0.1);
		\fill (m1) circle (0pt) node[above] {\tiny$\Omega_{31}$};
		\coordinate (c) at (-2,0);
		\fill[red] (c) circle (1pt) node[below] {\scriptsize$-1$};
		\coordinate (d) at (2,0);
		\fill[red] (d) circle (1pt) node[below] {\scriptsize$1$};
		\end{tikzpicture}
		\label{case1}}
	\subfigure[]{
		\begin{tikzpicture}
		\draw[cyan!20, fill=cyan!20](-6.5,0.9)--(-6.5,-0.9)--(-5.4,0)--(-4.7,-0.6)--(-4,0)--(-3.1,-0.7)--(-2.2,0)--(-1.5,-0.6)--(-0.8,0)
--(-0,-0.7)--(0.8,0)--(1.5,-0.6)--(2.2,0)--(3.1,-0.7)--(4,0)--(4.7,-0.6)--(5.4,0)--(6.5,-0.9)--(6.5,0.9)--(5.4,0)--(4.7,0.6)
--(4,0)--(3.1,0.7)--(2.2,0)--(1.5,0.6)--(0.8,0)--(-0,0.7)--(-0.8,0)--(-1.5,0.6)--(-2.2,0)--(-3.1,0.7)--(-4,0)--(-4.7,0.6)--(-5.4,0)--(-6.5,0.9);
		\draw[dashed](-6.5,0)--(6.5,0)node[right]{ Re$z$};
		\coordinate (I) at (0,0);
		\fill (I) circle (1pt) node[below] {$0$};
		\coordinate (c) at (-3,0);
		\fill[red] (c) circle (1pt) node[below] {\scriptsize$-1$};
		\coordinate (D) at (3,0);
		\fill[red] (D) circle (1pt) node[below] {\scriptsize$1$};
			\draw(-0.8,0)--(-0,0.7);
		\draw[->](-0.8,0)--(-0.4,0.35)node[above]{\scriptsize$\Sigma_{51}$};
		\draw(-0.8,0)--(-1.5,0.6);
		\draw[-<](-0.8,0)--(-1.15,-0.3)node[below]{\scriptsize$\Sigma_{53}$};
		\draw(-0.8,0)--(-0,-0.7);
		\draw[-<](-0.8,0)--(-1.15,0.3)node[above]{\scriptsize$\Sigma_{54}$};
		\draw(-0.8,0)--(-1.5,-0.6);
		\draw[->](-0.8,0)--(-0.4,-0.35)node[below]{\scriptsize$\Sigma_{52}$};
		\draw(-2.2,0)--(-1.5,0.6);
		\draw[-<](-2.2,0)--(-2.65,0.35)node[above]{\scriptsize$\Sigma_{61}$};
		\draw(-2.2,0)--(-1.5,-0.6);
		\draw[->](-2.2,0)--(-1.85,-0.3)node[below]{\scriptsize$\Sigma_{63}$};
		\draw(-2.2,0)--(-3.1,0.7);
		\draw[->](-2.2,0)--(-1.85,0.3)node[above]{\scriptsize$\Sigma_{64}$};
		\draw(-2.2,0)--(-3.1,-0.7);
		\draw[-<](-2.2,0)--(-2.65,-0.35)node[below]{\scriptsize$\Sigma_{62}$};
		\draw(-5.4,0)--(-6.5,0.9)node[above]{\scriptsize$\Sigma_{81}$};
		\draw[-<](-5.4,0)--(-5.95,0.45);
		\draw(-5.4,0)--(-4.7,0.6);
		\draw[->](-5.4,0)--(-5.05,-0.3)node[below]{\scriptsize$\Sigma_{83}$};
		\draw(-5.4,0)--(-6.5,-0.9)node[below]{\scriptsize$\Sigma_{82}$};
		\draw[->](-5.4,0)--(-5.05,0.3)node[above]{\scriptsize$\Sigma_{84}$};
		\draw(-5.4,0)--(-4.7,-0.6);
		\draw[-<](-5.4,0)--(-5.95,-0.45);
		\draw(-4,0)--(-3.1,0.7);
		\draw[->](-4,0)--(-3.55,0.35)node[above]{\scriptsize$\Sigma_{71}$};
		\draw(-4,0)--(-4.7,0.6);
		\draw[-<](-4,0)--(-4.35,-0.3)node[below]{\scriptsize$\Sigma_{73}$};
		\draw(-4,0)--(-3.1,-0.7);
		\draw[-<](-4,0)--(-4.35,0.3)node[above]{\scriptsize$\Sigma_{74}$};
		\draw(-4,0)--(-4.7,-0.6);
		\draw[->](-4,0)--(-3.55,-0.35)node[below]{\scriptsize$\Sigma_{72}$};
		\draw[->](-1.5,0)--(-1.5,0.6)node[above]{\scriptsize$\Sigma_{5,+}'$};
		\draw[->](-1.5,0)--(-1.5,-0.6)node[below]{\scriptsize$\Sigma_{5,-}'$};
		\draw[->](-4.7,0)--(-4.7,0.6)node[above]{\scriptsize$\Sigma_{7,+}'$};
		\draw[->](-4.7,0)--(-4.7,-0.6)node[below]{\scriptsize$\Sigma_{7,-}'$};
		\draw[->](-3.1,0)--(-3.1,0.7)node[above]{\scriptsize$\Sigma_{6,+}'$};
		\draw[->](-3.1,0)--(-3.1,-0.7)node[below]{\scriptsize$\Sigma_{6,-}'$};
		\draw(0.8,0)--(0,0.7);
		\draw[-<](0.8,0)--(0.4,0.35)node[above]{\scriptsize$\Sigma_{41}$};
		\draw(0.8,0)--(1.5,0.6);
		\draw[->](0.8,0)--(1.15,-0.3)node[below]{\scriptsize$\Sigma_{43}$};
		\draw(0.8,0)--(0,-0.7);
		\draw[->](0.8,0)--(1.15,0.3)node[above]{\scriptsize$\Sigma_{44}$};
		\draw(0.8,0)--(1.5,-0.6);
		\draw[-<](0.8,0)--(0.4,-0.35)node[below]{\scriptsize$\Sigma_{42}$};
		\draw(2.2,0)--(1.5,0.6);
		\draw[->](2.2,0)--(2.65,0.35)node[above]{\scriptsize$\Sigma_{31}$};
		\draw(2.2,0)--(1.5,-0.6);
		\draw[-<](2.2,0)--(1.85,-0.3)node[below]{\scriptsize$\Sigma_{33}$};
		\draw(2.2,0)--(3.1,0.7);
		\draw[-<](2.2,0)--(1.85,0.3)node[above]{\scriptsize$\Sigma_{34}$};
		\draw(2.2,0)--(3.1,-0.7);
		\draw[->](2.2,0)--(2.65,-0.35)node[below]{\scriptsize$\Sigma_{32}$};
		\draw(5.4,0)--(6.5,0.9)node[above]{\scriptsize$\Sigma_{11}$};
		\draw[->](5.4,0)--(5.95,0.45);
		\draw(5.4,0)--(4.7,0.6);
		\draw[-<](5.4,0)--(5.05,-0.3)node[below]{\scriptsize$\Sigma_{13}$};
		\draw(5.4,0)--(6.5,-0.9)node[below]{\scriptsize$\Sigma_{12}$};
		\draw[-<](5.4,0)--(5.05,0.3)node[above]{\scriptsize$\Sigma_{14}$};
		\draw(5.4,0)--(4.7,-0.6);
		\draw[->](5.4,0)--(5.95,-0.45);
		\draw(4,0)--(3.1,0.7);
		\draw[-<](4,0)--(3.55,0.35)node[above]{\scriptsize$\Sigma_{21}$};
		\draw(4,0)--(4.7,0.6);
		\draw[->](4,0)--(4.35,-0.3)node[below]{\scriptsize$\Sigma_{23}$};
		\draw(4,0)--(3.1,-0.7);
		\draw[->](4,0)--(4.35,0.3)node[above]{\scriptsize$\Sigma_{24}$};
		\draw(4,0)--(4.7,-0.6);
		\draw[-<](4,0)--(3.55,-0.35)node[below]{\scriptsize$\Sigma_{22}$};
		\draw[->](1.5,0)--(1.5,0.6)node[above]{\scriptsize$\Sigma_{3+}'$};
		\draw[->](1.5,0)--(1.5,-0.6)node[below]{\scriptsize$\Sigma_{3-}'$};
		\draw[->](4.7,0)--(4.7,0.6)node[above]{\scriptsize$\Sigma_{1+}'$};
		\draw[->](4.7,0)--(4.7,-0.6)node[below]{\scriptsize$\Sigma_{1-}'$};
		\draw[->](3.1,0)--(3.1,0.7)node[above]{\scriptsize$\Sigma_{2+}'$};
		\draw[->](3.1,0)--(3.1,-0.7)node[below]{\scriptsize$\Sigma_{2-}'$};
		\draw[->](0,0)--(0,0.7)node[above]{\scriptsize$\Sigma_{4+}'$};
		\draw[->](0,0)--(0,-0.7)node[below]{\scriptsize$\Sigma_{4-}'$};
		\coordinate (A) at (-5.4,0);
		\fill (A) circle (1pt) node[below] {$\xi_8$};
		\coordinate (b) at (-4,0);
		\fill (b) circle (1pt) node[below] {$\xi_7$};
		\coordinate (C) at (-0.8,0);
		\fill (C) circle (1pt) node[below] {$\xi_5$};
		\coordinate (d) at (-2.2,0);
		\fill (d) circle (1pt) node[below] {$\xi_6$};
		\coordinate (E) at (5.4,0);
		\fill (E) circle (1pt) node[below] {$\xi_1$};
		\coordinate (R) at (4,0);
		\fill (R) circle (1pt) node[below] {$\xi_2$};
		\coordinate (T) at (0.8,0);
		\fill (T) circle (1pt) node[below] {$\xi_4$};
		\coordinate (Y) at (2.2,0);
		\fill (Y) circle (1pt) node[below] {$\xi_3$};
		\coordinate (q) at (6.2,-0.1);
		\fill (q) circle (0pt) node[above] {\tiny$\Omega_{11}$};
		\coordinate (q1) at (6.2,0.05);
		\fill (q1) circle (0pt) node[below] {\tiny$\Omega_{12}$};
		\coordinate (w) at (4.95,-0.1);
		\fill (w) circle (0pt) node[above] {\tiny$\Omega_{14}$};
		\coordinate (w1) at (4.95,0.1);
		\fill (w1) circle (0pt) node[below] {\tiny$\Omega_{13}$};
		\coordinate (e) at (4.47,-0.1);
		\fill (e) circle (0pt) node[above] {\tiny$\Omega_{24}$};
		\coordinate (e1) at (4.47,0.1);
		\fill (e1) circle (0pt) node[below] {\tiny$\Omega_{23}$};
		\coordinate (r) at (3.4,-0.1);
		\fill (r) circle (0pt) node[above] {\tiny$\Omega_{21}$};
		\coordinate (r1) at (3.4,0.05);
		\fill (r1) circle (0pt) node[below] {\tiny$\Omega_{22}$};
		\coordinate (t) at (2.7,-0.1);
		\fill (t) circle (0pt) node[above] {\tiny$\Omega_{31}$};
		\coordinate (t1) at (2.7,0.05);
		\fill (t1) circle (0pt) node[below] {\tiny$\Omega_{32}$};
		\coordinate (y) at (1.75,-0.1);
		\fill (y) circle (0pt) node[above] {\tiny$\Omega_{34}$};
		\coordinate (y1) at (1.75,0.1);
		\fill (y1) circle (0pt) node[below] {\tiny$\Omega_{33}$};
		\coordinate (l) at (1.26,-0.1);
		\fill (l) circle (0pt) node[above] {\tiny$\Omega_{44}$};
		\coordinate (l1) at (1.26,0.1);
		\fill (l1) circle (0pt) node[below] {\tiny$\Omega_{43}$};
		\coordinate (k) at (0.3,-0.1);
		\fill (k) circle (0pt) node[above] {\tiny$\Omega_{41}$};
		\coordinate (k1) at (0.3,0.1);
		\fill (k1) circle (0pt) node[below] {\tiny$\Omega_{42}$};
			\coordinate (q8) at (-6.2,-0.1);
		\fill (q8) circle (0pt) node[above] {\tiny$\Omega_{81}$};
		\coordinate (q18) at (-6.2,0.05);
		\fill (q18) circle (0pt) node[below] {\tiny$\Omega_{82}$};
		\coordinate (w8) at (-4.95,-0.1);
		\fill (w8) circle (0pt) node[above] {\tiny$\Omega_{84}$};
		\coordinate (w18) at (-4.95,0.1);
		\fill (w18) circle (0pt) node[below] {\tiny$\Omega_{83}$};
		\coordinate (e7) at (-4.47,-0.1);
		\fill (e7) circle (0pt) node[above] {\tiny$\Omega_{74}$};
		\coordinate (e17) at (-4.47,0.1);
		\fill (e17) circle (0pt) node[below] {\tiny$\Omega_{73}$};
		\coordinate (7r) at (-3.4,-0.1);
		\fill (7r) circle (0pt) node[above] {\tiny$\Omega_{71}$};
		\coordinate (r17) at (-3.4,0.05);
		\fill (r17) circle (0pt) node[below] {\tiny$\Omega_{72}$};
		\coordinate (t6) at (-2.7,-0.1);
		\fill (t6) circle (0pt) node[above] {\tiny$\Omega_{61}$};
		\coordinate (t16) at (-2.7,0.05);
		\fill (t16) circle (0pt) node[below] {\tiny$\Omega_{62}$};
		\coordinate (y6) at (-1.75,-0.1);
		\fill (y6) circle (0pt) node[above] {\tiny$\Omega_{64}$};
		\coordinate (y16) at (-1.75,0.1);
		\fill (y16) circle (0pt) node[below] {\tiny$\Omega_{63}$};
		\coordinate (l5) at (-1.26,-0.1);
		\fill (l5) circle (0pt) node[above] {\tiny$\Omega_{54}$};
		\coordinate (l15) at (-1.26,0.1);
		\fill (l15) circle (0pt) node[below] {\tiny$\Omega_{53}$};
		\coordinate (k5) at (-0.3,-0.1);
		\fill (k5) circle (0pt) node[above] {\tiny$\Omega_{51}$};
		\coordinate (k15) at (-0.3,0.1);
		\fill (k15) circle (0pt) node[below] {\tiny$\Omega_{52}$};
		\end{tikzpicture}
		\label{case2}}
	\caption{Figure (a) and (b) are corresponding to the  $0\leq\xi<2$ and  $-\frac{1}{4}<\xi<0$ respectively. $\Sigma_{ij}$ separate
		complex plane $\mathbb{C}$ into some  sectors denoted by $\Omega_{ij}$.}
	\label{FigOmig}
\end{figure}
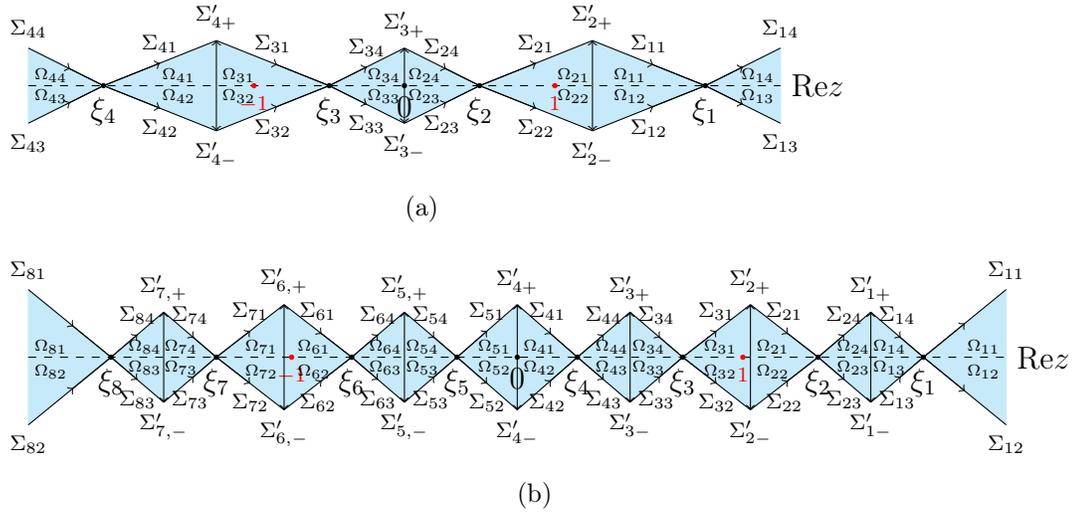
 In addition, for these two cases, let
\begin{align}
	&\Omega(\xi)=\underset{j=1,..,n(\xi)}{\underset{k=1,...,4,}{\cup}}\Omega_{jk},\hspace{0.5cm} \Omega_\pm(\xi)=\mathbb{C}\setminus\Omega,\\
	&\tilde{\Sigma}(\xi)=\left( \underset{j=1,..,n(\xi)}{\underset{k=1,...,4,}{\cup}}\Sigma_{jk}\right) \cup\left(\underset{j=1,...,n(\xi)}{\cup}\Sigma'{j\pm} \right) ,\hspace{0.5cm}\\
	&\Sigma^{(2)}(\xi)=\tilde{\Sigma}(\xi)\underset{n\in\mathcal{N}\setminus\Lambda}{\cup}\left( \partial\overline{\mathbb{D}}_n\cup\partial\mathbb{D}_n\right)  .
\end{align}

\begin{figure}[H]
	\centering
		\begin{tikzpicture}[node distance=2cm]
		\draw[pink!30, fill=pink!30] (0,0)--(3,-0.5)--(3,0.5)--(0,0)--(-3,-0.5)--(-3,0.5)--(0,0);
		\draw(0,0)--(3,0.5)node[above]{$\Sigma_1$};
		\draw(0,0)--(-3,0.5)node[left]{$\Sigma_2$};
		\draw(0,0)--(-3,-0.5)node[left]{$\Sigma_3$};
		\draw(0,0)--(3,-0.5)node[right]{$\Sigma_4$};
		\draw[->](-4,0)--(4,0)node[right]{ Re$z$};
		\draw[->](0,-3)--(0,3)node[above]{ Im$z$};
		\draw[-latex](0,0)--(-1.5,-0.25);
		\draw[-latex](0,0)--(-1.5,0.25);
		\draw[-latex](0,0)--(1.5,0.25);
		\draw[-latex](0,0)--(1.5,-0.25);
		\coordinate (C) at (-0.2,2.2);
		\coordinate (D) at (2.2,0.2);
		\fill (D) circle (0pt) node[right] {\footnotesize $\Omega_1$};
		\coordinate (J) at (-2.2,-0.2);
		\fill (J) circle (0pt) node[left] {\footnotesize $\Omega_3$};
		\coordinate (k) at (-2.2,0.2);
		\fill (k) circle (0pt) node[left] {\footnotesize $\Omega_2$};
		\coordinate (k) at (2.2,-0.2);
		\fill (k) circle (0pt) node[right] {\footnotesize $\Omega_4$};
		\coordinate (I) at (0.2,0);
		\fill (I) circle (0pt) node[below] {$0$};
		\draw[red] (2,0) arc (0:360:2);
		\draw[blue] (2,3) circle (0.12);
		\draw[blue][->](0,0)--(-1.5,0);
		\draw[blue][->](-1.5,0)--(-2.8,0);
		\draw[blue][->](0,0)--(1.5,0);
		\draw[blue][->](1.5,0)--(2.8,0);
		\draw[blue][->](0,2.7)--(0,2.2);
		\draw[blue][->](0,1.6)--(0,0.8);
		\draw[blue][->](0,-2.7)--(0,-2.2);
		\draw[blue][->](0,-1.6)--(0,-0.8);
		\coordinate (A) at (2,3);
		\coordinate (B) at (2,-3);
		\coordinate (C) at (-0.5546996232,0.8320505887);
		\coordinate (D) at (-0.5546996232,-0.8320505887);
		\coordinate (E) at (0.5546996232,0.8320505887);
		\coordinate (F) at (0.5546996232,-0.8320505887);
		\coordinate (G) at (-2,3);
		\coordinate (H) at (-2,-3);
		\coordinate (I) at (2,0);
		\draw[blue] (2,-3) circle (0.12);
		\draw[blue] (-0.55469962326,0.8320505887) circle (0.12);
		\draw[blue] (0.5546996232,0.8320505887) circle (0.12);
		\draw[blue] (-0.5546996232,-0.8320505887) circle (0.12);
		\draw[blue] (0.5546996232,-0.8320505887) circle (0.12);
		\draw[blue] (-2,3) circle (0.12);
		\draw[blue] (-2,-3) circle (0.12);
		\coordinate (J) at (1.7320508075688774,1);
		\coordinate (K) at (1.7320508075688774,-1);
		\coordinate (L) at (-1.7320508075688774,1);
		\coordinate (M) at (-1.7320508075688774,-1);
		\fill (A) circle (1pt) node[right] {$z_n$};
		\fill (B) circle (1pt) node[right] {$\bar{z}_n$};
		\fill (C) circle (1pt) node[left] {$-\frac{1}{z_n}$};
		\fill (D) circle (1pt) node[left] {$-\frac{1}{\bar{z}_n}$};
		\fill (E) circle (1pt) node[right] {$\frac{1}{\bar{z}_n}$};
		\fill (F) circle (1pt) node[right] {$\frac{1}{z_n}$};
		\fill (G) circle (1pt) node[left] {$-\bar{z}_n$};
		\fill (H) circle (1pt) node[left] {$-z_n$};
		\fill (I) circle (1pt) node[above] {$1$};
		\fill (J) circle (1pt) node[right] {$w_m$};
		\fill (K) circle (1pt) node[right] {$\bar{w}_m$};
		\fill (L) circle (1pt) node[left] {$-\bar{w}_m$};
		\fill (M) circle (1pt) node[left] {$-w_m$};
		\end{tikzpicture}
	\caption{The yellow region is $\Omega(\xi)$. The blue circle around poles not on $\left\lbrace z\in\mathbb{C}|\text{Im }\theta(z)=0\right\rbrace  $(here take $z_n$ as an example) constitute $\Sigma^{(2)}(\xi)$ together. }
	\label{figR2}
\end{figure}
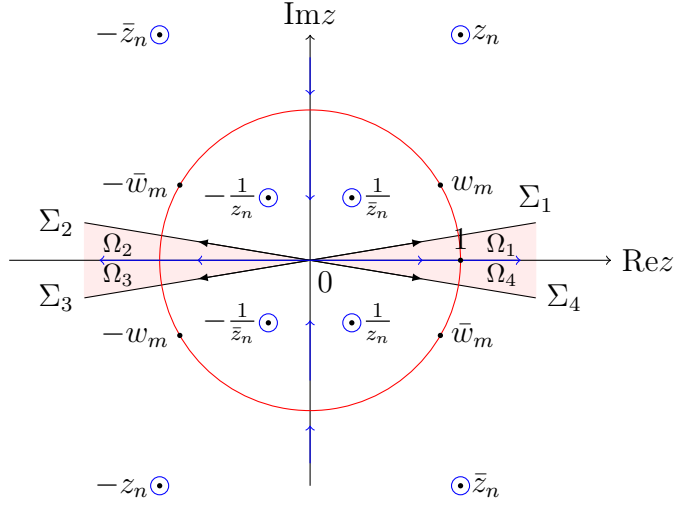

\begin{lemma}
	Set $\xi=\frac{y}{t}\in(-\infty,-0.25)\cup(2,+\infty)$. And $f(x)=x+\frac{1}{x}$ is a  real-valued function for $x\in\mathbb{R}$. Then the imaginary part of phase function (\ref{Reitheta}) $\text{Im }\theta(z)$ have following estimation:
		
	Case I: for $\xi=\frac{y}{t}\in(2,+\infty)$,
	\begin{align}
	&\text{Im }\theta(z)\geq |\sin \varphi|f(l)\left(\frac{\xi}{4}-\frac{1}{\cos2\varphi+1}\right) ,\hspace{0.5cm} \text{as }z\in\Omega_3, \Omega_4;\\
	&\text{Im }\theta(z)\leq -|\sin \varphi|f(l)\left(\frac{\xi}{4}-\frac{1}{\cos2\varphi+1}\right) ,\hspace{0.5cm} \text{as }z\in\Omega_1, \Omega_2.
	\end{align}
	
	Case IV: for $\xi=\frac{y}{t}\in(-\infty,-0.25)$,
	\begin{align}
		&\text{Im }\theta(z)\geq |\sin \varphi|f(l)\left(-\frac{\xi}{4}-\frac{1}{8(\cos2\varphi+1)} \right) ,\hspace{0.5cm} \text{as }z\in\Omega_1, \Omega_2;\\
		&\text{Im }\theta(z)\leq -|\sin \varphi|f(l)\left(-\frac{\xi}{4}-\frac{1}{8(\cos2\varphi+1)} \right) ,\hspace{0.5cm} \text{as }z\in\Omega_3, \Omega_4.
	\end{align}
\end{lemma}
\begin{proof}
	We  take $z\in\Omega_1$ as an example, and the other regions are similarly.
	From (\ref{Reitheta}),  for $z=le^{i\phi}$, rewrite $\text{Im }\theta(z)$ as
	\begin{align}
	\text{Im }\theta(z)=f(l)\sin\phi\left(-\frac{\xi}{4}+2\frac{2\cos2\phi+6-f(l)^2}{\left( f(l)^2+2\cos2\phi-2\right)^2 } \right) .
	\end{align}
	Denote
	\begin{align}
	h(x;a)=\frac{2a+6-x}{\left( x+2a-2\right)^2 },
	\end{align}
	with $x\geq 4$ and $0<a\leq 1$. Then
	\begin{align}
	\frac{\partial h}{\partial x}=\frac{x-(6a+10)}{\left( x+2a-2\right)^3 }.
	\end{align}
	So $h(x;a)$ has minimum value
	\begin{align}
	h(6a+10;a)=-\frac{1}{8(a+1)}.
	\end{align}
	Together with
	\begin{align}
	h(4;a)=\frac{1}{2(a+1)},\hspace{0.5cm}\lim_{x\to +\infty}h(x;a)=0,
	\end{align}
	we  have that $h(x;a)\in\left(-\frac{1}{8(a+1)},\frac{1}{2(a+1)} \right) $. Then the result is obtained.
\end{proof}
\begin{corollary}\label{Imtheta}
		Set $\xi=\frac{y}{t}\in(-\infty,-0.25)\cup(2,+\infty)$. There exist a constant $c(\xi)>0$ relied on $\xi$ that the imaginary part of phase function (\ref{Reitheta}) $\text{Im }\theta(z)$ have following evaluation for $z=le^{i\phi}=u+vi$:

	Case I: for $\xi=\frac{y}{t}\in(2,+\infty)$,
	\begin{align}
		&\text{Im }\theta(z)\geq c(\xi)v ,\hspace{0.5cm} \text{as }z\in\Omega_3, \Omega_4;\\
		&\text{Im }\theta(z)\leq -c(\xi)v ,\hspace{0.5cm} \text{as }z\in\Omega_1, \Omega_2.\label{3}
	\end{align}

	 Case IV: for $\xi=\frac{y}{t}\in(-\infty,-0.25)$
	 \begin{align}
	 &\text{Im }\theta(z)\geq c(\xi)v ,\hspace{0.5cm} \text{as }z\in\Omega_1, \Omega_2 \label{2};\\
	 &\text{Im }\theta(z)\leq  -c(\xi)v ,\hspace{0.5cm} \text{as }z\in\Omega_3, \Omega_4.
	 \end{align}
\end{corollary}
\begin{lemma}\label{theta2}
 There exist a constant $c(\xi)>0$ relied on $\xi=\frac{y}{t}\in(-0.25,2)$ that the imaginary part of phase function (\ref{Reitheta}) $\text{Im }\theta(z)$ have following estimation for $i=1,...,n(\xi)$:
	\begin{align}
	&\text{Im }\theta(z)\geq c(\xi)\text{Im}z\frac{|z|^2-\xi_i^2}{4+|z|^2} ,\hspace{0.5cm} \text{as }z\in\Omega_{i1}, \Omega_{i3};\\
	&\text{Im }\theta(z)\leq -c(\xi)\text{Im}z\frac{|z|^2-\xi_i^2}{4+|z|^2}  ,\hspace{0.5cm} \text{as }z\in\Omega_{i2}, \Omega_{i4}.
	\end{align}	
\end{lemma}
\begin{proof}
	We only give the detail of Case III( $\xi=\frac{y}{t}\in(-0.25,0)$) and take $z\in\Omega_{11}$ as an example, and the other regions are similarly.
	Denote $z=u+\xi_1+vi$ with $u,v\in \mathbb{R}$ and
	\begin{align}
	k=-\frac{1}{4}\left( z-\frac{1}{z}\right) ,\hspace{0.5cm}	k_1=-\frac{1}{4}\left( \xi_1-\frac{1}{\xi_1}\right).
	\end{align}
	Take notice of that $\xi_1>1$, so $k_1<1$. Moreover,
	 \begin{align}
	 \xi=2\frac{1-4k_1^2}{(1+4k_1^2)^2}<0.\label{xi1}
	 \end{align}
	 And denote
	\begin{align}
	&x\triangleq\text{Re}(k-k_1)=-\frac{1}{4}\left[ u+\frac{u^2+2\xi_1u+v^2}{\xi_1^2[(u+\xi_1)^2+v^2]}\right] ,\\
	&y\triangleq\text{Im}(k-k_1)=-\frac{1}{4}v\left( 1+\frac{1}{|z|^2}\right).
	\end{align}
	Then  the imaginary part of phase function (\ref{Reitheta}) $\text{Im }\theta(z)$ can be rewrite as
	\begin{align}
	\text{Im }\theta(z)=y\left[ \xi-2\frac{1-4|k|^2}{|1+4k^2|^2}\right] .
	\end{align}
	Obviously, a simple calculation gives that $\frac{1-4|k|^2}{|1+4k^2|^2}$ is a monotone increasing function of $y$, so
	 \begin{align}
	 \text{Im }\theta(z)\geq y\left[ \xi-2\frac{1-4(x+k_1)^2}{(1+4(x+k_1)^2)^2}\right] .
	 \end{align}
	 Substitute (\ref{xi1}) into $\xi-2\frac{1-4(x+k_1)^2}{(1+4(x+k_1)^2)^2}$ and obtained
	 \begin{align}
	 \xi-2\frac{1-4(x+k_1)^2}{(1+4(x+k_1)^2)^2}=8\left[ (x+k_1)^2-k_1^2\right] \frac{3-16k_1^2(x+k_1)^2+4(k_1^2+(x+k_1)^2)}{\left[1+4(x+k_1)^2 \right]^2(1+4k_1^2)^2 }
	 \end{align}
	 In the product above, the last item has nonzero  upper and lower bound for $x\geq0$, so
	 \begin{align}
	 \text{Im }\theta(z)&\gtrsim v\left( 1+\frac{1}{|z|^2}\right)\frac{(x+k_1)^2-k_1^2}{1+4(x+k_1)^2}\nonumber\\
	 &\gtrsim v\frac{|z|^2-\xi_1^2}{1+4|k|^2}
	 \gtrsim v\frac{|z|^2-\xi_1^2}{4+|z|^2}.
	 \end{align}
\end{proof}

For Case I and Case VI, introduce following functions for brief:
\begin{align}
	&p_1(z,\xi)=p_2(z,\xi)=\left\{\begin{array}{lll}-\dfrac{\bar{r}(z) }{1+|r(z)|^2},\text{ for }\xi<-0.25\\
	-r(z),\text{ for }\xi> 2\end{array}\right.,\\
	&p_3(z,\xi)=p_4(z,\xi)=\left\{\begin{array}{lll}\dfrac{r(z) }{1+|r(z)|^2},\text{ for }\xi<-0.25\\
	\bar{r}(z),\text{ for }\xi> 2\end{array}\right..
\end{align}
As in Case II and Case III, for $j=1,...,n(\xi)$,
\begin{align}
&p_{j1}(z,\xi)=-\dfrac{\bar{r}(z) }{1+|r(z)|^2},\hspace{0.5cm}p_{j3}(z,\xi)=-r(z),\\
&p_{j2}(z,\xi)=\dfrac{r(z) }{1+|r(z)|^2},\hspace{0.6cm}p_{j4}(z,\xi)=\bar{r}(z).
\end{align}
Besides, from $r\in W^{2,2}(\mathbb{R})$, it also has that $p_1'(z)$ and $p_3'(z)$ exist and are in $L^2(\mathbb{R})\cup L^\infty(\mathbb{R})$. And $\parallel p_1'(z)\parallel_p\lesssim \parallel r'(z)\parallel_p$ for $p=2,\infty$.
Then the next step is to construct a matrix function $R^{(2)}$. We need to remove jump on $\mathbb{R}$ and $i\mathbb{R}$, and  have some mild control on $\bar{\partial}R^{(2)}$ sufficient to ensure that the $\bar{\partial}$-contribution to the long-time asymptotics of $q(x, t)$ is negligible. Note that $\theta(z)$ has different property in different cases, so the construction of $R^{(2)}(z)$ depend on $x\xi$.
Then we choose $R^{(2)}(z,\xi)$ as:

Case I: for $\xi=\frac{y}{t}\in(2,+\infty)$,
\begin{equation}
	R^{(2)}(z,\xi)=\left\{\begin{array}{lll}
		\left(\begin{array}{cc}
			1 & 0\\
			R_j(z,\xi)e^{-2it\theta} & 1
		\end{array}\right), & z\in \Omega_j,j=1,2;\\
		\\
		\left(\begin{array}{cc}
			1 & R_j(z,\xi)e^{2it\theta}\\
			0 & 1
		\end{array}\right),  &z\in \Omega_j,j=3,4;\\
		\\
		I,  &elsewhere;\\
	\end{array}\right.\label{R(2)+}
\end{equation}

Case VI: for  $\xi=\frac{y}{t}\in(-\infty,-0.25)$,
\begin{equation}
R^{(2)}(z,\xi)=\left\{\begin{array}{lll}
\left(\begin{array}{cc}
1 & R_j(z,\xi)e^{2it\theta}\\
0 & 1
\end{array}\right), & z\in \Omega_j,j=1,2;\\
\\
\left(\begin{array}{cc}
1 & 0\\
R_j(z,\xi)e^{-2it\theta} & 1
\end{array}\right),  &z\in \Omega_j,j=3,4;\\
\\
I,  &elsewhere;\\
\end{array}\right.\label{R(2)-}
\end{equation}

where  the functions $R_j$, $j=1,2,..,8$, is defined in following Proposition.
\begin{Proposition}\label{proR}
	 $R_j$: $\bar{\Omega}_j\to C$, $j=1,2,..,8$ have boundary values as follow:
	
	 Case I: for $\xi=\frac{y}{t}\in(-\infty,-0.25)$,
	\begin{align}
&R_1(z,\xi)=\Bigg\{\begin{array}{ll}
	p_1(z,\xi)T_+(z)^{-2} & z\in \mathbb{R}^+,\\
	0  &z\in \Sigma_1,\\
\end{array} ,\hspace{0.6cm}
R_2(z,\xi)=\Bigg\{\begin{array}{ll}
	0  &z\in \Sigma_2,\\
	p_2(z,\xi)T_+(z)^{-2} &z\in  \mathbb{R}^-,\\
\end{array} \\
&R_3(z,\xi)=\Bigg\{\begin{array}{ll}
	p_3(z,\xi)T_-(z)^{2} &z\in \mathbb{R}^-, \\
	0 &z\in \Sigma_3,\\
\end{array} ,\hspace{0.6cm}
R_4(z,\xi)=\Bigg\{\begin{array}{ll}
	0  &z\in \Sigma_4,\\
	p_4(z,\xi)T_-(z)^{2} &z\in  \mathbb{R}^+,\\
\end{array} .
	\end{align}	
	Case II: for $\xi=\frac{y}{t}\in(2,+\infty)$,
		\begin{align}
		&R_1(z,\xi)=\Bigg\{\begin{array}{ll}
			p_1(z,\xi)T(z)^{2} & z\in \mathbb{R}^+,\\
			0  &z\in \Sigma_1,\\
		\end{array} ,\hspace{0.6cm}
		R_2(z,\xi)=\Bigg\{\begin{array}{ll}
			0  &z\in \Sigma_2,\\
			p_2(z,\xi)T(z)^{2} &z\in  \mathbb{R}^-,\\
		\end{array} \\
		&R_3(z,\xi)=\Bigg\{\begin{array}{ll}
			p_3(z,\xi)T(z)^{-2} &z\in \mathbb{R}^-, \\
			0 &z\in \Sigma_3,\\
		\end{array} ,\hspace{0.6cm}
		R_4(z,\xi)=\Bigg\{\begin{array}{ll}
			0  &z\in \Sigma_4,\\
			p_4(z,\xi)T(z)^{-2} &z\in  \mathbb{R}^+,\\
		\end{array} .
	\end{align}	
	
	And $R_j$  have following property:
	for $j=1,2,3,4,$
	\begin{align}
	&|\bar{\partial}R_j(z)|\lesssim|p_j'(|z|)|+|z|^{-1/2}, \text{for all $z\in \Omega_j$,}\label{dbarRj}
	\end{align}
	moreover
	\begin{align}
	 	&|\bar{\partial}R_j(z)|\lesssim|p_j'(|z|)|+|z|^{-1}, \text{for all $z\in \Omega_j$.}\label{dbarRj2}
	\end{align}
	And
	\begin{equation}
	\bar{\partial}R_j(z)=0,\hspace{0.5cm}\text{if } z\in elsewhere.
	\end{equation}
\end{Proposition}

\begin{proof}
	For brief, we only proof case I.
	 Taking $R_1(z)$ as an example, its extensions can be constructed by:
	\begin{equation}
		R_1(z)=p_1(|z|)T^{-2}(z)\cos(k_0 \arg z),\hspace{0.5cm}k_0=\frac{ \pi}{2\varphi}.
	\end{equation}
	The other cases are easily inferred.  Denote $z=le^{i\phi}$, then we have $\bar{\partial}=\frac{e^{i\phi}}{2}\left(\partial_r+\frac{i}{r} \partial_\phi\right) $. So
	\begin{align}
	\bar{\partial}R_1(z)=\frac{e^{i\phi}}{2}T^2(z)\left(p_1'(r)\cos(k_0\phi)-\frac{i}{l}p_1(l)k_0\sin(k_0\phi) \right) .
	\end{align}
	There are two way to bound second term. First we use Cauchy-Schwarz inequality and obtain
	\begin{equation}
	|p_1(l)|=  |p_1(l)-p_1(0)|=|\int_{0}^lp_1'(s)ds|\leq \parallel p_1'(s)\parallel_{L^2} l^{1/2}\lesssim l^{1/2}.
	\end{equation}
	
	And note that $T(z)$ is a bounded function in $\bar{\Omega}_1$. Then the boundedness of (\ref{dbarRj})  follows immediately. On the side, $p_1(l)\in L^\infty$ ,  which implies (\ref{dbarRj2}).
\end{proof}

As in Case II($\xi=\frac{y}{t}\in[0,2)$) and Case III($\xi=\frac{y}{t}\in(-0.25,0)$),
\begin{equation}
R^{(2)}(z,\xi)=\left\{\begin{array}{lll}
\left(\begin{array}{cc}
1 & R_{kj}(z,\xi)e^{2it\theta}\\
0 & 1
\end{array}\right), & z\in \Omega_{kj},j=1,3,\ k=1,...,n(\xi);\\
\\
\left(\begin{array}{cc}
1 & 0\\
R_{kj}(z,\xi)e^{-2it\theta} & 1
\end{array}\right),  &z\in \Omega_{kj},j=2,4\ k=1,...,n(\xi);\\
\\
I,  &elsewhere;\\
\end{array}\right.\label{R(2)1}
\end{equation}

where  the functions $R_{kj}$, $j=1,2,3,4$, $k=1,...,n(\xi)$ are defined in following Proposition.
\begin{Proposition}\label{proR1}
	As in Case II($\xi=\frac{y}{t}\in[0,2)$) and Case III($\xi=\frac{y}{t}\in(-0.25,0)$),  the functions $R_{kj}$: $\bar{\Omega}_{kj}\to \mathbb{C}$, $j=1,2,3,4$, $k=1,...,n(\xi)$ have boundary values as follow:
	\begin{align}
	&R_{k1}(z,\xi)=\Bigg\{\begin{array}{ll}
	p_{k1}(z,\xi)T_+(z)^{-2} & z\in I_{k1},\\
	p_{k1}(\xi_k,\xi)T_k(\xi)^{-2}(z-\xi_k)^{-2i\nu(\xi_k)}  &z\in \Sigma_{k1},\\
	\end{array} ,\\
	&R_{k2}(z,\xi)=\Bigg\{\begin{array}{ll}
	p_{k2}(\xi_k,\xi)T_k(\xi)^{2}(z-\xi_k)^{2i\nu(\xi_k)}  &z\in \Sigma_{k2},\\
	p_{k2}(z,\xi)T_-(z)^{2} &z\in  I_{k2},\\
	\end{array} \\
	&R_{k3}(z,\xi)=\Bigg\{\begin{array}{ll}
	p_{k3}(z,\xi)T(z)^{-2} &z\in I_{k3}, \\
	p_{k3}(\xi_k,\xi)T_k(\xi)^{-2}(z-\xi_k)^{-2i\nu(\xi_k)} &z\in \Sigma_{k3},\\
	\end{array} \\
	&R_{k4}(z,\xi)=\Bigg\{\begin{array}{ll}
	p_{k4}(\xi_k,\xi)T_k(\xi)^{2}(z-\xi_k)^{2i\nu(\xi_k)}  &z\in \Sigma_{k4},\\
	p_{k4}(z,\xi)T(z)^{2} &z\in I_{k4},\\
	\end{array} ,
	\end{align}	
	where $I_{kj}$ is specified in (\ref{In1})-(\ref{In2}). And $R_{kj}$  have following property:
	\begin{align}
	&|R_{kj}(z,\xi)|\lesssim \sin^2(k_0\arg(z-\xi_k))+ \left(1+ \text{Re}(z)^2\right) ^{-1/2}, \text{for all $z\in \Omega_{kj}$},\label{R}\\
	&|\bar{\partial}R_{kj}(z,\xi)|\lesssim|p_{kj}'(\text{Re}z)|+|z-\xi_k|^{-1/2}, \text{for all $z\in \Omega_{kj}$.}\label{dbarRj3}
	\end{align}
	And
	\begin{equation}
	\bar{\partial}R_{kj}(z,\xi)=0,\hspace{0.5cm}\text{if } z\in elsewhere.
	\end{equation}
\end{Proposition}
\begin{proof}
	We give the  details for $R_{11}$ only. The other cases are easily inferred. Using the constants $T_k(\xi)$ defined in proposition \ref{proT}, give the extension of $R_{11}(z,\xi)$ on $\Omega_{11}$:
	\begin{align}
	R_{11}(z,\xi)=&p_{11}(\xi_1,\xi)T_1(\xi)^{-2}(z-\xi_1)^{-2i\nu(\xi_1)}\left[1-\cos\big(k_0\arg(z-\xi_1) \big)\right] \\
	&+\cos\big(k_0\arg(z-\xi_1)\big)p_{11}(\text{Re}z,\xi)T(z)^{-2}.
	\end{align}
	Let $z-\xi_1=le^{i\psi}=u+vi$, $l,\psi,u,v\in\mathbb{R}$. And from $r\in H^{1,1}(\mathbb{R})$, which means $p_{11}\in H^{1,1}(R)$ we have $|p_{11}(u)|\lesssim (1+u^2)^{-1/2}$. Together with (\ref{key}) we have (\ref{R}). Since
	\begin{equation*}
	\bar{\partial}=\frac{1}{2}\left( \partial_u+i\partial_v\right) =\frac{e^{i\psi}}{2}\left( \partial_l+il^{-1}\partial_\psi\right),
	\end{equation*}
	we have
	\begin{align}
	\bar{\partial}R_{11}&=\left(p_{11}(u,\xi)T(z)^{-2}-p_{11}(\xi_1,\xi)T_1(\xi)^{-2}(z-\xi_1)^{-2i\nu(\xi_1)} \right)\bar{\partial}\cos (k_0\psi)\\
	& +\frac{1}{2}T(z)^{-2}p_{11}'(u,\xi)\cos (k_0\psi).
	\end{align}
	Substitute (\ref{T-TJ}) into above equation, (\ref{dbarRj3}) comes immediately.
\end{proof}
In addition, from Proposition \ref{sym}, $R^{(2)}$ achieve the symmetry:
\begin{equation}
	R^{(2)}(z)=\sigma_3\overline{R^{(2)}(-\bar{z})}\sigma_3=\overline{R^{(2)}(-1/\bar{z})}=\sigma_3R^{(2)}(-1/z)\sigma_3.
\end{equation}
We now  use $R^{(2)}$ to define the new transformation \begin{equation}
	M^{(2)}(z;y,t)\triangleq M^{(2)}(z)=M^{(1)}(z)R^{(2)}(z)\label{transm2},
\end{equation}
which satisfies the following mixed $\bar{\partial}$-RH problem.

\begin{RHP}\label{RHP4}
Find a matrix valued function  $ M^{(2)}(z)$ with following properties:

$\blacktriangleright$ Analyticity:  $M^{(2)}(z)$ is continuous in $\mathbb{C}$,  sectionally continuous first partial derivatives in
$\mathbb{C}\setminus \left( \Sigma^{(2)}\cup \left\lbrace\zeta_n,\bar{\zeta}_n \right\rbrace_{n\in\Lambda} \right) $  and meromorphic out $\bar{\Omega}$;

$\blacktriangleright$ Symmetry: $M^{(2)}(z)=\sigma_3\overline{M^{(2)}(-\bar{z})}\sigma_3$=$F^{-2}\overline{M^{(2)}(-\bar{z}^{-1})}$=$F^2\sigma_3M^{(2)}(-z^{-1})\sigma_3$;

$\blacktriangleright$ Jump condition: $M^{(2)}$ has continuous boundary values $M^{(2)}_\pm$ on $\Sigma^{(2)}$ and
\begin{equation}
	M^{(2)}_+(z)=M^{(2)}_-(z)V^{(2)}(z),\hspace{0.5cm}z \in \Sigma^{(2)},
\end{equation}
where for $\xi=\frac{y}{t}\in(2,+\infty)$ or $\xi=\frac{y}{t}\in(-\infty,-0.25)$
\begin{equation}
	V^{(2)}(z)=\left\{ \begin{array}{ll}
		\left(\begin{array}{cc}
			1 & 0\\
			-C_n(z-\zeta_n)^{-1}T^2(z)e^{-2it\theta_n} & 1
		\end{array}\right),   &\text{as } 	z\in\partial\mathbb{D}_n,n\in\nabla;\\[12pt]
		\left(\begin{array}{cc}
			1 & -C_n^{-1}(z-\zeta_n)T^{-2}(z)e^{2it\theta_n}\\
			0 & 1
		\end{array}\right),   &\text{as } z\in\partial\mathbb{D}_n,n\in\Delta;\\
		\left(\begin{array}{cc}
			1 & \bar{C}_n(z-\bar{\zeta}_n)^{-1}T^{-2}(z)e^{2it\bar{\theta}_n}\\
			0 & 1
		\end{array}\right),   &\text{as } 	z\in\partial\overline{\mathbb{D}}_n,n\in\nabla;\\
		\left(\begin{array}{cc}
			1 & 0	\\
			\bar{C}_n^{-1}(z-\bar{\zeta}_n)e^{-2it\bar{\theta}_n}T^2(z) & 1
		\end{array}\right),   &\text{as } 	z\in\partial\overline{\mathbb{D}}_n,n\in\Delta;\\
	\end{array}\right.,\label{jumpv2}
\end{equation}
and for $\xi=\frac{y}{t}\in(-0.25,2)$
\begin{equation}
V^{(2)}(z)=\left\{
\begin{array}{ll}
R^{(2)}(z)|_{\Sigma_{k1}\cup\Sigma_{k4}}&\text{as } 	z\in\Sigma_{k1}\cup\Sigma_{k4};\\[12pt]
R^{(2)}(z)^{-1}|_{\Sigma_{k2}\cup\Sigma_{k3}}&\text{as } 	z\in\Sigma_{k2}\cup\Sigma_{k3};\\[12pt]
R^{(2)}(z)^{-1}|_{\Sigma_{k\ (3\pm1)/2}}R^{(2)}(z)|_{\Sigma_{(k-1)\ (3\pm1)/2}}&\text{as } 	z\in\Sigma_{k\pm}',\ k \text{ is even} ;\\[12pt]
R^{(2)}(z)^{-1}|_{\Sigma_{k\ (7\pm1)/2}}R^{(2)}(z)|_{\Sigma_{(k-1)\ (7\pm1)/2}}&\text{as } 	z\in\Sigma_{k\pm}',\ k \text{ is odd} ;\\[12pt]
\left(\begin{array}{cc}
1 & 0\\
-C_n(z-\zeta_n)^{-1}T^2(z)e^{-2it\theta_n} & 1
\end{array}\right),   &\text{as } 	z\in\partial\mathbb{D}_n,n\in\nabla;\\[12pt]
\left(\begin{array}{cc}
1 & -C_n^{-1}(z-\zeta_n)T^{-2}(z)e^{2it\theta_n}\\
0 & 1
\end{array}\right),   &\text{as } z\in\partial\mathbb{D}_n,n\in\Delta;\\
\left(\begin{array}{cc}
1 & \bar{C}_n(z-\bar{\zeta}_n)^{-1}T^{-2}(z)e^{2it\bar{\theta}_n}\\
0 & 1
\end{array}\right),   &\text{as } 	z\in\partial\overline{\mathbb{D}}_n,n\in\nabla;\\
\left(\begin{array}{cc}
1 & 0	\\
\bar{C}_n^{-1}(z-\bar{\zeta}_n)e^{-2it\bar{\theta}_n}T^2(z) & 1
\end{array}\right),   &\text{as } 	z\in\partial\overline{\mathbb{D}}_n,n\in\Delta;\\
\end{array}\right.;\label{jumpv21}
\end{equation}

$\blacktriangleright$ Asymptotic behaviors:
	\begin{align}
	M^{(2)}(z) =& I+\mathcal{O}(z^{-1}),\hspace{0.5cm}z \rightarrow \infty,\\
	M^{(2)}(z) =&F^{-1} \left[ I+(z-i)\left(\begin{array}{cc}
		0 & -\frac{1}{2}(u+u_x) \\
		-\frac{1}{2}(u-u_x) & 0
	\end{array}\right)\right]\nonumber\\
	&e^{\frac{1}{2}c_+\sigma_3}T(i)^{\sigma_3}\big(I-I_0\sigma_3(z-i) \big) +\mathcal{O}\left( (z-i)^2\right) +\mathcal{O}\left( (z-i)^2\right);\label{asyM2}
\end{align}

$\blacktriangleright$ $\bar{\partial}$-Derivative: For $z\in\mathbb{C}$
we have
\begin{align}
	\bar{\partial}M^{(2)}=M^{(2)}\bar{\partial}R^{(2)},
\end{align}
where
Case I: for $\xi=\frac{y}{t}\in(-\infty,-0.25)$
\begin{equation}
	\bar{\partial}R^{(2)}(z,\xi)=\left\{\begin{array}{lll}
		\left(\begin{array}{cc}
			0 & \bar{\partial}R_j(z,\xi)e^{2it\theta}\\
			0 & 0
		\end{array}\right), & z\in \Omega_j,j=1,2;\\
		\\
		\left(\begin{array}{cc}
			0 & 0\\
			\bar{\partial}R_j(z,\xi)e^{-2it\theta} & 0
		\end{array}\right),  &z\in \Omega_j,j=3,4;\\
		\\
		0,  &elsewhere;\\
	\end{array}\right.\label{DBARR1}
\end{equation}

Case II: for $\xi=\frac{y}{t}\in(2,+\infty)$
\begin{equation}
	\bar{\partial}R^{(2)}(z,\xi)=\left\{\begin{array}{lll}
		\left(\begin{array}{cc}
			0 & 0\\
			\bar{\partial}R_j(z,\xi)e^{-2it\theta} & 0
		\end{array}\right), & z\in \Omega_j,j=1,2;\\
		\\
		\left(\begin{array}{cc}
			0 & \bar{\partial}R_j(z,\xi)e^{2it\theta}\\
			0 & 0
		\end{array}\right),  &z\in \Omega_j,j=3,4;\\
		\\
		0,  &elsewhere;\\
	\end{array}\right.\label{DBARR2}
\end{equation}
Case II($\xi=\frac{y}{t}\in[0,2)$) and Case III($\xi=\frac{y}{t}\in(-0.25,0)$)
\begin{equation}
\bar{\partial}R^{(2)}(z,\xi)=\left\{\begin{array}{lll}
\left(\begin{array}{cc}
0 & 0\\
\bar{\partial}R_{kj}(z,\xi)e^{-2it\theta} & 0
\end{array}\right), & z\in \Omega_{kj},j=1,3,\ k=1,...,n(\xi);\\
\\
\left(\begin{array}{cc}
0 & \bar{\partial}R_{kj}(z,\xi)e^{2it\theta}\\
0 & 0
\end{array}\right),  &z\in \Omega_{kj},j=2,4\ k=1,...,n(\xi);\\
\\
0,  &elsewhere;\\
\end{array}\right.\label{DbarR(2)1}
\end{equation}

$\blacktriangleright$ Residue conditions: $M^{(2)}$ has simple poles at each point $\zeta_n$ and $\bar{\zeta}_n$ for $n\in\Lambda$ with:
\begin{align}
	&\res_{z=\zeta_n}M^{(2)}(z)=\lim_{z\to \zeta_n}M^{(2)}(z)\left(\begin{array}{cc}
		0 & 0\\
		C_ne^{-2it\theta_n}T^2(\zeta_n) & 0
	\end{array}\right),\\
	&\res_{z=\bar{\zeta}_n}M^{(2)}(z)=\lim_{z\to \bar{\zeta}_n}M^{(2)}(z)\left(\begin{array}{cc}
		0 & -\bar{C}_nT^{-2}(\bar{\zeta}_n)e^{2it\bar{\theta}_n}\\
		0 & 0
	\end{array}\right).
\end{align}
	
\end{RHP}

\section{ Decomposition of the mixed $\bar{\partial}$-RH problem }\label{sec5}
\quad To solve RHP2,  we decompose it into a model   RH  problem  for $M^{R}(z;y,t)\triangleq M^{R}(z)$  with $\bar\partial R^{(2)}\equiv0$   and a pure $\bar{\partial}$-Problem with nonzero $\bar{\partial}$-derivatives.
First  we establish  a   RH problem  for the  $M^{R}(z)$   as follows.

\begin{RHP}\label{RHP5}
Find a matrix-valued function  $  M^{R}(z)$ with following properties:

$\blacktriangleright$ Analyticity: $M^{R}(z)$ is  meromorphic  in $\mathbb{C}\setminus \Sigma^{(2)}$;

$\blacktriangleright$ Jump condition: $M^{R}$ has continuous boundary values $M^{R}_\pm$ on $\Sigma^{(2)}$ and
\begin{equation}
	M^{R}_+(z)=M^{R}_-(z)V^{(2)}(z),\hspace{0.5cm}z \in \Sigma^{(2)};\label{jump5}
\end{equation}

$\blacktriangleright$ Symmetry: $M^{R}(z)=\sigma_3\overline{M^{R}(-\bar{z})}\sigma_3$=$F^{-2}\overline{M^{R}(-\bar{z}^{-1})}$=$F^2\sigma_3M^{R}(-z^{-1})\sigma_3$;

$\blacktriangleright$ $\bar{\partial}$-Derivative:  $\bar{\partial}R^{(2)}=0$, for $ z\in \mathbb{C}$;

$\blacktriangleright$ Asymptotic behaviors:
	\begin{align}
	M^{R}(z) =& I+\mathcal{O}(z^{-1}),\hspace{0.5cm}z \rightarrow \infty,\\
	M^{R}(z) =&F^{-1} \left[ I+(z-i)\left(\begin{array}{cc}
		0 & -\frac{1}{2}(u+u_x) \\
		-\frac{1}{2}(u-u_x) & 0
	\end{array}\right)\right] \nonumber\\
	&e^{\frac{1}{2}c_+\sigma_3}T(i)^{\sigma_3}\big(I-I_0\sigma_3(z-i) \big) +\mathcal{O}\left( (z-i)^2\right);\label{asyMr}
\end{align}

$\blacktriangleright$ Residue conditions: $M^{R}$ has simple poles at each point $\zeta_n$ and $\bar{\zeta}_n$ for $n\in\Lambda$ with:
\begin{align}
	&\res_{z=\zeta_n}M^{R}(z)=\lim_{z\to \zeta_n}M^{R}(z)\left(\begin{array}{cc}
		0 & 0\\
		C_ne^{-2it\theta_n}T^2(\zeta_n) & 0
	\end{array}\right),\\
	&\res_{z=\bar{\zeta}_n}M^{R}(z)=\lim_{z\to \bar{\zeta}_n}M^{R}(z)\left(\begin{array}{cc}
		0 & -\bar{C}_nT^{-2}(\bar{\zeta}_n)e^{2it\bar{\theta}_n}\\
		0 & 0
	\end{array}\right).\label{resMr}
\end{align}	
\end{RHP}

In the case of $\xi=\frac{y}{t}\in(-0.25,2)$, it can be found that compared with $\xi=\frac{y}{t}\in(2,+\infty)\cup(-\infty,-0.25)$, its jump matrix $V^{(2)}$ has additional portion on $\Sigma_{jk}$ and $\Sigma_{j\pm}$. So this case is  more difficult  to deal with. And denote $ U(\xi)$ as the union set of neighborhood of $\xi_j$ for $j=1,...,n(\xi)$
\begin{equation}
U(\xi)=\underset{j=1,...,n(\xi)}{\cup}U_{\xi_j},\ U_{\xi_j}= \left\lbrace z:|z-\xi_j|\leq \min\left\lbrace \varrho, \frac{1}{3}\min_{ j\neq i\in \mathcal{N}}|\zeta_i-\zeta_j|\right\rbrace \right\rbrace .
\end{equation}
Then this additional part of  jump matrix $V^{(2)}$ has following  estimation.
\begin{Proposition}\label{prov2}
	As $t\to\infty$, for $1\leq p\leq+\infty$, there exist a positive constant $K_p$ relied on $p$ satisfies that the jump matrix $V^{(2)}$ defined in (\ref{jumpv21}) admits
	\begin{align}
	\parallel V^{(2)}-I\parallel_{L^p(\Sigma_{kj}\setminus U(\xi_k))}= \mathcal{O}( e^{-K_pt}),\\
	\end{align}
	for $k=1,...,n(\xi)$ and $j=1,...,4$.
	And when $1\leq p<+\infty$, there also exist a positive constant $K_p'$ relied on $p$ satisfies that the jump matrix $V^{(2)}$  admits
	\begin{align}
	\parallel V^{(2)}-I\parallel_{L^p(\Sigma_{k\pm}')}= \mathcal{O}( e^{-K_p't}),\\
	\end{align}
	for $k=1,...,n(\xi)$.
\end{Proposition}
\begin{proof}
	We prove the case $\xi=\frac{y}{t}\in(-0.25,0)$, and the another case can be proved in similar way. For $z\in\Sigma_{11}\setminus U_{\xi_1}$, when $1\leq p<+\infty$, by using  definition of $V^{(2)}$ and (\ref{R}),  we have
	\begin{align}
	\parallel V^{(2)}-I\parallel_{L^p(\Sigma_{11}\setminus U_{\xi_1})}&=\parallel p_{11}(\xi_1,\xi)T_1(\xi)^{-2}(z-\xi_1)^{-2i\nu(\xi_1)} e^{2it\theta}\parallel_{L^p(\Sigma_{11}\setminus U_{\xi_1})}\nonumber\\
	&\lesssim \parallel e^{2it\theta}\parallel_{L^p(\Sigma_{11}\setminus U_{\xi_1})}.
	\end{align}
	For $z\in\Sigma_{11}\setminus U_{\xi_1}$, denote $z=\xi_1+le^{i\varphi}$, $l\in(\varrho,+\infty)$. Then lemma \ref{theta2} gives that
	\begin{align}
	\parallel V^{(2)}-I\parallel_{L^p(\Sigma_{11}\setminus U_{\xi_1})}^p&\lesssim \int_{\Sigma_{11}\setminus U_{\xi_1}}\exp\left( -pc(\xi)t\text{Im}z\frac{|z|^2-\xi_1^2}{4+|z|^2}\right)dz \nonumber\\
	&\lesssim \int_{\varrho}^{+\infty}\exp\left( -pc'(\xi)tl\right)dl\lesssim t^{-1}\exp\left( -pc'(\xi)t\varrho\right).
	\end{align}
	The second step is from  $\frac{|z|^2-\xi_1^2}{4+|z|^2}$ has nonzero  boundary on $\Sigma_{11}\setminus U_{\xi_1}$. And when $p=+\infty$ is obviously. For $z\in\Sigma_{k\pm}'$, we only give the details of $\Sigma_{1+}'$. there also has that
	\begin{align}
	\parallel V^{(2)}-I\parallel_{L^p(\Sigma_{1+}')}&=\parallel (R_{24}-R_{14})e^{-2it\theta}\parallel_{L^p(\Sigma_{1+}')}\lesssim\parallel e^{-2it\theta}\parallel_{L^p(\Sigma_{1+}')}\nonumber\\
	&\lesssim t^{-1/p}\exp\left( -c''(\xi)t\right).
	\end{align}
\end{proof}
This proposition means that the jump matrix $V^{(2)}(z)$    uniformly goes to  $I$  on     $\tilde{\Sigma}\setminus U(\xi)$.
So outside the $U(\xi)$ there is only exponentially small error (in $t$) by completely ignoring the jump condition of  $M^{R}(z)$.
And this proposition enlightens  us to construct the solution $M^{R}(z)$ as follow:
\begin{equation}
M^{R}(z)=\left\{\begin{array}{ll}
E(z,\xi)M^{(r)}(z) & z\notin U(\xi)\\
E(z,\xi)M^{(r)}(z)M^{lo}(z)  &z\in U(\xi)\\
\end{array}\right..\label{transm4}
\end{equation}
Note that, when $\xi=\frac{y}{t}\in(2,+\infty)$ or $\xi=\frac{y}{t}\in(-\infty,-0.25)$, $M^{(r)}(z)$ has no jump except the circle around poles not in $\Lambda$, and it has  no  phase point. So $U(\xi)=\emptyset$ in these case, which means $M^{R}(z)=M^{(r)}(z)$. And it is more easy.  And for the case $\xi=\frac{y}{t}\in(-0.25,2)$, from the definition we can easily find that $M^{R}$ is pole free. This construction  decomposes $M^{R}$ to two part: $M^{(o)}$ solves the pure RHP obtained by ignoring the jump conditions of RHP \ref{RHP5}, which is shown in Section \ref{sec6}; $M^{lo}$ uses parabolic cylinder functions to build a matrix to match  jumps  of $M^{(2)}$ in a neighborhood of each critical point $\xi_j$ which is shown in Section \ref{secpc}. And $E(z,\xi)$ is the error function, which will be different in different case of $\xi$ and is a solution of a small-norm Riemann-Hilbert problem shown in Section \ref{sec7}.

We now use $M^{R}(z)$ to construct  a new matrix function
\begin{equation}
M^{(3)}(z;y,t)\triangleq M^{(3)}(z)=M^{(2)}(z)M^{R}(z)^{-1}.\label{transm3}
\end{equation}
which   removes   analytical component  $M^{R}$    to get  a  pure $\bar{\partial}$-problem.

\noindent\textbf{$\bar{\partial}$-problem}. Find a matrix-valued function  $ M^{(3)}(z;y,t)\triangleq M^{(3)}(z)$ with following identities:

$\blacktriangleright$ Analyticity: $M^{(3)}(z)$ is continuous   and has sectionally continuous first partial derivatives in $\mathbb{C}$.

$\blacktriangleright$ Asymptotic behavior:
\begin{align}
&M^{(3)}(z) \sim I+\mathcal{O}(z^{-1}),\hspace{0.5cm}z \rightarrow \infty;\label{asymbehv7}
\end{align}

$\blacktriangleright$ $\bar{\partial}$-Derivative: We have
$$\bar{\partial}M^{(3)}=M^{(3)}W^{(3)},\ \ z\in \mathbb{C},$$
where
\begin{equation}
W^{(3)}=M^{R}(z)\bar{\partial}R^{(2)}(z)M^{R}(z)^{-1}.
\end{equation}

\begin{proof}
	By using  properties  of  the   solutions   $M^{(2)}$ and $M^{R}$  for  RHP \ref{RHP5}  and $\bar{\partial}$-problem ,
 the analyticity is obtained   immediately.
Since $M^{(2)}$ and $M^{R}$ achieve same jump matrix, we have
	\begin{align*}
	M_-^{(3)}(z)^{-1}M_+^{(3)}(z)&=M_-^{(2)}(z)^{-1}M_-^{R}(z)M_+^{R}(z)^{-1}M_+^{(2)}(z)\\
	&=M_-^{(2)}(z)^{-1}V^{(2)}(z)^{-1}M_+^{(2)}(z)=I,
	\end{align*}
	which means $ M^{(3)}$ has no jumps and is everywhere continuous.  We also can show  that $ M^{(3)}$ has no pole. For
 $\lambda \in \left\lbrace \zeta_n,\bar{\zeta}_n \right\rbrace_{n\in\Lambda} $,  let $\mathcal{W}$ denote the  nilpotent matrix which appears in the left side of the
corresponding residue condition of RHP \ref{RHP4}  and  RHP \ref{RHP5},
 we have the Laurent expansions in $z-\lambda$
	\begin{align}
&M^{(2)}(z)=a(\lambda) \left[ \dfrac{\mathcal{W}}{z-\lambda}+I\right] +\mathcal{O}(z-\lambda),\nonumber\\
&	M^{R}(z)=A(\lambda) \left[ \dfrac{\mathcal{W}}{z-\lambda}+I\right] +\mathcal{O}(z-\lambda),\nonumber
\end{align}
	where $a(\lambda)$ and $A(\lambda)$ are the constant  matrix in their respective expansions.
Then
	\begin{align}
	M^{(3)}(z)&=\left\lbrace a(\lambda) \left[ \dfrac{\mathcal{W}}{z-\lambda}+I\right]\right\rbrace \left\lbrace\left[ \dfrac{-\mathcal{W}}{z-\lambda}+I\right]\sigma_2A(\lambda)^T\sigma_2\right\rbrace + \mathcal{O}(z-\lambda)\nonumber\\
	&=\mathcal{O}(1),
	\end{align}
	which  implies that  $ M^{(3)}(z)$ has removable singularities at $\lambda$.
 And the $\bar{\partial}$-derivative of  $ M^{(3)}(z)$ come  from    $ M^{(3)}(z)$  due to   analyticity of $M^{R}(z)$.
\end{proof}
The unique existence  and asymptotic  of  $M^{(3)}(z)$  will shown in   section \ref{sec7}.

\section{ The asymptotic $\mathcal{N}(\Lambda)$-soliton solution } \label{sec6}

\quad In this subsection, we build a reflectionless case of  RHP \ref{RHP2} to  show that  its solution can  approximated  with  $M^{(r)}(z)$. As $W^{(3)}(z)\equiv0$, RHP \ref{RHP4} reduces to RHP \ref{RHP5} for the sectionally meromorphic function $M^{(r)}(z)$ with jump discontinuities on the union of circles. Then, by relate $M^{(r)}(z)$ with original Riemann Hilbert problem \ref{RHP2}, we  show the existence and uniqueness of solution of the above RHP \ref{RHP5}.
\begin{Proposition}
	If $M^{(r)}(z)$ is the solution of the RH problem \ref{RHP5} with scattering data $\mathcal{D}=\left\lbrace  r(z),\left\lbrace \zeta_n,C_n\right\rbrace_{n\in\mathcal{N}}\right\rbrace$, $M^{(r)}(z)$ exists unique.
\end{Proposition}
\begin{proof}
	To transform $M^{(r)}(z)$ to the soliton-solution  of RHP \ref{RHP2}, the jumps and poles need to be restored. We reverses the triangularity effected in (\ref{transm1}) and (\ref{transm2}):
	\begin{equation}
		N(z;\tilde{\mathcal{D}})=\left(\prod_{n\in \Delta}\zeta_n \right)^{-\sigma_3} M^{(r)}(z)T^{-\hat{\sigma}_3}G^{-1}(z)\left( \prod_{n\in \Delta}\dfrac{z-\zeta_n}{\bar{\zeta}_n^{-1}z-1}\right) ^{-\sigma_3},\label{N}
	\end{equation}
with $G(z)$ defined in (\ref{funcG}) and $\tilde{\mathcal{D}}=\left\lbrace  r(z),\left\lbrace \zeta_n,C_n\delta(\zeta_n)\right\rbrace_{n\in\mathcal{N}}\right\rbrace$. First we verify $N(z;\tilde{\mathcal{D}})$ satisfying RHP \ref{RHP2}. This transformation to $N(z;\tilde{\mathcal{D}})$ preserves the normalization conditions at the origin and infinity obviously. And comparing with (\ref{transm1}), this transformation  restore the jump on   $\overline{\mathbb{D}}_n$ and $\mathbb{D}_n$ to residue for $n\notin\Lambda$.  As for $n\in\Lambda$, take $\zeta_n$ as an example. Substitute (\ref{resMr}) into the transformation:
\begin{align}
	\res_{z=\zeta_n}N(z;\tilde{\mathcal{D}})=&\left(\prod_{n\in \Delta}\zeta_n \right)^{-\sigma_3}\res_{z=\zeta_n}M^{(r)}(z)T^{-\hat{\sigma}_3}G(z)^{-1}\left( \prod_{n\in \Delta}\dfrac{z-\zeta_n}{\bar{\zeta}_n^{-1}z-1}\right) ^{-\sigma_3}\nonumber\\
	=&\lim_{z\to \zeta_n}-\left(\prod_{n\in \Delta}\zeta_n \right)^{-\sigma_3}M^{(r)}(z)\left(\begin{array}{cc}
		0 & 0\\
		C_ne^{-2it\theta_n}T^2(\zeta_n) & 0
	\end{array}\right)\left( \prod_{n\in \Delta}\dfrac{z-\zeta_n}{\bar{\zeta}_n^{-1}z-1}\right) ^{-\sigma_3}\nonumber\\
		=&\lim_{z\to \zeta_n}N(z;\tilde{\mathcal{D}})\left(\begin{array}{cc}
			0 & 0\\
			C_n\delta(\zeta_n)e^{-2it\theta_n} & 0
		\end{array}\right).
\end{align}
Its analyticity and symmetry follow from the Proposition of $M^{(r)}(z)$, $T(z)$ and $G(z)$ immediately. Although $N(z;\tilde{\mathcal{D}})$ doesn't preserve the normalization conditions at $z=i$ as (\ref{asyMi}), $z=i$ isn't the pole of $N(z)$. So it make no difference. Then $N(z;\tilde{\mathcal{D}})$ is solution of RHP \ref{RHP2} with absence of reflection, whose  exact solution  exists and can be obtained as described similarly in \cite{SandRNLS} Appendix A. And its uniqueness comes from  Liouville's theorem. Then the uniqueness and existences of $M^{(r)}(z)$  come from (\ref{N}).
\end{proof}

Although $M^{(r)}(z)$ has uniqueness and  existence, not all discrete spectra have contribution as $t\to\infty$. Following Lemma give that   the jump matrices is uniformly near identity and do not meaningfully, contribute to the asymptotic behavior of the solution.
\begin{lemma}\label{lemmav2}
	The jump matrix $ V^{(2)}(z)$ in (\ref{jumpv2}) satisfies
	\begin{align}
	&\parallel V^{(2)}(z)-I\parallel_{L^\infty(\Sigma^{(2)})}=\mathcal{\mathcal{O}}(e^{- 2\rho_0t} ),\label{7.1}\hspace{0.5cm}\text{ with $\rho_0$ specified in (\ref{rho0}).}
\end{align}
\end{lemma}
\begin{proof}
	Take $z\in\partial\mathbb{D}_n,$ $n\in\nabla$ as an example.
	\begin{align}
		\parallel V^{(2)}(z)-I\parallel_{L^\infty(\partial\mathbb{D}_n)}&=|C_n(z-\zeta_n)^{-1}T^2(z)e^{-2it\theta_n}|\nonumber\\
		&\lesssim \varrho^{-1}e^{-\text{Re}(2it\theta_n)}\lesssim e^{2t\text{Im}(\theta_n)}\nonumber\\
		&\leq e^{-2\rho_0t}.
	\end{align}
The last step follows from that for $n\in\nabla$, $ \text{Im}\theta_n<0$.
\end{proof}
\begin{corollary}\label{v2p}
	For $1\leq p\leq +\infty$, the jump matrix $V^{(2)}(z)$ satisfies
	\begin{equation}
		\parallel V^{(2)}(z)-I\parallel_{L^p(\Sigma^{(2)})}\leq K_pe^{- 2\rho_0t} ,
	\end{equation}
for some constant $K_p\geq 0$ depending on $p$.
\end{corollary}
This  estimation of $V^{(2)}(z)$ inspires us to consider to completely ignore the jump condition on $M^{(r)}(z)$, because there is only exponentially small error (in t). We decompose $M^{(r)}(z)$ as
\begin{equation}
	M^{(r)}(z)=\tilde{E}(z)M^{(r)}_\Lambda(z).\label{transMr}
\end{equation}

$\tilde{E}(z)$ is a error function, which is a solution of a small-norm RH problem and we will discuss it in next subsection \ref{sec61}. And $M^{(r)}_\Lambda(z)$ solves RHP \ref{RHP5} with $V^{(2)}(z)\equiv0$.

Then the RHP \ref{RHP5}   reduces to the following RH problem.

\begin{RHP}\label{RHP6}
Find a matrix-valued function  $ M^{(r)}_\Lambda(z)$ with following properties:

$\blacktriangleright$ Analyticity: $M^{(r)}_\Lambda(z)$ is analytical  in $\mathbb{C}\setminus \left\lbrace\zeta_n,\bar{\zeta}_n \right\rbrace_{n\in\Lambda} $;

$\blacktriangleright$ Symmetry: $M^{(r)}_\Lambda(z)=\sigma_3\overline{M^{(r)}_\Lambda(-\bar{z})}\sigma_3$=$F^{-2}\overline{M^{(r)}_\Lambda(-\bar{z}^{-1})}$=$F^2\sigma_3M^{(r)}_\Lambda(-z^{-1})\sigma_3$;

$\blacktriangleright$ Asymptotic behaviors:
	\begin{align}
	M^{(r)}_\Lambda(z) =& I+\mathcal{O}(z^{-1}),\hspace{0.5cm}z \rightarrow \infty;\label{asyMrL}
\end{align}

$\blacktriangleright$ Residue conditions: $M^{(r)}_\Lambda$ has simple poles at each point $\zeta_n$ and $\bar{\zeta}_n$ for $n\in\Lambda$ with:
\begin{align}
	&\res_{z=\zeta_n}M^{(r)}_\Lambda(z)=\lim_{z\to \zeta_n}M^{(r)}_\Lambda(z)\left(\begin{array}{cc}
		0 & 0\\
		C_ne^{-2it\theta_n}T^2(\zeta_n) & 0
	\end{array}\right),\\
	&\res_{z=\bar{\zeta}_n}M^{(r)}_\Lambda(z)=\lim_{z\to \bar{\zeta}_n}M^{(r)}_\Lambda(z)\left(\begin{array}{cc}
		0 & -\bar{C}_nT^{-2}(\bar{\zeta}_n)e^{2it\bar{\theta}_n}\\
		0 & 0
	\end{array}\right).\label{resMrsol}
\end{align}	
\end{RHP}
For  convenience, denote the asymptotic expansion of $M^{(r)}_\Lambda(z)$ as $z\to i$:
\begin{align}
M^{(r)}_\Lambda(z)=M^{(r)}_\Lambda(i)+M^{(r)}_{\Lambda,1}(z-i)+\mathcal{O}((z-i)^{-2}).\label{asymr}
\end{align}

\begin{Proposition}\label{unim}	 The RHP \ref{RHP6}  exists an  unique solution.  Moreover, $M^{(r)}_\Lambda(z)$ is equivalent  to a reflectionless solution of the original RHP \ref{RHP2} with modified scattering data $\tilde{\mathcal{D}}_\Lambda=\left\lbrace  0,\left\lbrace \zeta_n,C_nT^2(\zeta_n)\right\rbrace_{n\in\Lambda}\right\rbrace$ as follows:\\
	\textbf{Case I}: if $\Lambda=\varnothing$, then
	\begin{equation}
		M^{(r)}_\Lambda(z)=I;\label{msol1}
	\end{equation}
	\textbf{Case I}: if $\Lambda\neq\varnothing$ with $\Lambda=\left\lbrace \zeta_{j_k}\right\rbrace_{k=1}^{\mathcal{N}} $, then
	\begin{align}
		M^{(r)}_\Lambda(z)&=I+
\sum_{k=1}^{\mathcal{N}}\left(\begin{array}{cc}
	\frac{\beta_k}{z-\zeta_{j_k}} & \frac{-\overline{\varsigma_k}}{z-\bar{\zeta}_{j_k}}\\
	\frac{\varsigma_k}{z-\zeta_{j_k}} & \frac{\overline{\beta_k}}{z-\bar{\zeta}_{j_k}}
\end{array}\right) ,\label{msol2}
	\end{align}
where  $\beta_s=\beta_s(x,t)$ and $\varsigma_s=\varsigma_s(x,t)$   with linearly dependant equations:
\begin{align}
	c_{j_k}^{-1}T(z_{j_k})^{-2}e^{-2i\theta(z_{j_k})t}\beta_k&=\sum_{h=1}^{\mathcal{N}}\frac{-\overline{\varsigma_h}}{\zeta_{j_k}-\bar{\zeta}_{j_h}} , \\	
	c_{j_k}^{-1}T(z_{j_k})^{-2}e^{-2i\theta(z_{j_k})t}\varsigma_k&=1+\sum_{h=1}^{\mathcal{N}}\frac{\overline{\beta_h}}{\zeta_{j_k}-\bar{\zeta}_{j_h}} ,
\end{align}
 for $k=1,...,\mathcal{N}$ respectively.
\end{Proposition}

\begin{proof}
	The uniqueness of solution follows from the Liouville's theorem. Case I can be simple obtain. As for Case II, the symmetries  of $M^{(r)}_\Lambda(z)$  means that
it  admits a partial fraction expansion of following form as above. And  in order to obtain $\beta_k$, $\varsigma_k$, $\alpha_s$ and $\kappa_s$, we substitute  (\ref{msol2}) into (\ref{resMrsol}) and obtain four linearly dependant equations set above.
\end{proof}
\begin{corollary}\label{sol}
When $r(s)\equiv0$, the scattering matrices $S(z)\equiv I$. Denote $u^r(x,t;\tilde{\mathcal{D}})$ is the $\mathcal{N}(\Lambda)$-soliton with   scattering data $\tilde{\mathcal{D}}_\Lambda=\left\lbrace  0,\left\lbrace \zeta_n,C_nT^2(\zeta_n)\right\rbrace_{n\in\Lambda}\right\rbrace$. By the reconstruction formula (\ref{recons u}) and (\ref{recons x}), the solution $u^r(x,t;\tilde{\mathcal{D}})$  of (\ref{mch}) with  scattering data $\tilde{\mathcal{D}}_\Lambda=\left\lbrace  0,\left\lbrace \zeta_n,C_nT^2(\zeta_n)\right\rbrace_{n\in\Lambda}\right\rbrace$   is given by:
\begin{align}
	u^r(x,t;\tilde{\mathcal{D}}_\Lambda)&=u^r(y(x,t),t;\tilde{\mathcal{D}}_\Lambda)\nonumber\\
	&=\lim_{z\to i}\frac{1}{z-i}\left(1- \dfrac{([M^{(r)}_\Lambda]_{11}(z)+[M^{(r)}_\Lambda]_{21}(z))([M^{(r)}_\Lambda]_{12}(z)+[M^{(r)}_\Lambda]_{22}(z)) }{([M^{(r)}_\Lambda]_{11}(i)+[M^{(r)}_\Lambda]_{21}(i))([M^{(r)}_\Lambda]_{12}(i)+[M^{(r)}_\Lambda]_{22}(i))}\right) ,\label{recons ur}
\end{align}
where
\begin{equation}
	x(y,t;\tilde{\mathcal{D}}_\Lambda)=y+c_+^r(x,t;\tilde{\mathcal{D}}_\Lambda)=y-\ln\left( \frac{[M^{(r)}_\Lambda]_{12}(i)+[M^{(r)}_\Lambda]_{22}(i)}{[M^{(r)}_\Lambda]_{11}(i)+[M^{(r)}_\Lambda]_{21}(i)}\right) .
\end{equation}
Then in case I,
\begin{equation}
	u^r(x,t;\tilde{\mathcal{D}}_\Lambda)=c_+^r(x,t;\tilde{\mathcal{D}}_\Lambda)=0.\label{u1}
\end{equation}
As for case II,
\begin{align}
	u^r(x,t;\tilde{\mathcal{D}}_\Lambda)&=\lim_{z\to i}\frac{1}{z-i}\left(1- \dfrac{([M^{(r)}_\Lambda]_{11}(z)+[M^{(r)}_\Lambda]_{21}(z))([M^{(r)}_\Lambda]_{12}(z)+[M^{(r)}_\Lambda]_{22}(z)) }{([M^{(r)}_\Lambda]_{11}(i)+[M^{(r)}_\Lambda]_{21}(i))([M^{(r)}_\Lambda]_{12}(i)+[M^{(r)}_\Lambda]_{22}(i))}\right) \nonumber\\
	&=\left[\sum_{k=1}^{\mathcal{N}}\left( \frac{-\overline{\varsigma_k}}{(i-\bar{\zeta}_{j_k})^2}+\frac{\overline{\beta_k}}{(i-\bar{\zeta}_{j_k})^2}\right) \right]/\left[ 1+\sum_{k=1}^{\mathcal{N}}\left( \frac{-\overline{\varsigma_k}}{i-\bar{\zeta}_{j_k}}+\frac{\overline{\beta_k}}{i-\bar{\zeta}_{j_k}}\right)  \right] \nonumber\\
	&+ \left[\sum_{k=1}^{\mathcal{N}}\frac{\beta_k}{(i-\zeta_{j_k})^2}+\frac{\varsigma_k}{(i-\zeta_{j_k})^2} \right]/\left[ 1+\sum_{k=1}^{\mathcal{N}}\left( \frac{\beta_k}{i-\zeta_{j_k}}+\frac{\varsigma_k}{i-\zeta_{j_k}}\right)  \right], \label{u2}
\end{align}
and
\begin{align}	
	x(y,t;\tilde{\mathcal{D}}_\Lambda)&=y+c_+^r(x,t;\tilde{\mathcal{D}}_\Lambda)=y-\ln\left( \frac{[M^{(r)}_\Lambda]_{12}(i)+[M^{(r)}_\Lambda]_{22}(i)}{[M^{(r)}_\Lambda]_{11}(i)+[M^{(r)}_\Lambda]_{21}(i)}\right)\nonumber\\
	&= y-\ln\left( \frac{ 1+\sum_{k=1}^{\mathcal{N}}\left( \frac{-\overline{\varsigma_k}}{i-\bar{\zeta}_{j_k}}+\frac{\overline{\beta_k}}{i-\bar{\zeta}_{j_k}}\right) }{1+\sum_{k=1}^{\mathcal{N}}\left( \frac{\beta_k}{i-\zeta_{j_k}}+\frac{\varsigma_k}{i-\zeta_{j_k}}\right) }\right).
\end{align}
\end{corollary}

\subsection{The error function $\tilde{E}(z)$ between $M^{(r)}$ and $M^{(r)}_\Lambda$}\label{sec61}
In this section,  we consider the error matrix-function $\tilde{E}(z)$ and  show that  the error function $\tilde{E}(z)$ solves a small norm RH problem which  can be expanded asymptotically for large times.
From the definition (\ref{transMr}), we can obtain a RH problem  for the matrix function  $\tilde{E}(z)$.

\begin{RHP}\label{RHP7}
	Find a matrix-valued function $\tilde{E}(z)$  with following identities:
	
	$\blacktriangleright$ Analyticity: $\tilde{E}(z)$ is analytical  in $\mathbb{C}\setminus  \Sigma^{(2)} $;

	$\blacktriangleright$ Asymptotic behaviors:
	\begin{align}
	&\tilde{E}(z) \sim I+\mathcal{O}(z^{-1}),\hspace{0.5cm}|z| \rightarrow \infty;
	\end{align}

	$\blacktriangleright$ Jump condition: $\tilde{E}$ has continuous boundary values $\tilde{E}_\pm$ on $\Sigma^{(2)}$ satisfying
	$$\tilde{E}_+(z)=\tilde{E}_-(z)V^{\tilde{E}}(z),$$
	where the jump matrix $V^{\tilde{E}}(z)$ is given by
	\begin{equation}
	V^{\tilde{E}}(z)=M^{(r)}_\Lambda(z)V^{(2)}(z)M^{(r)}_\Lambda(z)^{-1}. \label{tVE}
	\end{equation}
\end{RHP}

{Proposition \ref{unim}} implies that $M^{(r)}_\Lambda(z)$ is bound on $\Sigma^{(2)}$. By using Lemma \ref{lemmav2} and Corollary \ref{v2p}, we have the following evaluation
\begin{equation}
\parallel V^{\tilde{E}}(z)-I \parallel_p\lesssim \parallel V^{(2)}-I \parallel_p=\mathcal{O}(e^{- 2\rho_0t} ) ,\hspace{0.3cm}\text{for $1\leq p \leq +\infty$.} \label{tVE-I}
\end{equation}
This uniformly vanishing bound $\parallel V^{\tilde{E}}-I \parallel$ establishes RHP \ref{RHP7} as a small-norm RH problem.
Therefore,    the   existence and uniqueness  of  the RHP \ref{RHP7} is  shown  by using  a  small-norm RH problem \cite{RN9,RN10} with
\begin{equation}
\tilde{E}(z)=I+\frac{1}{2\pi i}\int_{\Sigma^{(2)}}\dfrac{\left( I+\eta(s)\right) (V^{\tilde{E}}-I)}{s-z}ds,\label{tEz}
\end{equation}
where the $\eta\in L^2(\Sigma^{(2)})$ is the unique solution of following equation:
\begin{equation}
(1-C_{\tilde{E}})\eta=C_{\tilde{E}}\left(I \right).
\end{equation}
Here $C_{\tilde{E}}$:$L^2(\Sigma^{(2)})\to L^2(\Sigma^{(2)})$ is a integral operator defined by
\begin{equation}
C_{\tilde{E}}(f)(z)=C_-\left( f(V^{\tilde{E}}-I)\right) ,
\end{equation}
with  the Cauchy projection operator $C_-$    on $\Sigma^{(2)}$ :
\begin{equation}
C_-(f)(s)=\lim_{z\to \Sigma^{(2)}_-}\frac{1}{2\pi i}\int_{\Sigma^{(2)}}\dfrac{f(s)}{s-z}ds.
\end{equation}
Then by (\ref{tVE}) we have
\begin{equation}
\parallel C_{\tilde{E}}\parallel\leq\parallel C_-\parallel \parallel V^{\tilde{E}}-I\parallel_\infty \lesssim \mathcal{O}(e^{- 2\rho_0t} ),
\end{equation}
which means $\parallel C_{\tilde{E}}\parallel<1$ for sufficiently large t,   therefore  $1-C_{\tilde{E}}$ is invertible,  and   $\eta$  exists and is unique.
Moreover,
\begin{equation}
\parallel \eta\parallel_{L^2(\Sigma^{(2)})}\lesssim\dfrac{\parallel C_{\tilde{E}}\parallel}{1-\parallel C_{\tilde{E}}\parallel}\lesssim\mathcal{O}(e^{- 2\rho_0t} ).\label{normeta}
\end{equation}
Then we have the existence and boundedness of $\tilde{E}(z)$. In order to reconstruct the solution $q(x,t)$ of (\ref{mch}), we need the asymptotic behavior of $\tilde{E}(z)$ as $z\to \infty$ and the long time asymptotic behavior of $\tilde{E}(i)$.
\begin{Proposition}\label{tasyE}
	For $\tilde{E}(z)$ defined in (\ref{tEz}), it stratifies
	\begin{equation}
	|\tilde{E}(z)-I|\lesssim\mathcal{O}(e^{- 2\rho_0t}) .
	\end{equation}
	When $z=i$,
	\begin{equation}
	\tilde{E}(i)=I+\frac{1}{2\pi i}\int_{\Sigma^{(2)}}\dfrac{\left( I+\eta(s)\right) (V^{\tilde{E}}-I)}{s-i}ds,\label{tEi}
	\end{equation}
	As $z\to i$, $\tilde{E}(z)$ has expansion  at $z=i$
	\begin{align}
	\tilde{E}(z)=\tilde{E}(i)+\tilde{E}_1(z-i)+\mathcal{O}((z-i)^{2}),\label{texpE}
	\end{align}
	where
	\begin{equation}
	\tilde{E}_1=\frac{1}{2\pi i}\int_{\Sigma^{(2)}}\frac{\left( I+\eta(s)\right) (V^{\tilde{E}}-I)}{(s-i)^2}ds.
	\end{equation}
	Moreover, $\tilde{E}(i)$ and  $\tilde{E}_1$ satisfy following long time asymptotic behavior condition:
	\begin{equation}
	|\tilde{E}(i)-I|\lesssim\mathcal{O}(e^{- 2\rho_0t}) ,\hspace{0.5cm}\tilde{E}_1\lesssim\mathcal{O}(e^{- 2\rho_0t}).\label{tE1t}
	\end{equation}
\end{Proposition}
\begin{proof}
	By combining (\ref{normeta}) and (\ref{tVE-I}), we obtain
	\begin{equation}
	|\tilde{E}(z)-I|\leq|(1-C_{\tilde{E}})(\eta)|+|C_{\tilde{E}}(\eta)|\lesssim\mathcal{O}(e^{- 2\rho_0t}).
	\end{equation}
	And the asymptotic behavior $\tilde{E}(i)$ in (\ref{tE1t}) is obtained by taking $z=i$ in above estimation. As $z\to i$, geometrically expanding $(s-z)^{-1}$
	for $z$ large in (\ref{tEz}) leads to (\ref{texpE}). Finally for $\tilde{E}_1$, noting that $|s-i|^{-2}$ is bounded on $\Sigma^{(2)}$, then
	\begin{align}
	|\tilde{E}_1|\lesssim \parallel V^{\tilde{E}}-I \parallel_1+\parallel \eta \parallel_2\parallel V^{\tilde{E}}-I \parallel_2\lesssim\mathcal{O}(e^{- 2\rho_0t}).
	\end{align}
\end{proof}

\section{ A local solvable  RH model near phase points for $\xi\in(-0.25,2)$}\label{secpc}
\quad When $\xi\in(-0.25,2)$, proposition  \ref{prov2} gives that out of $U(\xi)$, the jumps are exponentially close to the identity. Hence we need to continue our investigation near  the stationary phase points in this section. Denote a new contour $\Sigma^{(0)}= (\underset{j=1,..,n(\xi)}{\underset{k=1,...,4,}{\cup}}\Sigma_{jk} )\cap U(\xi)$ in Figure \ref{sigma0}.
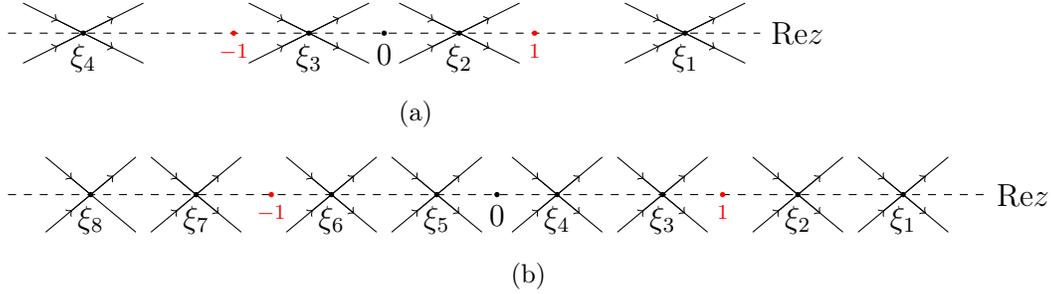
\begin{figure}[h]
	\subfigure[]{
		\begin{tikzpicture}
		\draw(-4,0)--(-4.8,0.4);
		\draw[-<](-4,0)--(-4.4,0.2);
		\draw(-4,0)--(-3.2,0.4);
		\draw[->](-4,0)--(-3.6,-0.2);
		\draw(-4,0)--(-4.8,-0.4);
		\draw[->](-4,0)--(-3.6,0.2);
		\draw(-4,0)--(-3.2,-0.4);
		\draw[-<](-4,0)--(-4.4,-0.2);
		\draw(-1,0)--(-0.2,0.4);
		\draw[->](-1,0)--(-0.6,0.2);
		\draw(-1,0)--(-1.8,0.4);
		\draw[-<](-1,0)--(-1.4,-0.2);
		\draw(-1,0)--(-0.2,-0.4);
		\draw[-<](-1,0)--(-1.4,0.2);
		\draw(-1,0)--(-1.8,-0.4);
		\draw[->](-1,0)--(-0.6,-0.2);
		\draw[dashed](-5,0)--(5,0)node[right]{ Re$z$};
		\draw(1,0)--(0.2,0.4);
		\draw[-<](1,0)--(0.6,0.2);
		\draw(1,0)--(0.2,-0.4);
		\draw[->](1,0)--(1.4,-0.2);
		\draw(1,0)--(1.8,0.4);
		\draw[->](1,0)--(1.4,0.2);
		\draw(1,0)--(1.8,-0.4);
		\draw[-<](1,0)--(0.6,-0.2);
		\draw(4,0)--(4.8,0.4);
		\draw[->](4,0)--(4.4,0.2);
		\draw(4,0)--(3.2,0.4);
		\draw[-<](4,0)--(3.6,-0.2);
		\draw(4,0)--(4.8,-0.4);
		\draw[-<](4,0)--(3.6,0.2);
		\draw(4,0)--(3.2,-0.4);
		\draw[->](4,0)--(4.4,-0.2);
		\coordinate (I) at (0,0);
		\fill (I) circle (1pt) node[below] {$0$};
		\coordinate (A) at (-4,0);
		\fill (A) circle (1pt) node[below] {$\xi_4$};
		\coordinate (b) at (-1,0);
		\fill (b) circle (1pt) node[below] {$\xi_3$};
		\coordinate (e) at (4,0);
		\fill (e) circle (1pt) node[below] {$\xi_1$};
		\coordinate (f) at (1,0);
		\fill (f) circle (1pt) node[below] {$\xi_2$};
		\coordinate (c) at (-2,0);
		\fill[red] (c) circle (1pt) node[below] {\scriptsize$-1$};
		\coordinate (d) at (2,0);
		\fill[red] (d) circle (1pt) node[below] {\scriptsize$1$};
		\end{tikzpicture}
		\label{si1}}
	\subfigure[]{
		\begin{tikzpicture}
		\draw[dashed](-6.5,0)--(6.5,0)node[right]{ Re$z$};
		\coordinate (I) at (0,0);
		\fill (I) circle (1pt) node[below] {$0$};
		\coordinate (c) at (-3,0);
		\fill[red] (c) circle (1pt) node[below] {\scriptsize$-1$};
		\coordinate (D) at (3,0);
		\fill[red] (D) circle (1pt) node[below] {\scriptsize$1$};
		\draw(-0.8,0)--(-0.2,0.5);
		\draw[->](-0.8,0)--(-0.5,0.25);
		\draw(-0.8,0)--(-1.4,0.5);
		\draw[-<](-0.8,0)--(-1.1,-0.25);
		\draw(-0.8,0)--(-0.2,-0.5);
		\draw[-<](-0.8,0)--(-1.1,0.25);
		\draw(-0.8,0)--(-1.4,-0.5);
		\draw[->](-0.8,0)--(-0.5,-0.25);
		\draw(-2.2,0)--(-1.6,0.5);
		\draw[-<](-2.2,0)--(-2.5,0.25);
		\draw(-2.2,0)--(-1.6,-0.5);
		\draw[->](-2.2,0)--(-1.9,-0.25);
		\draw(-2.2,0)--(-2.8,0.5);
		\draw[->](-2.2,0)--(-1.9,0.25);
		\draw(-2.2,0)--(-2.8,-0.5);
		\draw[-<](-2.2,0)--(-2.5,-0.25);
		\draw(-5.4,0)--(-6,0.5);
		\draw[-<](-5.4,0)--(-5.7,0.25);
		\draw(-5.4,0)--(-4.8,0.5);
		\draw(-5.4,0)--(-6,-0.5);
		\draw[->](-5.4,0)--(-5.1,0.25);
		\draw(-5.4,0)--(-4.8,-0.5);
		\draw[-<](-5.4,0)--(-5.7,-0.25);
		\draw(-4,0)--(-3.4,0.5);
		\draw[->](-4,0)--(-3.7,0.25);
		\draw(-4,0)--(-4.6,0.5);
		\draw[-<](-4,0)--(-4.3,-0.25);
		\draw(-4,0)--(-3.4,-0.5);
		\draw[-<](-4,0)--(-4.3,0.25);
		\draw(-4,0)--(-4.6,-0.5);
		\draw[->](-4,0)--(-3.7,-0.25);
		\draw(0.8,0)--(0.2,0.5);
		\draw[-<](0.8,0)--(0.5,0.25);
		\draw(0.8,0)--(1.4,0.5);
		\draw[->](0.8,0)--(1.1,-0.25);
		\draw(0.8,0)--(0.2,-0.5);
		\draw[->](0.8,0)--(1.1,0.25);
		\draw(0.8,0)--(1.4,-0.5);
		\draw[-<](0.8,0)--(0.5,-0.25);
		\draw(2.2,0)--(1.6,0.5);
		\draw[->](2.2,0)--(2.5,0.25);
		\draw(2.2,0)--(1.6,-0.5);
		\draw[-<](2.2,0)--(1.9,-0.25);
		\draw(2.2,0)--(2.8,0.5);
		\draw[-<](2.2,0)--(1.9,0.25);
		\draw(2.2,0)--(2.8,-0.5);
		\draw[->](2.2,0)--(2.5,-0.25);
		\draw(5.4,0)--(6,0.5);
		\draw[->](5.4,0)--(5.7,0.25);
		\draw(5.4,0)--(4.8,0.5);
		\draw[-<](5.4,0)--(5.1,-0.25);
		\draw(5.4,0)--(6,-0.5);
		\draw[-<](5.4,0)--(5.1,0.25);
		\draw(5.4,0)--(4.8,-0.5);
		\draw[->](5.4,0)--(5.7,-0.25);
		\draw(4,0)--(3.4,0.5);
		\draw[-<](4,0)--(3.7,0.25);
		\draw(4,0)--(4.6,0.5);
		\draw[->](4,0)--(4.3,-0.25);
		\draw(4,0)--(3.4,-0.5);
		\draw[->](4,0)--(4.3,0.25);
		\draw(4,0)--(4.6,-0.5);
		\draw[-<](4,0)--(3.7,-0.25);
		\coordinate (A) at (-5.4,0);
		\fill (A) circle (1pt) node[below] {$\xi_8$};
		\coordinate (b) at (-4,0);
		\fill (b) circle (1pt) node[below] {$\xi_7$};
		\coordinate (C) at (-0.8,0);
		\fill (C) circle (1pt) node[below] {$\xi_5$};
		\coordinate (d) at (-2.2,0);
		\fill (d) circle (1pt) node[below] {$\xi_6$};
		\coordinate (E) at (5.4,0);
		\fill (E) circle (1pt) node[below] {$\xi_1$};
		\coordinate (R) at (4,0);
		\fill (R) circle (1pt) node[below] {$\xi_2$};
		\coordinate (T) at (0.8,0);
		\fill (T) circle (1pt) node[below] {$\xi_4$};
		\coordinate (Y) at (2.2,0);
		\fill (Y) circle (1pt) node[below] {$\xi_3$};
		\end{tikzpicture}
		\label{si2}}
	\caption{Figure (a) and (b) shows $\Sigma^{(0)}$, and are corresponding to the  $0\leq\xi<2$ and  $-\frac{1}{4}<\xi<0$ respectively.}
	\label{sigma0}
\end{figure}
Consider following RHP:
\begin{RHP}
	Find a matrix-valued function  $ M^{lo}(z)$ with following properties:
	
	$\blacktriangleright$ Analyticity: $M^{lo}(z)$ is analytical  in $\mathbb{C}\setminus \Sigma^{(0)} $;
	
	$\blacktriangleright$ Symmetry: $M^{lo}(z)=\sigma_3\overline{M^{lo}(-\bar{z})}\sigma_3$=$F^{-2}\overline{M^{lo}(-\bar{z}^{-1})}$=$F^2\sigma_3M^{lo}(-z^{-1})\sigma_3$;
	
	$\blacktriangleright$ Jump condition: $M^{lo}$ has continuous boundary values $M^{lo}_\pm$ on $\Sigma^{(0)}$ and
	\begin{equation}
	M^{lo}_+(z)=M^{lo}_-(z)V^{(2)}(z),\hspace{0.5cm}z \in \Sigma^{(0)};\label{jump6}
	\end{equation}
	
	$\blacktriangleright$ Asymptotic behaviors:
	\begin{align}
	M^{lo}(z) =& I+\mathcal{O}(z^{-1}),\hspace{0.5cm}z \rightarrow \infty;
	\end{align}
\end{RHP}	
This RHP only has jump condition and has no poles. The  matrix $V^{(x)}(z)$ is a  upper/lower matrix with l's on the  diagonal. For $k=1,...,n(\xi)$, we denote
\begin{align}
w_{kj}(z)=\left\{\begin{array}{lll}
\left(\begin{array}{cc}
0 & -R_{kj}(z,\xi)e^{2it\theta}\\
0 & 0
\end{array}\right), &z\in \Sigma_{kj},j=1,3,\\[10pt]
\left(\begin{array}{cc}
0 & 0\\
-R_{kj}(z,\xi)e^{-2it\theta} & 0
\end{array}\right),  &z\in \Sigma_{kj},j=2,4.
\end{array}\right.
\end{align}
Then $V^{(2)}(z)=I-w_{kj}(z)$ for $z\in \Sigma_{kj}$. Moreover, let $\Sigma^{(0)}_k=\cup_{j=1,...,4}\Sigma_{kj}$, $w_k(z)=\sum_{j=1,...,4} w_{kj}(z)$, $w_{kj}^\pm(z)=w_{kj}(z)|_{\mathbb{C}^\pm}$,  $w_k^\pm(z)=w_k(z)|_{\mathbb{C}^\pm}$ and $w^\pm(z)=w(z)|_{\mathbb{C}^\pm}$. Recall the Cauchy projection operator $C_\pm$    on $\Sigma^{(2)}$ :
\begin{equation}
C_{\pm}(f)(s)=\lim_{z\to \Sigma^{(2)}_\pm}\frac{1}{2\pi i}\int_{\Sigma^{(2)}}\dfrac{f(s)}{s-z}ds.
\end{equation}
By using it, define operator
\begin{align}
C_w(f)=C_+(fw^-)+C_-(fw^+),\hspace{0.5cm}C_{w_k}(f)=C_+(fw_k^-)+C_-(fw_k^+).
\end{align}
Then $C_w=\sum_{k=1}^{n(\xi)}C_{w_k}$.
\begin{lemma}
	The matrix functions $w_{kj}$ defined above admits following  estimation:
	\begin{align}
	\parallel w_{kj}\parallel_{L^p(\Sigma_{kj})}=\mathcal{O}(t^{-1/2}),\ 1\leq p<+\infty.
	\end{align}
\end{lemma}
This lemma can be obtained by simple calculation. And it implies that $I-C_w$ and $I-C_{w_k}$ are reversible. So the solution of above RHP exist unique, and it can be written as
\begin{align}
M^{lo}=I+\frac{1}{2\pi i}\int_{\Sigma^{(0)}}\frac{(I-C_w)^{-1}I\ w}{s-z}ds.
\end{align}
Next, we show the  contributions of every crosses $\Sigma^{(0)}_k$ can be separated out.
\begin{corollary}
	As $t\to+\infty$,
	\begin{align}
	\parallel C_{w_k}C_{w_j}\parallel_{B(L^2(\Sigma^{(0)}))}\lesssim t^{-1},\hspace{0.5cm}\parallel C_{w_k}C_{w_j}\parallel_{L^\infty(\Sigma^{(0)})\to L^2(\Sigma^{(0)})}\lesssim t^{-1}.
	\end{align}
\end{corollary}
Direct calculation establishes that
\begin{align}
&\left(I- C_w\right) \left(I+\sum_{k=1}^{n(\xi)}C_{w_k}(I-C_{w_k})^{-1} \right) =I-\sum_{1\leq k\neq j\leq n(\xi)}C_{w_j}C_{w_k}(I-C_{w_k})^{-1},\\
& \left(I+\sum_{k=1}^{n(\xi)}C_{w_k}(I-C_{w_k})^{-1} \right) \left(I- C_w\right)=I-\sum_{1\leq k\neq j\leq n(\xi)}(I-C_{w_k})^{-1}C_{w_k}C_{w_j}.
\end{align}
Then following the step of \cite{RN6}, we derive the proposition:
\begin{Proposition}\label{dividepc}
	As $t\to +\infty$,
	\begin{align}
	\int_{\Sigma^{(0)}}\frac{(I-C_w)^{-1}I\ w}{s-z}ds=\sum_{ k=1 }^{n(\xi)}\int_{\Sigma^{(0)}_k}\frac{(I-C_{w_k})^{-1}I\ w_k}{s-z}ds+\mathcal{O}(t^{-3/2}).
	\end{align}
\end{Proposition}
So as $t\to +\infty$, we can only consider to reduce above RHP to a model RHP whose solution can be given explicitly in terms of parabolic cylinder
functions on every contour $\Sigma^{(0)}_k$ respectively. And we only give the details of $\Sigma^{(0)}_1$, the model of other  critical point can be  constructed similar. We denote $\hat{\Sigma}^{(0)}_1$ as the contour $\{z=\xi_1+le^{\pm\varphi i},\ l\in\mathbb{R}\}$ oriented from $\Sigma^{(0)}_1$, and $\hat{\Sigma}_{1j}$ is the extension of $\Sigma_{1j}$  respectively. And for $z$ near $\xi_1$, rewrite phase function as
\begin{align}
\theta(z)=\theta(\xi_1)+(z-\xi_1)^2\theta''(\xi_1)+\mathcal{O}((z-\xi_1)^3).
\end{align}
When $\xi\in[0,2)$, $\theta''(\xi_1)<0$ and when $\xi\in(-0.25,0)$, $\theta''(\xi_1)>0$.
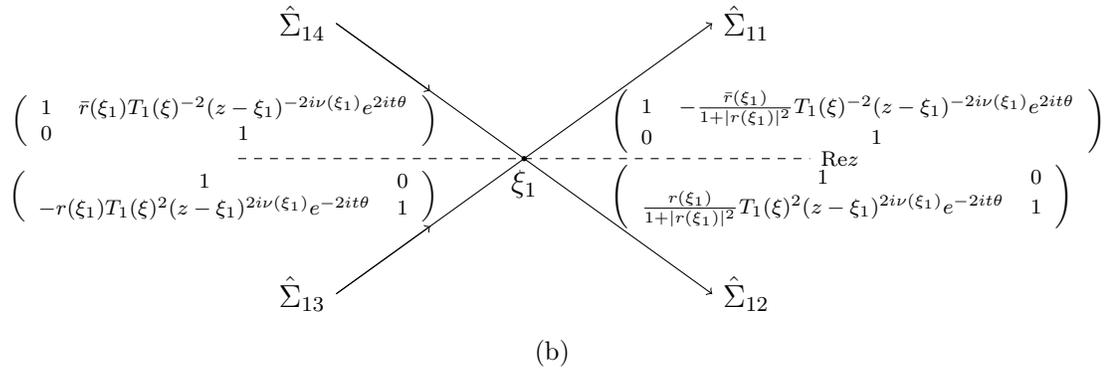
\begin{figure}
	\centering
	\subfigure[]{
		\begin{tikzpicture}[node distance=2cm]
		\draw[](0,2.7)node[above]{$\xi\in[0,2)$};
		\draw[->](0,0)--(2.5,1.4)node[right]{$\hat{\Sigma}_{14}$};
		\draw(0,0)--(-2.5,1.4)node[left]{$\hat{\Sigma}_{11}$};
		\draw(0,0)--(-2.5,-1.4)node[left]{$\hat{\Sigma}_{12}$};
		\draw[->](0,0)--(2.5,-1.4)node[right]{$\hat{\Sigma}_{13}$};
		\draw[dashed](-3.8,0)--(3.8,0)node[right]{\scriptsize Re$z$};
		\draw[->](-2.5,-1.4)--(-1.25,-0.7);
		\draw[->](-2.5,1.4)--(-1.25,0.7);
		\coordinate (A) at (-1.2,0.5);
		\coordinate (B) at (-1.2,-0.5);
		\coordinate (G) at (1.4,0.5);
		\coordinate (H) at (1.4,-0.5);
		\coordinate (I) at (0,0);
		\fill (A) circle (0pt) node[left] {\scriptsize$\left(\begin{array}{cc}
			1 & -\frac{\bar{r}(\xi_1)}{1+|r(\xi_1)|^2}T_1(\xi)^{-2}(z-\xi_1)^{-2i\nu(\xi_1)}e^{2it\theta}\\
			0 & 1
			\end{array}\right)$};
		\fill (B) circle (0pt) node[left] {\scriptsize$\left(\begin{array}{cc}
			1 & 0\\
			\frac{r(\xi_1)}{1+|r(\xi_1)|^2}T_1(\xi)^{2}(z-\xi_1)^{2i\nu(\xi_1)}e^{-2it\theta} & 1
			\end{array}\right)$};
		\fill (G) circle (0pt) node[right] {\scriptsize$\left(\begin{array}{cc}
			1& \bar{r}(\xi_1)T_1(\xi)^{-2}(z-\xi_1)^{-2i\nu(\xi_1)}e^{2it\theta}\\
			0&1
			\end{array}\right)$};
		\fill (H) circle (0pt) node[right] {\scriptsize$\left(\begin{array}{cc}
			1 & 0\\
			-r(\xi_1)T_1(\xi)^{2}(z-\xi_1)^{2i\nu(\xi_1)}e^{-2it\theta} & 1
			\end{array}\right)$};
		\fill (I) circle (1pt) node[below] {$\xi_1$};
		\end{tikzpicture}
	}
	\subfigure[]{
		\begin{tikzpicture}[node distance=2cm]
		\draw[](0,2.7) node[above]{$\xi\in(-0.25,0)$};
		\draw[->](0,0)--(2.5,1.8)node[right]{$\hat{\Sigma}_{11}$};
		\draw(0,0)--(-2.5,1.8)node[left]{$\hat{\Sigma}_{14}$};
		\draw(0,0)--(-2.5,-1.8)node[left]{$\hat{\Sigma}_{13}$};
		\draw[->](0,0)--(2.5,-1.8)node[right]{$\hat{\Sigma}_{12}$};
		\draw[dashed](-3.8,0)--(3.8,0)node[right]{\scriptsize Re$z$};
		\draw[->](-2.5,-1.8)--(-1.25,-0.9);
		\draw[->](-2.5,1.8)--(-1.25,0.9);
		\coordinate (A) at (1,0.5);
		\coordinate (B) at (1,-0.5);
		\coordinate (G) at (-1,0.5);
		\coordinate (H) at (-1,-0.5);
		\coordinate (I) at (0,0);
		\fill (A) circle (0pt) node[right] {\scriptsize$\left(\begin{array}{cc}
		1 & -\frac{\bar{r}(\xi_1)}{1+|r(\xi_1)|^2}T_1(\xi)^{-2}(z-\xi_1)^{-2i\nu(\xi_1)}e^{2it\theta}\\
		0 & 1
		\end{array}\right)$};
	\fill (B) circle (0pt) node[right] {\scriptsize$\left(\begin{array}{cc}
		1 & 0\\
		\frac{r(\xi_1)}{1+|r(\xi_1)|^2}T_1(\xi)^{2}(z-\xi_1)^{2i\nu(\xi_1)}e^{-2it\theta} & 1
		\end{array}\right)$};
	\fill (G) circle (0pt) node[left] {\scriptsize$\left(\begin{array}{cc}
		1& \bar{r}(\xi_1)T_1(\xi)^{-2}(z-\xi_1)^{-2i\nu(\xi_1)}e^{2it\theta}\\
		0&1
		\end{array}\right)$};
	\fill (H) circle (0pt) node[left] {\scriptsize$\left(\begin{array}{cc}
		1 & 0\\
		-r(\xi_1)T_1(\xi)^{2}(z-\xi_1)^{2i\nu(\xi_1)}e^{-2it\theta} & 1
		\end{array}\right)$};
		\fill (I) circle (1pt) node[below] {$\xi_1$};
		\end{tikzpicture}
	}
	\caption{The contour $\hat{\Sigma}^{(0)}_1$ and the jump matrix on it in case $\xi\in[0,2)$ and $\xi\in(-0.25,0)$ respectively.}
	\label{figS0}
\end{figure}
Consider following local RHP:
\begin{RHP}\label{RHPlo1}
	Find a matrix-valued function  $ M^{lo,1}(z)$ with following properties:
	
	$\blacktriangleright$ Analyticity: $M^{lo,1}(z)$ is analytical  in $\mathbb{C}\setminus \hat{\Sigma}_1 $;

	$\blacktriangleright$ Jump condition: $M^{lo,1}$ has continuous boundary values $M^{lo,1}_\pm$ on $\hat{\Sigma}_1$ and
	\begin{equation}
	M^{lo,1}_+(z)=M^{lo,1}_-(z)V^{lo,1}(z),\hspace{0.5cm}z \in \hat{\Sigma}^{(0)}_1,
	\end{equation}
	where
	\begin{align}
	V^{lo,1}(z)=\left\{\begin{array}{ll}
	\left(\begin{array}{cc}
	1 & -\frac{\bar{r}(\xi_1)}{1+|r(\xi_1)|^2}T_1(\xi)^{-2}(z-\xi_1)^{-2i\nu(\xi_1)}e^{2it\theta}\\
	0 & 1
	\end{array}\right),  & z\in \hat{\Sigma}_{11},\\[10pt]
	\left(\begin{array}{cc}
	1 & 0\\
	\frac{r(\xi_1)}{1+|r(\xi_1)|^2}T_1(\xi)^{2}(z-\xi_1)^{2i\nu(\xi_1)}e^{-2it\theta} & 1
	\end{array}\right),   & z\in \hat{\Sigma}_{12},\\[10pt]
	\left(\begin{array}{cc}
	1& \bar{r}(\xi_1)T_1(\xi)^{-2}(z-\xi_1)^{-2i\nu(\xi_1)}e^{2it\theta}\\
	0&1
	\end{array}\right),   & z\in \hat{\Sigma}_{13},\\[10pt]
	\left(\begin{array}{cc}
	1 & 0\\
	-r(\xi_1)T_1(\xi)^{2}(z-\xi_1)^{2i\nu(\xi_1)}e^{-2it\theta} & 1
	\end{array}\right),   & z\in \hat{\Sigma}_{14}.
	\end{array}\right.
	\end{align}
	
	$\blacktriangleright$ Asymptotic behaviors:
	\begin{align}
	M^{lo,1}(z) =& I+\mathcal{O}(z^{-1}),\hspace{0.5cm}z \rightarrow \infty;
	\end{align}
\end{RHP}	

RHP \ref{RHPlo1} does not possess the symmetry condition shared by preceding RHP, because it is a local model and will only be used for bounded values of $z$.
In order to motivate the model, let $\zeta = \zeta(z)$ denote the rescaled local variable
\begin{align}
\zeta(z)=t^{1/2}\sqrt{4\eta(\xi)\theta''(\xi_1)}(z-\xi_1),
\end{align}
where, $\eta(\xi)=-1$, when $\xi\in[0,2)$,  and $\eta(\xi)=1$ when $\xi\in(-0.25,0)$. This change of variable maps $U_{\xi_1}$ to an expanding neighborhood of $\zeta= 0$. Additionally, let
\begin{align}
r_{\xi_1}=r(\xi_1)T_1(\xi)^{2}e^{-2it\theta(\xi_1)}\exp\left\lbrace -i\nu(\xi_1)\log \left( 4t\theta''(\xi_1)\eta(\xi_1)\right) \right\rbrace ,
\end{align}
with $|r_{\xi_1}|=|r(\xi_1)|$.
In the above expression, the complex powers are defined by choosing the branch of
the logarithm with  $-\pi< \arg \zeta < \pi$ in the cases$\xi\in[0,2)$, and the branch of the logarithm with $0 < \arg \zeta < 2\pi$ in the case $\xi\in(-0.25,0)$.

Through this change of variable,  the jump $V^{lo,1}(z)$ approximates to  the jump of a parabolic cylinder model problem as follow:
 \begin{RHP}\label{RHPpc}
 	Find a matrix-valued function  $ M^{pc}(\zeta;\xi)$ with following properties:
 	
 	$\blacktriangleright$ Analyticity: $M^{pc}(\zeta;\xi)$ is analytical  in $\mathbb{C}\setminus \Sigma^{pc} $ with $\Sigma^{pc}=\left\lbrace\mathbb{R}e^{\varphi i} \right\rbrace \cup \left\lbrace\mathbb{R}e^{(\pi-\varphi) i} \right\rbrace$ shown in Figure \ref{sigpc};

 	$\blacktriangleright$ Jump condition: $M^{pc}$ has continuous boundary values $M^{pc}_\pm$ on $\Sigma^{pc}$ and
 	\begin{equation}
 	M^{pc}_+(\zeta;\xi)=M^{pc}_-(\zeta;\xi)V^{pc}(\zeta),\hspace{0.5cm}\zeta \in \Sigma^{\zeta},
 	\end{equation}
 	where   in the case $\xi\in[0,2)$
 	\begin{align}
 		V^{pc}(\zeta;\xi)=\left\{\begin{array}{ll}
 	\left(\begin{array}{cc}
 	1 & 0\\
 	-r_{\xi_1}\zeta^{2i\nu(\xi_1)}e^{-\frac{i}{2}\zeta^2} & 1
 	\end{array}\right),  & \zeta\in\mathbb{R}^+e^{\varphi i},\\[10pt]
 	\left(\begin{array}{cc}
 	1& \bar{r}_{\xi_1}\zeta^{-2i\nu(\xi_1)}e^{\frac{i}{2}\zeta^2}\\
 	0&1
 	\end{array}\right),   & \zeta\in \mathbb{R}^+e^{-\varphi i},\\[10pt]
 	\left(\begin{array}{cc}
 	1 & 0\\
 	\frac{r_{\xi_1}}{1+|r_{\xi_1}|^2}\zeta^{2i\nu(\xi_1)}e^{-\frac{i}{2}\zeta^2} & 1
 	\end{array}\right),   & \zeta\in \mathbb{R}^+e^{(-\pi+\varphi) i},\\[10pt]
 	\left(\begin{array}{cc}
 	1 & -\frac{\bar{r}_{\xi_1}}{1+|r_{\xi_1}|^2}\zeta^{-2i\nu(\xi_1)}e^{\frac{i}{2}\zeta^2}\\
 	0 & 1
 	\end{array}\right),   & \zeta\in \mathbb{R}^+e^{(\pi-\varphi) i}.
 	\end{array}\right.
 	\end{align}
 	and in the case $\xi\in(-0.25,0)$
 		\begin{align}
	V^{pc}(\zeta;\xi)=\left\{\begin{array}{ll}
 \left(\begin{array}{cc}
 1 & -\frac{\bar{r}_{\xi_1}}{1+|r_{\xi_1}|^2}\zeta^{-2i\nu(\xi_1)}e^{\frac{i}{2}\zeta^2}\\
 0 & 1
 \end{array}\right),  & \zeta\in\mathbb{R}^+e^{\varphi i},\\[10pt]
 \left(\begin{array}{cc}
 1 & 0\\
 \frac{r_{\xi_1}}{1+|r_{\xi_1}|^2}\zeta^{2i\nu(\xi_1)}e^{-\frac{i}{2}\zeta^2} & 1
 \end{array}\right),   & \zeta\in \mathbb{R}^+e^{(2\pi-\varphi) i},\\[10pt]
 \left(\begin{array}{cc}
 1& \bar{r}_{\xi_1}\zeta^{-2i\nu(\xi_1)}e^{\frac{i}{2}\zeta^2}\\
 0&1
 \end{array}\right),   & \zeta\in \mathbb{R}^+e^{(\pi+\varphi) i},\\[10pt]
 \left(\begin{array}{cc}
 1 & 0\\
 -r_{\xi_1}\zeta^{2i\nu(\xi_1)}e^{-\frac{i}{2}\zeta^2} & 1
 \end{array}\right),   & \zeta\in \mathbb{R}^+e^{ (\pi-\varphi)i}.
 \end{array}\right.
 	\end{align}

 	$\blacktriangleright$ Asymptotic behaviors:
 	\begin{align}
 	M^{pc}(\zeta;\xi) =& I+M^{pc}_1\zeta^{-1}+\mathcal{O}(\zeta^{-2}),\hspace{0.5cm}\zeta \rightarrow \infty.
 	\end{align}
 \end{RHP}	
\begin{figure}
	\centering
	\subfigure[]{
		\begin{tikzpicture}[node distance=2cm]
		\draw[](0,2.7)node[above]{$\xi\in[0,2)$};
		\draw[->](0,0)--(2,1.2)node[above]{$\mathbb{R}^+e^{\varphi i}$};
		\draw(0,0)--(-2,1.2)node[above]{$\mathbb{R}^+e^{(\pi-\varphi)i}$};
		\draw(0,0)--(-2,-1.2)node[below]{$\mathbb{R}^+e^{(-\pi+\varphi)i}$};
		\draw[->](0,0)--(2,-1.2)node[below]{$\mathbb{R}^+e^{-\varphi i}$};
		\draw[dashed](-2,0)--(2,0)node[right]{\scriptsize Re$z$};
		\draw[->](-2,-1.2)--(-1,-0.6);
		\draw[->](-2,1.2)--(-1,0.6);
		\coordinate (A) at (-1.2,0.5);
		\coordinate (B) at (-1.2,-0.5);
		\coordinate (G) at (1.4,0.5);
		\coordinate (H) at (1.4,-0.5);
		\coordinate (I) at (0,0);
		\fill (I) circle (1pt) node[below] {$0$};
		\end{tikzpicture}
	}
	\subfigure[]{
		\begin{tikzpicture}[node distance=2cm]
		\draw[](0,2.7) node[above]{$\xi\in(-0.25,0)$};
		\draw[->](0,0)--(2,1.5)node[above]{$\mathbb{R}^+e^{\varphi i}$};
		\draw(0,0)--(-2,1.5)node[above]{$\mathbb{R}^+e^{(\pi-\varphi)i}$};
		\draw(0,0)--(-2,-1.5)node[below]{$\mathbb{R}^+e^{(\pi+\varphi)i}$};
		\draw[->](0,0)--(2,-1.5)node[below]{$\mathbb{R}^+e^{(2\pi-\varphi)i}$};
		\draw[dashed](-2,0)--(2,0)node[right]{\scriptsize Re$z$};
		\draw[->](-2,-1.5)--(-1,-0.75);
		\draw[->](-2,1.5)--(-1,0.75);
		\coordinate (A) at (1,0.5);
		\coordinate (B) at (1,-0.5);
		\coordinate (G) at (-1,0.5);
		\coordinate (H) at (-1,-0.5);
		\coordinate (I) at (0,0);
		\fill (I) circle (1pt) node[below] {$0$};
		\end{tikzpicture}
	}
	\caption{The contour $\Sigma^{pc}$ in case $\xi\in[0,2)$ and $\xi\in(-0.25,0)$ respectively.}
	\label{sigpc}
\end{figure}
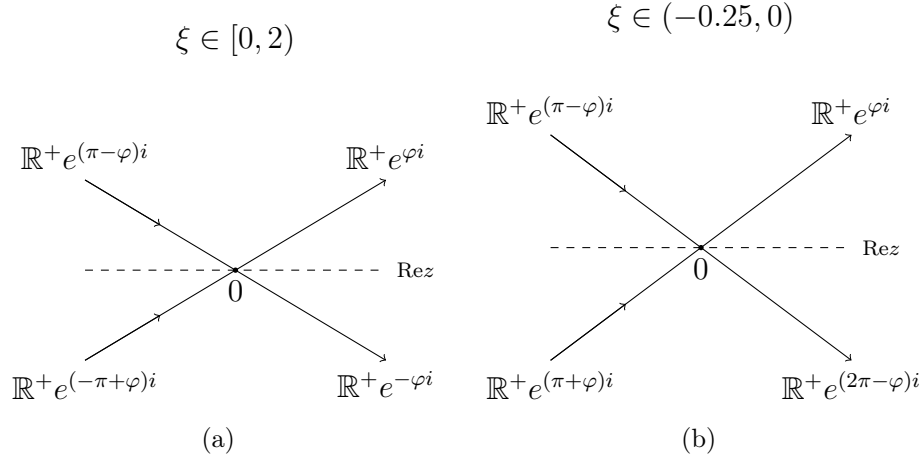
Then \cite{HG2009} Theorem A.1-A.6 proved that
\begin{align}
M^{lo,1}(z)=I+\frac{t^{-1/2}}{z-\xi_1}\frac{i\eta}{2} \left(\begin{array}{cc}
0 & [M^{pc}_1]_{12}\\
-[M^{pc}_1]_{21} & 0
\end{array}\right)+\mathcal{O}(t^{-1}).\label{asyMlo1}
\end{align}
And RHP \ref{RHPpc} has an explicit solution $	M^{pc}(\zeta)$, which is expressed in terms of solutions of the parabolic cylinder equation $\left(\frac{\partial^2}{\partial z^2}+\left(\frac{1}{2}-\frac{z^2}{2}+a \right)  \right)D_a(z)=0 $. In fact, Let
\begin{align}
M^{pc}(\zeta;\xi)=\Psi(\zeta;\xi)P(\xi)e^{\frac{i}{4}\eta\zeta^2\sigma_3}\zeta^{-i\eta\nu(\xi_1)\sigma_3},
\end{align}
where  in the case $\xi\in[0,2)$
\begin{align}
P(\xi)=\left\{\begin{array}{ll}
\left(\begin{array}{cc}
1 & 0\\
r_{\xi_1} & 1
\end{array}\right),  & \arg\zeta\in(0,\varphi) ,\\[10pt]
\left(\begin{array}{cc}
1& \bar{r}_{\xi_1}\\
0&1
\end{array}\right),   & \arg\zeta\in (-\varphi,0) ,\\[10pt]
\left(\begin{array}{cc}
1 & 0\\
\frac{r_{\xi_1}}{1+|r_{\xi_1}|^2} & 1
\end{array}\right),   & \arg\zeta\in (-\pi,-\pi+\varphi),\\[10pt]
\left(\begin{array}{cc}
1 & \frac{\bar{r}_{\xi_1}}{1+|r_{\xi_1}|^2}\\
0 & 1
\end{array}\right),   & \arg\zeta\in (\pi-\varphi,\pi),\\[10pt]
I,   & else.
\end{array}\right.
\end{align}
and in the case $\xi\in(-0.25,0)$
\begin{align}
P(\xi)=\left\{\begin{array}{ll}
\left(\begin{array}{cc}
1 & \frac{\bar{r}_{\xi_1}}{1+|r_{\xi_1}|^2}\\
0 & 1
\end{array}\right),  & \arg\zeta\in(0,\varphi) ,\\[10pt]
\left(\begin{array}{cc}
1 & 0\\
\frac{r_{\xi_1}}{1+|r_{\xi_1}|^2} & 1
\end{array}\right),   & \arg\zeta\in(2\pi-\varphi,2\pi) ,\\[10pt]
\left(\begin{array}{cc}
1& \bar{r}_{\xi_1}\\
0&1
\end{array}\right),   & \arg\zeta\in (\pi,\pi+\varphi) ,\\[10pt]
\left(\begin{array}{cc}
1 & 0\\
r_{\xi_1} & 1
\end{array}\right),   & \arg\zeta\in(\pi-\varphi,\pi),\\[10pt]
I,   & else.
\end{array}\right.
\end{align}
By construction, the matrix $\Psi$ is continuous along the rays of $\Sigma^{pc}$. And Due to the branch cut of the logarithmic function along $\eta\mathbb{R}^+$, the matrix $\Psi$  has the same (constant) jump matrix along the negative and positive real axis. The function $\Psi(\zeta;\xi)$  satisfies the following model RHP.
 \begin{RHP}\label{RHPpsi}
	Find a matrix-valued function  $ \Psi(\zeta;\xi)$ with following properties:
	
	$\blacktriangleright$ Analyticity: $\Psi(\zeta;\xi)$ is analytical  in $\mathbb{C}\setminus \mathbb{R}$;

	$\blacktriangleright$ Jump condition: $\Psi(\zeta;\xi)$ has continuous boundary values $\Psi_\pm(\zeta;\xi)$ on $\mathbb{R}$ and
	\begin{equation}
	\Psi_+(\zeta;\xi)=\Psi_-(\zeta;\xi)V^{\Psi}(\zeta),\hspace{0.5cm}\zeta \in \mathbb{R},
	\end{equation}
	where
	\begin{align}
	V^{\Psi}(\xi)=	\left(\begin{array}{cc}
	1+|r_{\xi_1}|^2 & \bar{r}_{\xi_1}\\
	r_{\xi_1} & 1
	\end{array}\right).
	\end{align}

	$\blacktriangleright$ Asymptotic behaviors:
	\begin{align}
	\Psi(\zeta;\xi) \sim& \left( I+M^{pc}_1\zeta^{-1}\right) \zeta^{i\eta\nu(\xi_1)\sigma_3}e^{-\frac{i}{4}\eta\zeta^2\sigma_3},\hspace{0.5cm}\zeta \rightarrow \infty.
	\end{align}
\end{RHP}	
For  brevity, denote $\tilde{\beta}^1_{12}=i\eta[M^{pc}_1]_{12}$ and $\tilde{\beta}^1=-i\eta[M^{pc}_1]_{21}$. The unique solution to Problem \ref{RHPpsi} is:

1. $\xi\in(-0.25,0)$,
when $\zeta\in\mathbb{C}^+$,
\begin{align}
\Psi(\zeta;\xi)=	\left(\begin{array}{cc}
e^{\frac{3}{4}\pi\nu(\xi_1)}D_{-i\nu(\xi_1)}(e^{-\frac{3}{4}\pi i}\zeta) & \frac{i\nu(\xi_1)}{\tilde{\beta}^1_{21}}e^{-\frac{\pi}{4}(\nu(\xi_1)+i)}D_{i\nu(\xi_1)-1}(e^{-\frac{\pi i}{4} }\zeta)\\
-\frac{i\nu(\xi_1)}{\tilde{\beta}^1_{12}}e^{\frac{3\pi}{4}(\nu(\xi_1)-i)}D_{-i\nu(\xi_1)-1}(e^{-\frac{3\pi i}{4} }\zeta) & e^{-\frac{\pi}{4}\nu(\xi_1)}D_{i\nu(\xi_1)}(e^{-\frac{\pi}{4} i}\zeta)
\end{array}\right),
\end{align}
when $\zeta\in\mathbb{C}^-$,
\begin{align}
\Psi(\zeta;\xi)=	\left(\begin{array}{cc}
e^{\frac{7\pi}{4}\nu(\xi_1)}D_{-i\nu(\xi_1)}(e^{-\frac{7\pi}{4} i}\zeta) & \frac{i\nu(\xi_1)}{\tilde{\beta}^1_{21}}e^{-\frac{5\pi}{4}(\nu(\xi_1)+i)}D_{i\nu(\xi_1)-1}(e^{-\frac{5\pi i}{4} }\zeta)\\
-\frac{i\nu(\xi_1)}{\tilde{\beta}^1_{12}}e^{\frac{7\pi}{4}(\nu(\xi_1)-i)}D_{-i\nu(\xi_1)-1}(e^{-\frac{7\pi i}{4} }\zeta) & e^{-\frac{5}{4}\pi\nu(\xi_1)}D_{i\nu(\xi_1)}(e^{-\frac{5}{4}\pi i}\zeta)
\end{array}\right).
\end{align}

2. $\xi\in[0,2)$,
when $\zeta\in\mathbb{C}^+$,
\begin{align}
\Psi(\zeta;\xi)=	\left(\begin{array}{cc}
e^{-\frac{\pi}{4}\nu(\xi_1)}D_{i\nu(\xi_1)}(e^{-\frac{\pi}{4} i}\zeta) & -\frac{i\nu(\xi_1)}{\tilde{\beta}^1_{21}}e^{\frac{3\pi}{4}(\nu(\xi_1)-i)}D_{-i\nu(\xi_1)-1}(e^{-\frac{3\pi i}{4} }\zeta)\\
\frac{i\nu(\xi_1)}{\tilde{\beta}^1_{12}}e^{-\frac{\pi}{4}(\nu(\xi_1)+i)}D_{i\nu(\xi_1)-1}(e^{-\frac{\pi i}{4} }\zeta) & e^{\frac{3}{4}\pi\nu(\xi_1)}D_{-i\nu(\xi_1)}(e^{-\frac{3}{4}\pi i}\zeta)
\end{array}\right),
\end{align}
when $\zeta\in\mathbb{C}^-$,
\begin{align}
\Psi(\zeta;\xi)=	\left(\begin{array}{cc}
 e^{\frac{3}{4}\pi\nu(\xi_1)}D_{i\nu(\xi_1)}(e^{\frac{3}{4}\pi i}\zeta) & -\frac{i\nu(\xi_1)}{\tilde{\beta}^1_{21}}e^{-\frac{\pi}{4}(\nu(\xi_1)-i)}D_{-i\nu(\xi_1)-1}(e^{\frac{\pi i}{4} }\zeta)\\
\frac{i\nu(\xi_1)}{\tilde{\beta}^1_{12}}e^{\frac{3\pi}{4}(\nu(\xi_1)+i)}D_{i\nu(\xi_1)-1}(e^{\frac{3\pi i}{4} }\zeta) &e^{-\frac{\pi}{4}\nu(\xi_1)}D_{-i\nu(\xi_1)}(e^{\frac{\pi}{4} i}\zeta)
\end{array}\right).
\end{align}
And when $\xi\in(-0.25,0)$,
\begin{align}
&\tilde{\beta}^1_{21}=\frac{\sqrt{2\pi}e^{\frac{5}{2}\pi\nu(\xi_1)}e^{-\frac{7\pi}{4} i}}{r_{\xi_1}\Gamma(-i\nu(\xi_1))},\hspace{0.5cm}\tilde{\beta}^1_{21}\tilde{\beta}^1_{12}=-\nu(\xi_1),\\
&|\tilde{\beta}^1_{21}|=-\frac{\nu(\xi_1)}{(1+|r(\xi_1)|^2)^3},\\
&\arg(\tilde{\beta}^1_{21})=\frac{5}{2}\pi\nu(\xi_1)-\frac{7\pi}{4} i-\arg r_{\xi_1}-\arg \Gamma(-i\nu(\xi_1));
\end{align}
when $\xi\in[0,2)$,
\begin{align}
&\tilde{\beta}^1_{21}=\frac{\sqrt{2\pi}e^{\frac{\pi}{2}\nu(\xi_1)}e^{-\frac{\pi}{4} i}}{r_{\xi_1}\Gamma(i\nu(\xi_1))},\hspace{0.5cm}\tilde{\beta}^1_{21}\tilde{\beta}^1_{12}=-\nu(\xi_1),\\
&|\tilde{\beta}^1_{21}|=-\frac{\nu(\xi_1)}{1+|r(\xi_1)|^2},\\
&\arg(\tilde{\beta}^1_{21})=\frac{\pi}{2}\nu(\xi_1)-\frac{\pi}{4} i-\arg r_{\xi_1}-\arg \Gamma(i\nu(\xi_1)).
\end{align}

A derivation of this result is given in \cite{RN6}, and a direct verification of the solution is given in  \cite{Liu3}. Substitute above consequence into (\ref{asyMlo1}) and obtain:
\begin{align}
M^{lo,1}(z)=I+\frac{t^{-1/2}}{z-\xi_1} \left(\begin{array}{cc}
0 & \tilde{\beta}^1_{12}\\
\tilde{\beta}^1_{21} & 0
\end{array}\right)+\mathcal{O}(t^{-1}).\label{asyMpc}
\end{align}
For the model around other stationary phase points, it  also admits
\begin{align}
M^{lo,k}(z)=I+\frac{t^{-1/2}}{z-\xi_k} \left(\begin{array}{cc}
0 & \tilde{\beta}^k_{12}\\
\tilde{\beta}^k_{21} & 0
\end{array}\right)+\mathcal{O}(t^{-1}),\label{asyMpck}
\end{align}
for $k=2,...,n(\xi)$.
When $\xi\in(-0.25,0)$, $k$ is  odd number or  $\xi\in[0,2)$, $k$ is  even number,
\begin{align}
r_{\xi_k}=r(\xi_k)T_k(\xi)^{2}e^{-2it\theta(\xi_k)}\exp\left\lbrace -i\nu(\xi_k)\log \left( 4t\theta''(\xi_k)\right) \right\rbrace,
\end{align}
and
\begin{align}
&\tilde{\beta}^k_{21}=\frac{\sqrt{2\pi}e^{\frac{5}{2}\pi\nu(\xi_k)}e^{-\frac{7\pi}{4} i}}{r_{\xi_k}\Gamma(-i\nu(\xi_k))},\hspace{0.5cm}\tilde{\beta}^k_{21}\tilde{\beta}^k_{12}=-\nu(\xi_k),\\
&|\tilde{\beta}^k_{21}|=-\frac{\nu(\xi_k)}{(1+|r(\xi_k)|^2)^3},\\
&\arg(\tilde{\beta}^k_{21})=\frac{5}{2}\pi\nu(\xi_k)-\frac{7\pi}{4} i-\arg r_{\xi_k}-\arg \Gamma(-i\nu(\xi_k));
\end{align}
and when $\xi\in(-0.25,0)$, $k$ is  even number or  $\xi\in[0,2)$, $k$ is  odd number,
\begin{align}
r_{\xi_k}=r(\xi_k)T_k(\xi)^{2}e^{-2it\theta(\xi_k)}\exp\left\lbrace -i\nu(\xi_k)\log \left(- 4t\theta''(\xi_k)\right) \right\rbrace,
\end{align}
and
\begin{align}
&\tilde{\beta}^k_{21}=\frac{\sqrt{2\pi}e^{\frac{\pi}{2}\nu(\xi_k)}e^{-\frac{\pi}{4} i}}{r_{\xi_k}\Gamma(i\nu(\xi_k))},\hspace{0.5cm}\tilde{\beta}^k_{21}\tilde{\beta}^k_{12}=-\nu(\xi_k),\\
&|\tilde{\beta}^k_{21}|=-\frac{\nu(\xi_k)}{1+|r(\xi_k)|^2},\\
&\arg(\tilde{\beta}^k_{21})=\frac{\pi}{2}\nu(\xi_k)-\frac{\pi}{4} i-\arg r_{\xi_k}-\arg \Gamma(i\nu(\xi_k)).
\end{align}
Then together with proposition \ref{dividepc}, wo final obtain
\begin{Proposition}\label{asymlo}
	As $t\to+\infty$,
	\begin{align}
	M^{lo}(z)=I+t^{-1/2}\sum_{ k=1 }^{n(\xi)}\frac{A_k(\xi)}{z-\xi_k} +\mathcal{O}(t^{-1}),
	\end{align}
	where
	\begin{align}
	A_k(\xi)=\left(\begin{array}{cc}
	0 & \tilde{\beta}^k_{12}\\
	\tilde{\beta}^k_{21} & 0
	\end{array}\right).
	\end{align}
\end{Proposition}

\section{The small norm RH problem  for error function }\label{sec7}

\quad \quad In this section,  we consider the error matrix-function $E(z;\xi)$. When $\xi\in(-\infty,-0.25)$ or $\xi\in(2,+\infty)$,
the definition (\ref{transm4})  implies that $E(z;\xi)\equiv I$, so only the case $\xi\in(-0.25,2)$ need to be investigate. And
we can obtain a RH problem  for the matrix function  $E(z;\xi)$ for $\xi\in(-0.25,2)$.

\noindent\textbf{RHP11.}   Find a matrix-valued function $E(z;\xi)$  with following properties:

$\blacktriangleright$ Analyticity: $E(z;\xi)$ is analytical  in $\mathbb{C}\setminus  \Sigma^{(E)} $, where
$$\Sigma^{(E)}= \partial U_(\xi)\cup
(\Sigma^{(2)}\setminus U_(\xi);$$

$\blacktriangleright$ Asymptotic behaviors:
\begin{align}
&E(z;\xi) \sim I+\mathcal{O}(z^{-1}),\hspace{0.5cm}|z| \rightarrow \infty;
\end{align}

$\blacktriangleright$ Jump condition: $E(z;\xi)$ has continuous boundary values $E_\pm(z;\xi)$ on $\Sigma^{(E)}$ satisfying
$$E_+(z;\xi)=E_-(z;\xi)V^{(E)}(z),$$
where the jump matrix $V^{(E)}(z)$ is given by
\begin{equation}
V^{(E)}(z)=\left\{\begin{array}{llll}
M^{(r)}(z)V^{(2)}(z)M^{(r)}(z)^{-1}, & z\in \Sigma^{(2)}\setminus U_(\xi),\\[4pt]
M^{(r)}(z)M^{lo}(z)M^{(r)}(z)^{-1},  & z\in \partial U_(\xi),
\end{array}\right. \label{deVE}
\end{equation}
which is  shown in  Figure \ref{figE}.

We will show  that for large times, the error function  $E(z;\xi)$  solves following small norm  RH problem.

\begin{figure}[H]
	\centering
\subfigure[]{
	\begin{tikzpicture}
\draw(4.35,0.18)--(5,0.5);
\draw[-<](4.35,0.18)--(4.5,0.25);
\draw(3.64,0.15)--(2.5,0.6);
\draw[->](3.64,-0.15)--(3.25,-0.3);
\draw(4.35,-0.18)--(5,-0.5);
\draw[->](3.64,0.15)--(3.25,0.3);
\draw(3.64,-0.15)--(2.5,-0.6);
\draw[-<](4.35,-0.18)--(4.5,-0.25);
\draw(-4.35,0.18)--(-5,0.5);
\draw[-<](-4.35,0.18)--(-4.5,0.25);
\draw(-3.64,0.15)--(-2.5,0.6);
\draw[->](-3.64,-0.15)--(-3.25,-0.3);
\draw(-4.35,-0.18)--(-5,-0.5);
\draw[->](-3.64,0.15)--(-3.25,0.3);
\draw(-3.64,-0.15)--(-2.5,-0.6);
\draw[-<](-4.35,-0.18)--(-4.5,-0.25);
\draw(-0.64,0.15)--(0,0.5);
\draw[->](-0.64,0.15)--(-0.4,0.28);
\draw(-1.35,0.16)--(-2.5,0.6);
\draw[-<](-1.35,-0.16)--(-1.75,-0.31);
\draw(-0.64,-0.15)--(0,-0.5);
\draw[-<](-1.35,0.16)--(-1.75,0.31);
\draw(-1.35,-0.16)--(-2.5,-0.6);
\draw[->](-0.64,-0.15)--(-0.4,-0.28);
\draw[dashed](-5,0)--(5,0)node[right]{ Re$z$};
\draw(0.64,0.15)--(0,0.5);
\draw[->](0.64,0.15)--(0.4,0.28);
\draw(1.35,0.16)--(2.5,0.6);
\draw[-<](1.35,-0.16)--(1.75,-0.31);
\draw(0.64,-0.15)--(0,-0.5);
\draw[-<](1.35,0.16)--(1.75,0.31);
\draw(1.35,-0.16)--(2.5,-0.6);
\draw[->](0.64,-0.15)--(0.4,-0.28);
\draw[->](2.5,0)--(2.5,0.6);
\draw[->](2.5,0)--(2.5,-0.6);
\draw[->](-2.5,0)--(-2.5,0.6);
\draw[->](-2.5,0)--(-2.5,-0.6);
\draw[->](0,0)--(0,0.5);
\draw[->](0,0)--(0,-0.5);
\coordinate (I) at (0,0);
\fill (I) circle (1pt) node[below] {$0$};
\coordinate (A) at (-4,0);
\fill (A) circle (1pt) node[below] {$\xi_4$};
\coordinate (b) at (-1,0);
\fill (b) circle (1pt) node[below] {$\xi_3$};
\coordinate (e) at (4,0);
\fill (e) circle (1pt) node[below] {$\xi_1$};
\coordinate (f) at (1,0);
\draw[thick,red](1,0) circle (0.4);
\fill (f) circle (1pt) node[below] {$\xi_2$};
\draw[thick,red](4,0) circle (0.4);
\draw[thick,red](-1,0) circle (0.4);
\draw[thick,red](-4,0) circle (0.4);
\coordinate (c) at (-2,0);
\fill[red] (c) circle (1pt) node[below] {\scriptsize$-1$};
\coordinate (d) at (2,0);
\fill[red] (d) circle (1pt) node[below] {\scriptsize$1$};
\end{tikzpicture}
}
\subfigure[]{
\begin{tikzpicture}
\draw[dashed](-6.5,0)--(6.5,0)node[right]{ Re$z$};
\coordinate (I) at (0,0);
\fill (I) circle (1pt) node[below] {$0$};
\coordinate (c) at (-3,0);
\fill[red] (c) circle (1pt) node[below] {\scriptsize$-1$};
\coordinate (D) at (3,0);
\fill[red] (D) circle (1pt) node[below] {\scriptsize$1$};
\draw(-0.575,0.187)--(-0,0.7);
\draw[->](-0.575,0.187)--(-0.4,0.35);
\draw(-1.03,0.186)--(-1.5,0.6);
\draw[-<](-1.03,-0.186)--(-1.15,-0.3);
\draw(-0.575,-0.187)--(-0,-0.7);
\draw[-<](-1.03,0.186)--(-1.15,0.3);
\draw(0.575,0.187)--(0,0.7);
\draw[-<](0.575,0.187)--(0.4,0.35);
\draw(1.03,0.186)--(1.5,0.6);
\draw[->](1.03,-0.186)--(1.15,-0.3);
\draw(0.575,-0.187)--(0,-0.7);
\draw[->](1.03,0.186)--(1.15,0.3);
\draw(1.03,-0.186)--(1.5,-0.6);
\draw[-<](0.575,-0.187)--(0.4,-0.35);
\draw(-1.03,-0.186)--(-1.5,-0.6);
\draw[->](-0.575,-0.187)--(-0.4,-0.35);
\draw(-1.97,0.187)--(-1.5,0.6);
\draw[-<](-2.43,0.187)--(-2.65,0.35);
\draw(-1.97,-0.187)--(-1.5,-0.6);
\draw[->](-1.97,-0.187)--(-1.85,-0.3);
\draw(-2.43,0.187)--(-3.1,0.7);
\draw[->](-1.97,0.187)--(-1.85,0.3);
\draw(-2.43,-0.187)--(-3.1,-0.7);
\draw[-<](-2.43,-0.187)--(-2.65,-0.35);
\draw(1.97,0.187)--(1.5,0.6);
\draw[->](2.43,0.187)--(2.65,0.35);
\draw(1.97,-0.187)--(1.5,-0.6);
\draw[-<](1.97,-0.187)--(1.85,-0.3);
\draw(2.43,0.187)--(3.1,0.7);
\draw[-<](1.97,0.187)--(1.85,0.3);
\draw(2.43,-0.187)--(3.1,-0.7);
\draw[->](2.43,-0.187)--(2.65,-0.35);
\draw(-5.64,0.18)--(-6.5,0.9);
\draw[-<](-5.64,0.18)--(-5.96,0.446);
\draw(-5.16,0.18)--(-4.7,0.6);
\draw[->](-5.16,-0.18)--(-4.93,-0.39);
\draw(-5.64,-0.18)--(-6.5,-0.9);
\draw[->](-5.16,0.18)--(-4.93,0.39);
\draw(-5.16,-0.18)--(-4.7,-0.6);
\draw[-<](-5.64,-0.18)--(-5.96,-0.446);
\draw(-3.75,0.187)--(-3.1,0.7);
\draw[->](-3.75,0.187)--(-3.55,0.35);
\draw(-4.25,0.187)--(-4.7,0.6);
\draw[-<](-4.25,-0.187)--(-4.42,-0.35);
\draw(-3.75,-0.187)--(-3.1,-0.7);
\draw[-<](-4.25,0.187)--(-4.42,0.35);
\draw(-4.25,-0.187)--(-4.7,-0.6);
\draw[->](-3.75,-0.187)--(-3.55,-0.35);
\draw(3.75,0.187)--(3.1,0.7);
\draw[->](3.75,0.187)--(3.55,0.35);
\draw(4.25,0.187)--(4.7,0.6);
\draw[-<](4.25,-0.187)--(4.42,-0.35);
\draw(3.75,-0.187)--(3.1,-0.7);
\draw[-<](4.25,0.187)--(4.42,0.35);
\draw(4.25,-0.187)--(4.7,-0.6);
\draw[->](3.75,-0.187)--(3.55,-0.35);
\draw[->](-1.5,0)--(-1.5,0.6);
\draw[->](-1.5,0)--(-1.5,-0.6);
\draw[->](-4.7,0)--(-4.7,0.6);
\draw[->](-4.7,0)--(-4.7,-0.6);
\draw[->](-3.1,0)--(-3.1,0.7);
\draw[->](-3.1,0)--(-3.1,-0.7);
\draw(5.64,0.18)--(6.5,0.9);
\draw[-<](5.64,0.18)--(5.96,0.446);
\draw(5.16,0.18)--(4.7,0.6);
\draw[->](5.16,-0.18)--(4.93,-0.39);
\draw(5.64,-0.18)--(6.5,-0.9);
\draw[->](5.16,0.18)--(4.93,0.39);
\draw(5.16,-0.18)--(4.7,-0.6);
\draw[-<](5.64,-0.18)--(5.96,-0.446);
\draw[->](1.5,0)--(1.5,0.6);
\draw[->](1.5,0)--(1.5,-0.6);
\draw[->](4.7,0)--(4.7,0.6);
\draw[->](4.7,0)--(4.7,-0.6);
\draw[->](3.1,0)--(3.1,0.7);
\draw[->](3.1,0)--(3.1,-0.7);
\draw[->](0,0)--(0,0.7);
\draw[->](0,0)--(0,-0.7);
\coordinate (A) at (-5.4,0);
\fill (A) circle (1pt) node[below] {$\xi_8$};
\draw[thick,red](-5.4,0) circle (0.3);
\coordinate (b) at (-4,0);
\draw[thick,red](-4,0) circle (0.3);
\fill (b) circle (1pt) node[below] {$\xi_7$};
\coordinate (C) at (-0.8,0);
\draw[thick,red](-0.8,0) circle (0.3);
\fill (C) circle (1pt) node[below] {$\xi_5$};
\coordinate (d) at (-2.2,0);
\draw[thick,red](-2.2,0) circle (0.3);
\fill (d) circle (1pt) node[below] {$\xi_6$};
\coordinate (E) at (5.4,0);
\draw[thick,red](5.4,0) circle (0.3);
\fill (E) circle (1pt) node[below] {$\xi_1$};
\coordinate (R) at (4,0);
\draw[thick,red](4,0) circle (0.3);
\fill (R) circle (1pt) node[below] {$\xi_2$};
\coordinate (T) at (0.8,0);
\draw[thick,red](0.8,0) circle (0.3);
\fill (T) circle (1pt) node[below] {$\xi_4$};
\coordinate (Y) at (2.2,0);
\draw[thick,red](2.2,0) circle (0.3);
\fill (Y) circle (1pt) node[below] {$\xi_3$};
\end{tikzpicture}
}
\caption{  The jump contour $\Sigma^{(E)}$ for the $E(z;\xi)$. The red circles are $U(\xi)$. }
	\label{figE}
\end{figure}
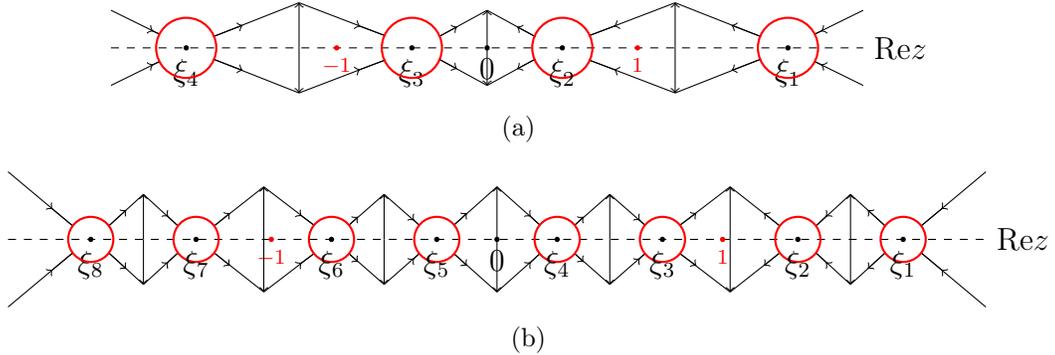

By using   {Proposition \ref{prov2}}, we have the following estimates of $V^{(E)}$:
\begin{equation}
\parallel V^{(E)}(z)-I\parallel_p\lesssim\left\{\begin{array}{llll}
\exp\left\{-tK_p\right\},  & z\in \Sigma_{kj}\setminus U(\xi),\\[6pt]
\exp\left\{-tK_p'\right\},   & z\in \Sigma'_{k\pm}.
\end{array}\right. \label{VE-I}
\end{equation}

For $z\in \partial U(\xi)$,  $M^{(r)}(z)$ is bounded,  so   by using  (\ref{asymlo}),  we find that
\begin{equation}
| V^{(E)}(z)-I|=   \big|M^{(r)}(z)^{-1}(M^{lo}(z)-I)M^{(r)}(z) \big| = \mathcal{O}(|t|^{-1/2}).\label{VE}
\end{equation}
Therefore,    the   existence and uniqueness  of  the RHP11 can  shown  by using  a  small-norm RH problem \cite{RN9,RN10}.  Moreover, according to Beal-Coifman theory,
the solution of  the RHP11  can be given by
\begin{equation}
E(z;\xi)=I+\frac{1}{2\pi i}\int_{\Sigma^{(E)}}\dfrac{\left( I+\varpi(s)\right) (V^{(E)}(s)-I)}{s-z}ds,\label{Ez}
\end{equation}
where the $\varpi\in L^\infty(\Sigma^{(E)})$ is the unique solution of following equation
\begin{equation}
(1-C_E)\varpi=C_E\left(I \right).
\end{equation}
And  $C_E$ is  a integral operator:$L^\infty(\Sigma^{(E)})\to L^2(\Sigma^{(E)})$ defined by
\begin{equation}
C_E(f)(z)=C_-\left( f(V^{(E)}(z) -I)\right) ,
\end{equation}
where the $C_-$ is the usual Cauchy projection operator on $\Sigma^{(E)}$
\begin{equation}
C_-(f)(s)=\lim_{z\to \Sigma^{(E)}_-}\frac{1}{2\pi i}\int_{\Sigma^{(E)}}\dfrac{f(s)}{s-z}ds.
\end{equation}
By  (\ref{VE}),   we have
\begin{equation}
\parallel C_E\parallel\leq\parallel C_-\parallel \parallel V^{(E)}(z)-I\parallel_2 \lesssim \mathcal{O}(t^{-1/2}),
\end{equation}
which implies that  $1-C_E$ is invertible  for   sufficiently large $t$.    So  $\varpi$  exists and is unique,
moreover
\begin{equation}
\parallel \varpi\parallel_{L^\infty(\Sigma^{(E)})}\lesssim\dfrac{\parallel C_E\parallel}{1-\parallel C_E\parallel}\lesssim|t|^{-1/2}.\label{normrho}
\end{equation}
In order to reconstruct the solution $u(y,t)$ of (\ref{mch}), we need the asymptotic behavior of $E(z;\xi)$ as $z\to i$ and the long time asymptotic behavior of $E(i)$. Note that when we estimate its  asymptotic behavior, from (\ref{Ez}) and (\ref{VE-I}) we only need to consider the calculation on $\partial U(\xi)$ because it  approach zero exponentially on other boundary.
\begin{Proposition}\label{asyE}
	As $z\to i$, we have
	\begin{align}
	E(z;\xi)=E(i)+E_1(z-i)+\mathcal{O}((z-i)^2),
	\end{align}
	where
	\begin{align}
	E(i)=I+\frac{1}{2\pi i}\int_{\Sigma^{(E)}}\dfrac{\left( I+\varpi(s)\right) (V^{(E)}-I)}{s-i}ds,
	\end{align}
	with long time asymptotic behavior
	\begin{equation}
	E(i)=I+t^{-1/2}H^{(0)}+\mathcal{O}(|t|^{-1}).\label{E0t}
	\end{equation}
	And
	\begin{align}
	H^{(0)}&=\sum_{ k=1 }^{n(\xi)}\frac{1}{2\pi i}\int_{\partial U_{\xi_k}}\dfrac{M^{(r)}(s)^{-1}A_k(\xi)M^{(r)}(s)}{(s-i)( s- \xi_k)}ds\nonumber\\
	&=\sum_{ k=1 }^{n(\xi)}\frac{1}{\xi_k-i} M^{(r)}(\xi_k)^{-1}A_k(\xi)M^{(r)}(\xi_k).
	\end{align}
	The last equality  follows from a residue calculation. Moreover,
	\begin{equation}
	E_1=-\frac{1}{2\pi i}\int_{\Sigma^{(E)}}\dfrac{\left( I+\varpi(s)\right) (V^{(E)}-I)}{(z-i)^2}ds,
	\end{equation}
	satisfying long time asymptotic behavior condition
	\begin{equation}
	E_1=|t|^{-1/2}H^{(1)}+\mathcal{O}(|t|^{-1}),\label{E1t}
	\end{equation}
	where
	\begin{align}
	H^{(1)}=-&\sum_{ k=1 }^{n(\xi)}\frac{1}{2\pi i}\int_{\partial U_{\xi_k}}\dfrac{M^{(r)}(s)^{-1}A_k(\xi)M^{(r)}(s)}{(s-i)^2( s- \xi_k)}ds\nonumber\\
	=&-\sum_{ k=1 }^{n(\xi)}\frac{1}{(\xi_k-i)^2} M^{(r)}(\xi_k)^{-1}A_k(\xi)M^{(r)}(\xi_k).
	\end{align}
\end{Proposition}
In order to facilitate calculation, denote
\begin{align}
f_{11}=[H^{(1)}]_{11}+[H^{(1)}]_{12}+[H^{(1)}]_{21}+[H^{(1)}]_{22},\label{f11}
\end{align}
and
\begin{align}
f_{12}=&\frac{[M^{(r)}_\Lambda]_{11}([H^{(0)}]_{11}+[H^{(1)}]_{21})+[M^{(r)}_\Lambda]_{21}([H^{(0)}]_{12}+[H^{(1)}]_{22})}{[M^{(r)}_\Lambda]_{11}+[M^{(r)}_\Lambda]_{21}}\nonumber\\
&-\frac{[M^{(r)}_\Lambda]_{12}([H^{(0)}]_{11}+[H^{(1)}]_{21})+[M^{(r)}_\Lambda]_{22}([H^{(0)}]_{12}+[H^{(1)}]_{22})}{[M^{(r)}_\Lambda]_{12}+[M^{(r)}_\Lambda]_{22}}.\label{f12}
\end{align}

\section{Analysis  on   the pure $\bar{\partial}$-Problem}\label{sec8}
\quad Now we consider the Proposition and the long time asymptotics behavior of $M^{(3)}$.
The $\bar{\partial}$-problem4  of $M^{(3)}$ is equivalent to the integral equation
\begin{equation}
M^{(3)}(z)=I+\frac{1}{\pi}\iint_\mathbb{C}\dfrac{M^{(3)}(s)W^{(3)} (s)}{s-z}dm(s),\label{m3}
\end{equation}
where $m(s)$ is the Lebesgue measure on the $\mathbb{C}$. Denote $C_z$ as the left Cauchy-Green integral  operator,
\begin{equation*}
fC_z(z)=\frac{1}{\pi}\iint_C\dfrac{f(s)W^{(3)} (s)}{s-z}dm(s).
\end{equation*}
Then above equation can be rewritten as
\begin{equation}
M^{(3)}(z)=I\cdot\left(I-C_z \right) ^{-1}.\label{deM3}
\end{equation}
As we discussing preceding, $W^{(3)}$ has different  properties and structures in the case $\xi\in(-\infty,-0.25)\cup(2,+\infty)$ and $\xi\in(-0.25,2)$. This means we need to research it respectively.

\subsection{$\xi\in(-\infty,-0.25)\cup(2,+\infty)$ case}
To prove the existence of operator $\left(I-C_z \right) ^{-1}$, we have following Lemma.
\begin{lemma}\label{Cz}
	The norm of the integral operator $C_z$ decay to zero as $t\to\infty$:
	\begin{equation}
	\parallel C_z\parallel_{L^\infty\to L^\infty}\lesssim |t|^{-1/2},
	\end{equation}
	which implies that  $\left(I-C_z \right) ^{-1}$ exists.
\end{lemma}
\begin{proof}
	For any $f\in L^\infty$,
	\begin{align}
	\parallel fC_z \parallel_{L^\infty}&\leq \parallel f \parallel_{L^\infty}\frac{1}{\pi}\iint_C\dfrac{|W^{(3)} (s)|}{|z-s|}dm(s)\nonumber.
	\end{align}
	Consequently, we only need to  evaluate the integral
	$\frac{1}{\pi}\iint_C\dfrac{|W^{(3)} (s)|}{|z-s|}dm(s)$. As $W^{(3)} (s)$ is a sectorial function,  we only need to consider it on ever sector. Recall the definition of $W^{(3)} (s)=M^{(r)}(z)\bar{\partial}R^{(2)}(z)M^{(r)}(z)^{-1}$.  $W^{(3)} (s)\equiv0$ out of $\bar{\Omega}$. We only detail the case for matrix functions having support in the sector $\Omega_1$ as $\xi< -0.25$.
	Proposition \ref{asyE} and \ref{unim}  implies  the boundedness of $M^{(r)}(z)$ and $M^{(r)}(z)^{-1}$ for $z\in \bar{\Omega}$, so
	\begin{equation}
		\frac{1}{\pi}\iint_{\Omega_1}\dfrac{|W^{(3)} (s)|}{|z-s|}dm(s)\lesssim \frac{1}{\pi}\iint_{\Omega_1}\dfrac{|\bar{\partial}R_1 (s) e^{2it\theta}|}{|z-s|}dm(s).
	\end{equation}
Referring  to (\ref{dbarRj}) in proposition \ref{proR}, the integral $\iint_{\Omega_1}\dfrac{|\bar{\partial}R_1 (s)|}{|z-s|}dm(s)$	can be divided to two part:
\begin{align}
	\iint_{\Omega_1}\dfrac{|\bar{\partial}R_1 (s)|e^{-2t\text{Im}\theta}}{|z-s|}dm(s)\lesssim I_1+I_2,	
\end{align}
with
\begin{align}
	I_1=\iint_{\Omega_1}\dfrac{|p_1' (s)|e^{-2t\text{Im}\theta}}{|z-s|}dm(s),\hspace{0.5cm}
	I_2=\iint_{\Omega_1}\dfrac{|s|^{-1/2}e^{-2t\text{Im}\theta}}{|z-s|}dm(s).	
\end{align}
Let $s=u+vi=le^{i\vartheta}$, $z=x+yi$. In the following computation,  we will use the inequality
\begin{align}
	\parallel |s-z|^{-1}\parallel_{L^q(v,+\infty)}&=\left\lbrace \int_{z_0}^{+\infty}\left[  \left( \frac{u-x}{v-y}\right) ^2+1\right]^{-q/2}
d\left( \frac{u-x}{|v-y|}\right)\right\rbrace ^{1/q}|v-y|^{1/q-1}\nonumber\\
&\lesssim |v-y|^{1/q-1},
\end{align}
with $1\leq q<+\infty$ and $\frac{1}{p}+\frac{1}{q}=1$. Moreover, by corollary \ref{Imtheta}, for $z\in\Omega_1$,
\begin{equation}
e^{-2t\text{Im}\theta}\leq e^{-c(\xi)tv}.
\end{equation}
Therefore,
	\begin{align}
	I_1&\leq \int_{0}^{+\infty}\int_{v}^{\infty}\dfrac{|p_1' (s)|}{|z-s|}dudv\leq \parallel |s-z|^{-1}\parallel_{L^2(\mathbb{R}^+)} \parallel p_1'\parallel_{L^2(\mathbb{R}^+)} e^{-c(\xi)tv}dv\nonumber\\
	&\lesssim \int_{0}^{+\infty}|v-y|^{-1/2} e^{-c(\xi)tv}dv\lesssim t^{-1/2}.
	\end{align}
Before we estimating the second item, we  consider for $p>2$,
\begin{align}
	\left( \int_{v}^{+\infty}|\sqrt{u^2+v^2}|^{-\frac{p}{2}}du\right) ^{\frac{1}{p}}&=\left( \int_{v}^{+\infty}|l|^{-\frac{p}{2}+1}u^{-1}dl\right) ^{\frac{1}{p}}\lesssim v^{-\frac{1}{2}+\frac{1}{p}}.
\end{align}
By Cauchy-Schwarz inequality,
\begin{align}
	I_2&\leq \int_{0}^{+\infty}\parallel |s-z|^{-1}\parallel_{L^q(\mathbb{R}^+)} \parallel |z|^{-1/2}\parallel_{L^p(\mathbb{R}^+)}e^{-c(\xi)tv}dv\nonumber\\
	&\lesssim\int_{0}^{+\infty}|v-y|^{1/q-1}v^{-\frac{1}{2}+\frac{1}{p}}e^{-c(\xi)tv}dv\nonumber\\
	&\lesssim\int_{0}^{+\infty}v^{-\frac{1}{2}}e^{-c(\xi)tv}dv\lesssim t^{-1/2}.
\end{align}
  So the proof is completed.
\end{proof}
Take $z=i$ in (\ref{m3}), then
\begin{equation}
	M^{(3)}(i)=I+\frac{1}{\pi}\iint_\mathbb{C}\dfrac{M^{(3)}(s)W^{(3)} (s)}{s-i}dm(s).
\end{equation}
To reconstruct the solution of (\ref{mch}), we need following proposition.
\begin{Proposition}\label{asyM3i}
	There exist a small positive constant $\frac{1}{4}>\rho$ such that the solution $M^{(3)}(z)$  of  $\bar{\partial}$-problem  admits the following estimation:
	\begin{align}
		\parallel M^{(3)}(i)-I\parallel=\parallel\frac{1}{\pi}\iint_\mathbb{C}\dfrac{M^{(3)}(s)W^{(3)} (s)}{s-i}dm(s)\parallel\lesssim t^{-1+2\rho}.\label{m3i}
	\end{align}
	As $z\to i$, $M^{(3)}(z)$ has asymptotic expansion:
	\begin{equation}
		M^{(3)}(z)=M^{(3)}(i)+M^{(3)}_1(x,t)(z-i)+\mathcal{O}((z-i)^{2}),
	\end{equation}
	where $M^{(3)}_1(x,t) $ is a $z$-independent coefficient with
\begin{equation}
	M^{(3)}_1(x,t)=\frac{1}{\pi}\iint_C\dfrac{M^{(3)}(s)W^{(3)} (s)}{(s-i)^2}dm(s).
\end{equation}
	There exist constants $T_1$, such that for all $t>T_1$, $M^{(3)}_1(x,t)$  satisfies
	\begin{equation}
		|M^{(3)}_1(x,t)|\lesssim t^{-1+2\rho}.\label{M31}
	\end{equation}
\end{Proposition}
\begin{proof}
	First we estimate (\ref{m3i}). The proof proceeds along the same steps as the proof of above Proposition. 	{Lemma \ref{Cz}} and (\ref{deM3}) implies that for large $t$,   $\parallel M^{(3)}\parallel_\infty \lesssim1$. And for same reason, we only estimate the integral on sector $\Omega_1$ as $\xi<-0.25$. Let $s=u+vi=le^{i\vartheta}$.	We also divide $M^{(3)}(i)-I$ to two parts, but  this time we use another estimation (\ref{dbarRj2}):
	\begin{equation}
	\frac{1}{\pi}\iint_\mathbb{C}\dfrac{M^{(3)}(s)W^{(3)} (s)}{s-i}dm(s)\lesssim I_3+I_4,
	\end{equation}
	with
	\begin{align}
		I_3=\iint_{\Omega_1}\dfrac{|p_1' (s)|e^{-2t\text{Im}\theta}}{|i-s|}dm(s),\hspace{0.5cm}
		I_4=\iint_{\Omega_1}\dfrac{|s|^{-1}e^{-2t\text{Im}\theta}}{|i-s|}dm(s).	
	\end{align}	
	For $r\in H^{1,1}(\mathbb{R})$, $r'\in L^1(\mathbb{R})$, which together with $|p_1'|\lesssim |r'|$ implies $p_1'\in L^1(\mathbb{R})$. So
	\begin{align}
		I_3\leq&\int_{0}^{+\infty}\int_{v}^{+\infty}\dfrac{|p_1' (s)|e^{-c(\xi)tv}}{|i-s|}dudv\nonumber\\
		&\leq \int_{0}^{+\infty} \parallel p_1'\parallel_{L^1(\mathbb{R}^+)} \sqrt{2} e^{-c(\xi)tv} dv\lesssim \int_{0}^{+\infty}  e^{-c(\xi)tv} dv\nonumber\\
		&\leq t^{-1}.
	\end{align}
The second  inequality  from $|i-s|\leq \frac{1}{\sqrt{2}}$ for $s\in\Omega_1$. And we use the sufficiently  small  positive constant $\rho$ to bound $I_4$. And we partition it to two parts:
\begin{align}
I_4\leq\int_{0}^{\frac{1}{2}}\int_{v}^{+\infty}\dfrac{|s|^{-1}e^{-c(\xi)tv}}{|i-s|}dudv+\int_{\frac{1}{2}}^{+\infty}\int_{v}^{+\infty}\dfrac{|s|^{-1}e^{-c(\xi)tv}}{|i-s|}dudv.
\end{align}
For $0<v<\frac{1}{2}$, $|s-i|^2=u^2+(v-1)^2>u^2+v^2=|s|^2$, while as $v>\frac{1}{2}$, $|s-i|^2<|s|^2$. Then  the first integral has:
\begin{align}
	&\int_{0}^{\frac{1}{2}}\int_{v}^{+\infty}\dfrac{|s|^{-1}e^{-c(\xi)tv}}{|i-s|}dudv\nonumber\\
	&\leq\int_{0}^{\frac{1}{2}}\int_{v}^{+\infty} (u^2+v^2)^{-\frac{1}{2}-\rho}(u^2+(v-1)^2)^{-\frac{1}{2}+\rho}due^{-c(\xi)tv}dv\nonumber\\
	&\leq \int_{0}^{\frac{1}{2}}\left[ \int_{v}^{+\infty}\left(1+\left(\frac{u}{v} \right)^2  \right) ^{-\frac{1}{2}-\rho}v^{-2\rho} d\frac{u}{v}\right]  (v^2+(v-1)^2)^{-\frac{1}{2}+\rho}e^{-c(\xi)tv}dv\nonumber\\
	&\lesssim  \int_{0}^{\frac{1}{2}}v^{-2\rho}(\frac{1}{2})^{-\frac{1}{2}+\rho}e^{-c(\xi)tv}dv\lesssim t^{-1+2\rho}.
\end{align}
The second integral can be bounded in same way:
\begin{align}
	&\int_{\frac{1}{2}}^{+\infty}\int_{v}^{+\infty}\dfrac{|s|^{-1}e^{-c(\xi)tv}}{|i-s|}dudv\nonumber\\
	&\leq\int_{\frac{1}{2}}^{+\infty}(2v^2)^{-\frac{1}{2}+\rho}|v-1|^{-2\rho}e^{-c(\xi)tv}dv\nonumber\\
	&\lesssim \int_{\frac{1}{2}}^{1} (1-v)^{-2\rho}e^{-c(\xi)tv}dv+\int_{1}^{+\infty}(v-1)^{-2\rho}e^{-c(\xi)tv}dv\nonumber\\
	&\leq e^{-\frac{c(\xi)}{2}t}\int_{\frac{1}{2}}^{1} (1-v)^{-2\rho}dv+e^{-c(\xi)t}\int_{1}^{+\infty}(v-1)^{-2\rho}e^{-c(\xi)t(v-1)}d(v-1)\nonumber\\
	&\lesssim e^{-\frac{c(\xi)}{2}t}.
\end{align}
This estimation  is strong enough to obtain the result (\ref{m3i}). And (\ref{M31}) is obtained by  exploiting the fact  $|i-s|\leq \frac{1}{\sqrt{2}}$ for $s\in\Omega_1$.
\end{proof}

\subsection{$\xi\in(-0.25,2)$ case}
\quad The first step also is to prove the existence of operator $\left(I-C_z \right) ^{-1}$.
\begin{lemma}\label{Cz1}
	The norm of the integral operator $C_z$ decay to zero as $t\to\infty$:
	\begin{equation}
	\parallel C_z\parallel_{L^\infty\to L^\infty}\lesssim |t|^{-1/4},
	\end{equation}
	which implies that  $\left(I-C_z \right) ^{-1}$ exists.
\end{lemma}
\begin{proof}
		For any $f\in L^\infty$,
	\begin{align}
	\parallel fC_z \parallel_{L^\infty}&\leq \parallel f \parallel_{L^\infty}\frac{1}{\pi}\iint_C\dfrac{|W^{(3)} (s)|}{|z-s|}dm(s)\nonumber.
	\end{align}
	Consequently, we only need to  evaluate the integral
	$\frac{1}{\pi}\iint_C\dfrac{|W^{(3)} (s)|}{|z-s|}dm(s)$. As $W^{(3)} (s)$ is a sectorial function, it is just need to consider it on ever sector. Recall the definition of $W^{(3)} (s)=M^{R}(z)\bar{\partial}R^{(2)}(z)M^{R}(z)^{-1}$.  $W^{(3)} (s)\equiv0$ out of $\bar{\Omega}$. We only detail the case for matrix functions having support in the sector $\Omega_{11}$ as $\xi\in (-0.25,0)$.
	Proposition \ref{asyE}, \ref{unim} and \ref{asymlo}  implies  the boundedness of $M^{R}(z)$ and $M^{R}(z)^{-1}$ for $z\in \bar{\Omega}$, so
	\begin{equation}
	\frac{1}{\pi}\iint_{\Omega_{11}}\dfrac{|W^{(3)} (s)|}{|z-s|}dm(s)\lesssim \frac{1}{\pi}\iint_{\Omega_{11}}\dfrac{|\bar{\partial}R_{11} (s) e^{2it\theta}|}{|z-s|}dm(s).
	\end{equation}
	Referring  to (\ref{dbarRj3}) in proposition \ref{proR1}, the integral $\iint_{\Omega_{11}}\dfrac{|\bar{\partial}R_{11} (s)|}{|z-s|}dm(s)$	can be divided to two part:
	\begin{align}
	\iint_{\Omega_{11}}\dfrac{|\bar{\partial}R_{11} (s)|e^{-2t\text{Im}\theta}}{|z-s|}dm(s)\lesssim \hat{I}_1+\hat{I}_2,	
	\end{align}
	with
	\begin{align}
	&\hat{I}_1=\iint_{\Omega_{11}}\dfrac{|p_{11}'(s)|e^{-2t\text{Im}\theta}}{|z-s|}dm(s),\\
	&\hat{I}_2=\iint_{\Omega_{11}}\dfrac{|s-\xi_1|^{-1/2}e^{-2t\text{Im}\theta}}{|z-s|}dm(s).
	\end{align}
	Recall that $z=x+yi$, $s=\xi_1+u+vi$ with $x,y,u,v\in\mathbb{R}$, then lemma \ref{theta2} gives that
	\begin{align}
	\hat{I}_1&\leq\int_{0}^{+\infty}\int_{v}^{+\infty}\dfrac{|p_{11}'(s)|}{|z-s|}e^{-c(\xi)tv\frac{u^2+2u\xi_1+v^2}{4+|s|^2}}dudv\nonumber\\
	&\leq\int_{0}^{+\infty}\parallel p_{11}'\parallel_2\parallel |z-s|^{-1}\parallel_2e^{-2c(\xi)tv^2\frac{u+\xi_1}{4+\xi_1^2+v^2}}dv\nonumber\\
	&\lesssim\int_{0}^{+\infty}|v-y|^{-1/2}e^{-2c(\xi)tv^2\frac{u+\xi_1}{4+\xi_1^2+v^2}}dv\nonumber\\
	&=\left(\int_{0}^{y}+\int_{y}^{+\infty}\right)|v-y|^{-1/2}e^{-2c(\xi)tv^2\frac{v+\xi_1}{4+\xi_1^2+v^2}}dv  .
	\end{align}
	For the first integral, we use the  inequality $e^{-z}\lesssim z^{-1/4}$
	\begin{align}
	\int_{0}^{y}(y-v)^{-1/2}e^{-2c(\xi)tv^2\frac{u+\xi_1}{4+\xi_1^2+v^2}}dv\lesssim \int_{0}^{y} (y-v)^{-1/2}v^{-1/2}dvt^{-1/4}\lesssim t^{-1/4}.
	\end{align}
	And for the second integral, we make the substitution $w=v-y:0\to +\infty$
	\begin{align}
	\int_{y}^{+\infty}(v-y)^{-1/2}e^{-2c(\xi)tv^2\frac{u+\xi_1}{4+\xi_1^2+v^2}}dv&=\int_{0}^{+\infty}w^{-1/2}e^{-2c(\xi)ty\frac{y+\xi_1}{4+\xi_1^2+y^2}(w+y)}dw\nonumber\\
	&=\int_{0}^{+\infty}w^{-1/2}e^{-2c(\xi)ty\frac{y+\xi_1}{4+\xi_1^2+y^2}w}dwe^{-2c(\xi)ty\frac{y+\xi_1}{4+\xi_1^2+y^2}y}\nonumber\\
	&\lesssim e^{-t2c(\xi)y\frac{y+\xi_1}{4+\xi_1^2+y^2}y}.
	\end{align}
	And similar with lemma \ref{Cz}, we bound $\hat{I}_2$. For $p>2$, and $1/p+1/q=1$,
	\begin{align}
	\hat{I}_2&\leq \int_{0}^{+\infty} \parallel |s-\xi_1|^{-1/2}\parallel_p\parallel |z-s|^{-1}\parallel_q e^{-c(\xi)tv\frac{u^2+2u\xi_1+v^2}{4+|s|^2}}dv\nonumber\\
	&\lesssim \int_{0}^{+\infty}v^{1/p-1/2}|y-v|^{1/q-1}e^{-c(\xi)tv\frac{u^2+2u\xi_1+v^2}{4+|s|^2}}dv\nonumber\\
	&=\left(\int_{0}^{y}+\int_{y}^{+\infty}\right)v^{1/p-1/2}|y-v|^{1/q-1}e^{-c(\xi)tv\frac{u^2+2u\xi_1+v^2}{4+|s|^2}}dv.
	\end{align}
	Analogously,
	\begin{align}
	\int_{0}^{y}v^{1/p-1/2}|y-v|^{1/q-1}e^{-c(\xi)tv\frac{u^2+2u\xi_1+v^2}{4+|s|^2}}dv&\lesssim \int_{0}^{y}v^{1/p-1}(y-v)^{1/q-1} dv t^{-1/4}\nonumber\\
	&\lesssim t^{-1/4},
	\end{align}
	and
	\begin{align}
	&\int_{y}^{+\infty}v^{1/p-1/2}|y-v|^{1/q-1}e^{-c(\xi)tv\frac{u^2+2u\xi_1+v^2}{4+|s|^2}}dv\nonumber\\
	&\leq \int_{y}^{+\infty}(v-y)^{-1/2} e^{-2c(\xi)ty\frac{y+\xi_1}{4+\xi_1^2+y^2}(v-y)}dve^{-2c(\xi)ty\frac{y+\xi_1}{4+\xi_1^2+y^2}y}\nonumber\\
	&\lesssim e^{-t2c(\xi)y\frac{y+\xi_1}{4+\xi_1^2+y^2}y}.
	\end{align}
	Then the result is confirmed.
\end{proof}
Take $z=i$ in (\ref{m3}), then
\begin{equation}
M^{(3)}(i)=I+\frac{1}{\pi}\iint_\mathbb{C}\dfrac{M^{(3)}(s)W^{(3)} (s)}{s-i}dm(s).
\end{equation}
To reconstruct the solution of (\ref{mch}), we need following proposition.
\begin{Proposition}\label{asyM3i1}
	The solution $M^{(3)}(z)$  of  $\bar{\partial}$-problem  admits the following estimation:
	\begin{align}
	\parallel M^{(3)}(i)-I\parallel=\parallel\frac{1}{\pi}\iint_\mathbb{C}\dfrac{M^{(3)}(s)W^{(3)} (s)}{s-i}dm(s)\parallel\lesssim t^{-3/4}.\label{m3i1}
	\end{align}
	As $z\to i$, $M^{(3)}(z)$ has asymptotic expansion:
	\begin{equation}
	M^{(3)}(z)=M^{(3)}(i)+M^{(3)}_1(x,t)(z-i)+\mathcal{O}((z-i)^{2}),
	\end{equation}
	where $M^{(3)}_1(x,t) $ is a $z$-independent coefficient with
	\begin{equation}
	M^{(3)}_1(x,t)=\frac{1}{\pi}\iint_C\dfrac{M^{(3)}(s)W^{(3)} (s)}{(s-i)^2}dm(s).
	\end{equation}
	There exist constants $T_1$, such that for all $t>T_1$, $M^{(3)}_1(x,t)$  satisfies
	\begin{equation}
	|M^{(3)}_1(x,t)|\lesssim t^{-3/4}.\label{M311}
	\end{equation}
\end{Proposition}
\begin{proof}
	Analogously in proposition \ref{asyM3i}, we first estimate (\ref{m3i1}). The proof proceeds along the same steps as the proof of above Proposition. 	{Lemma \ref{Cz1}} and (\ref{deM3}) implies that for large $t$,   $\parallel M^{(3)}\parallel_\infty \lesssim1$. So it is just need to bound $\iint_\mathbb{C}\dfrac{|W^{(3)} (s)|}{s-i}dm(s)$. And we only give the details on $\Omega_{11}$, the integral on other region can be obtained in the same way. Referring  to (\ref{dbarRj3}) in proposition \ref{proR1}, this integral 	can be divided to two part:
	\begin{equation}
	\frac{1}{\pi}\iint_\mathbb{C}\dfrac{M^{(3)}(s)W^{(3)} (s)}{s-i}dm(s)\lesssim \hat{I}_3+\hat{I}_4,
	\end{equation}
	with
	\begin{align}
	&\hat{I}_3=\iint_{\Omega_{11}}\dfrac{|p_{11}'(s)|e^{-2t\text{Im}\theta}}{|i-s|}dm(s),\\
	&\hat{I}_4=\iint_{\Omega_{11}}\dfrac{|s-\xi_1|^{-1/2}e^{-2t\text{Im}\theta}}{|i-s|}dm(s).
	\end{align}
	For $\hat{I}_3$, note that $|i-s|^{-1}$ has nonzero  maximum
	\begin{align}
	\hat{I}_3&\leq \int_{0}^{+\infty}\int_{v}^{+\infty}|p_{11}'(s)|e^{-c(\xi)tv\frac{u^2+2u\xi_1+v^2}{4+|s|^2}}dudv\nonumber\\
	&\leq \int_{0}^{+\infty}e^{-c(\xi)t\frac{v^3}{4+v^2+(v+\xi_1)^2}} \parallel p_{11}'(s) \parallel_2\left( \int_v ^{+\infty}e^{-u2c(\xi)tv\frac{v+2\xi_1}{4+v^2+(v+\xi_1)^2}}du\right) ^{1/2}dv\nonumber\\
	&\lesssim t^{-1/2}\int_{0}^{+\infty}\left( v\frac{v+2\xi_1}{4+v^2+(v+\xi_1)^2}\right) ^{-1/2}e^{-2c(\xi)tv^2\frac{v+2\xi_1}{4+v^2+(v+\xi_1)^2}}dv\nonumber\\
	&= t^{-1/2}\left( \int_{0}^{1}+\int_{1}^{+\infty}\right) \left( v\frac{v+2\xi_1}{4+v^2+(v+\xi_1)^2}\right) ^{-1/2}e^{-2c(\xi)tv^2\frac{v+2\xi_1}{4+v^2+(v+\xi_1)^2}}dv.
	\end{align}
	For the first integral, use that $\frac{4+v^2+(v+\xi_1)^2}{v(v+2\xi_1)}\lesssim v^{-1}$, then
	\begin{align}
	&\int_{0}^{1}\left( v\frac{v+2\xi_1}{4+v^2+(v+\xi_1)^2}\right) ^{-1/2}e^{-2c(\xi)tv^2\frac{v+2\xi_1}{4+v^2+(v+\xi_1)^2}}dv\nonumber\\
	&\lesssim \int_{0}^{1} v^{-1/2}e^{-c'tv^2}dv\nonumber\\
	&\lesssim t^{-1/4}.
	\end{align}
	As for the second item,
	\begin{align}
	&\int_{1}^{+\infty}\left( v\frac{v+2\xi_1}{4+v^2+(v+\xi_1)^2}\right) ^{-1/2}e^{-2c(\xi)tv^2\frac{v+2\xi_1}{4+v^2+(v+\xi_1)^2}}dv\nonumber\\
	&\lesssim \int_{1}^{+\infty} e^{-c't(v+\xi_1)}dv\nonumber\\
	&\lesssim t^{-1}e^{-c't(1+\xi_1)}.
	\end{align}
	So $\hat{I}_3 \lesssim t^{-3/4}$. And for $\hat{I}_4$, similarly we take $2<p<4$, and $1/p+1/q=1$, then
	\begin{align}
	\hat{I}_4&\leq \int_{0}^{+\infty}\int_{v}^{+\infty}|s-\xi_1|^{-1/2}e^{-c(\xi)tv\frac{u^2+2u\xi_1+v^2}{4+|s|^2}}dudv\nonumber\\
	&\leq \int_{0}^{+\infty}e^{-c(\xi)t\frac{v^3}{4+v^2+(v+\xi_1)^2}} \parallel |s-\xi_1|^{-1/2} \parallel_p\left( \int_v ^{+\infty}e^{-uqc(\xi)tv\frac{v+2\xi_1}{4+v^2+(v+\xi_1)^2}}du\right) ^{1/q}dv\nonumber\\
	&\lesssim t^{-1/q}\int_{0}^{+\infty}v^{2/p-1/2}\left( v\frac{v+2\xi_1}{4+v^2+(v+\xi_1)^2}\right) ^{-1/q}e^{-qc(\xi)tv^2\frac{v+2\xi_1}{4+v^2+(v+\xi_1)^2}}dv\nonumber\\
	&= t^{-1/q}\left( \int_{0}^{1}+\int_{1}^{+\infty}\right) v^{1/p-1/2}\left( v\frac{v+2\xi_1}{4+v^2+(v+\xi_1)^2}\right) ^{-1/q}e^{-qc(\xi)tv^2\frac{v+2\xi_1}{4+v^2+(v+\xi_1)^2}}dv.
	\end{align}
	The first  item has evaluation as
	\begin{align}
	&\int_{0}^{1} v^{1/p-1/2}\left( v\frac{v+2\xi_1}{4+v^2+(v+\xi_1)^2}\right) ^{-1/q}e^{-qc(\xi)tv^2\frac{v+2\xi_1}{4+v^2+(v+\xi_1)^2}}dv\nonumber\\
	&\int_{0}^{1} v^{2/p-3/2}e^{-c_p tv}dv\nonumber\\
	&\lesssim t^{1/4-1/p}.
	\end{align}
	And
	\begin{align}
	&\int_{1}^{+\infty} v^{1/p-1/2}\left( v\frac{v+2\xi_1}{4+v^2+(v+\xi_1)^2}\right) ^{-1/q}e^{-qc(\xi)tv^2\frac{v+2\xi_1}{4+v^2+(v+\xi_1)^2}}dv\nonumber\\
	&\lesssim\int_{1}^{+\infty}e^{-c_p't(v+\xi_1)}dv\nonumber\\
	& \lesssim t^{-1}e^{-c_p't(1+\xi_1)}.
	\end{align}
	So $\hat{I}_4 \lesssim t^{1/4-1/p-1/q}=t^{-3/4} $. And (\ref{M311}) comes from $|M^{(3)}_1(x,t)|\lesssim \parallel M^{(3)}(i)-I\parallel$.
\end{proof}

\section{Asymptotic approximation   for the  mCH equation }\label{sec9}

\quad Now we begin to construct the long time asymptotics of the mch equation (\ref{mch}).
 Inverting the sequence of transformations (\ref{transm1}), (\ref{transm2}), (\ref{transm3}) and (\ref{transMr}), we have
\begin{align}
M(z)=&M^{(3)}(z)E(z;\xi)M^{(r)}(z)R^{(2)}(z)^{-1}T(z)^{-\sigma_3}.
\end{align}
To  reconstruct the solution $u(x,t)$ by using (\ref{recons u}),   we take $z\to i$ out of $\bar{\Omega}$. In this case,  $ R^{(2)}(z)=I$. Further using   Propositions \ref{proT}  and  \ref{asyM3i},  we can obtain that as $z\to i$  behavior
\begin{align}
	M(z)=&\left(M^{(3)}(i)+ M^{(3)}_1(z)(z-i)\right) E(z;\xi)M^{(r)}_\Lambda(z)\nonumber\\
	&T(i)^{-\sigma_3}\left( 1-\frac{1}{2\pi i}\int _{I(\xi)}\dfrac{ \log (1+|r(s)|^2)}{(s-i)^2}ds(z-i)\right)^{-\sigma_3} + \mathcal{O}((z-i)^{2}).
\end{align}
Its  long time asymptotics rely on different case of $\xi$. For $\xi\in(-\infty,-0.25)\cup(2,+\infty)$,
\begin{align}
	M(z)=& M^{(r)}_\Lambda(z)T(i)^{-\sigma_3}\left(I- \frac{1}{2\pi i}\int _{I(\xi)}\dfrac{ \log (1+|r(s)|^2)}{(s-i)^2}ds (z-i) \right)^{-\sigma_3}\nonumber\\
	& + \mathcal{O}((z-i)^{-2})+\mathcal{O}(t^{-1+2\rho}),
\end{align}
and
\begin{align}
	M(i)=& M^{(r)}_\Lambda(i)T(i)^{-\sigma_3}+\mathcal{O}(t^{-1+2\rho}).
\end{align}
Substitute above estimation into  (\ref{recons u}) and (\ref{recons x}) and obtain
\begin{align}
	u(x,t)=&u(y(x,t),t)=\lim_{z\to i}\frac{1}{z-i}\left(1- \dfrac{(M_{11}(z)+M_{21}(z))(M_{12}(z)+M_{22}(z)) }{(M_{11}(i)+M_{21}(i))(M_{12}(i)+M_{22}(i))}\right) \nonumber\\
	=&\lim_{z\to i}\frac{1}{z-i}\left(1- \dfrac{([M^{(r)}_\Lambda]_{11}(z)+[M^{(r)}_\Lambda]_{21}(z))([M^{(r)}_\Lambda]_{12}(z)+[M^{(r)}_\Lambda]_{22}(z)) }{([M^{(r)}_\Lambda]_{11}(i)+[M^{(r)}_\Lambda]_{21}(i))([M^{(r)}_\Lambda]_{12}(i)+[M^{(r)}_\Lambda]_{22}(i))}\right)\nonumber\\
	&+\mathcal{O}(t^{-1+2\rho})\nonumber\\
	=&u^r(x,t;\tilde{\mathcal{D}}) +\mathcal{O}(t^{-1+2\rho}),  \label{resultu}
\end{align}
and
\begin{align}
	x(y,t)&=y+c_+(x,t)=y-\ln\left( \frac{M_{12}(i)+M_{22}(i)}{M_{11}(i)+M_{21}(i)}\right)\nonumber\\
	&=y+2\ln\left(  T(i) \right)+c_+^r(x,t;\tilde{\mathcal{D}}) +\mathcal{O}(t^{-1+2\rho}),
\end{align}
where $u^r(x,t;\tilde{\mathcal{D}})$ and $c_+^r(x,t;\tilde{\mathcal{D}})$ shown in Corollary \ref{sol}. And when $\xi\in(-0.25,2)$, Propositions \ref{proT},  \ref{asyE},  \ref{asyM3i} and (\ref{asymr}) give that
\begin{align}
M(z)=&\left( M^{(3)}(i)+M^{(3)}_1(z-i)\right)\left(E(i)+E_1(z-i) \right) \nonumber\\
& \left( M^{(r)}_\Lambda(i)+M^{(r)}_{\Lambda,1}(z-i)\right) T(i)^{-\sigma_3}\left( I-\frac{1}{2\pi i}\int _{I(\xi)}\dfrac{ \log (1+|r(s)|^2)}{(s-i)^2}ds\sigma_3(z-i)\right)\nonumber\\
&+\mathcal{O}((z-i)^{-2})\nonumber\\
=&\left(I+H^{(0)}t^{-1/2}+E_1(z-i) \right) \left( M^{(r)}_\Lambda(i)+M^{(r)}_{\Lambda,1}(z-i)\right) \nonumber\\
&T(i)^{-\sigma_3}\left( I-\frac{1}{2\pi i}\int _{I(\xi)}\dfrac{ \log (1+|r(s)|^2)}{(s-i)^2}ds\sigma_3(z-i)\right)\nonumber\\
&+\mathcal{O}((z-i)^{-2})+\mathcal{O}(t^{-3/4}).
\end{align}

Therefore, we achieve main result of this paper.

\begin{theorem}\label{last}   Let $u(x,t)$ be the solution for  the initial-value problem (\ref{mch}) with generic data   $u_0(x)\in \mathcal{S}(\mathbb{R})$ and scatting data $\left\lbrace  r(z),\left\lbrace \zeta_n,C_n\right\rbrace^{4N_1+2N_2}_{n=1}\right\rbrace$. Let $\xi=\frac{y}{t}$  and
denote $q^r_\Lambda(x,t)$ be the $\mathcal{N}(\Lambda)$-soliton solution corresponding to   scattering data
$\tilde{\mathcal{D}}=\left\lbrace  0,\left\lbrace \zeta_n,C_nT^2(\zeta_n)\right\rbrace_{n\in\Lambda}\right\rbrace$ shown in Corollary \ref{sol}. And $\Lambda$ is  defined in (\ref{devide}). Then there exist a large constant $T_1=T_1(\xi)$, for all $T_1<t\to\infty$,

1. when $\xi\in(-\infty,-0.25)\cup(2,+\infty)$,
\begin{align}
	u(x,t)=&u(y(x,t),t)=\lim_{z\to i}\frac{1}{z-i}\left(1- \dfrac{(M_{11}(z)+M_{21}(z))(M_{12}(z)+M_{22}(z)) }{(M_{11}(i)+M_{21}(i))(M_{12}(i)+M_{22}(i))}\right) \nonumber\\
	=&\lim_{z\to i}\frac{1}{z-i}\left(1- \dfrac{([M^{(r)}_\Lambda]_{11}(z)+[M^{(r)}_\Lambda]_{21}(z))([M^{(r)}_\Lambda]_{12}(z)+[M^{(r)}_\Lambda]_{22}(z)) }{([M^{(r)}_\Lambda]_{11}(i)+[M^{(r)}_\Lambda]_{21}(i))([M^{(r)}_\Lambda]_{12}(i)+[M^{(r)}_\Lambda]_{22}(i))}\right)\nonumber\\
	&+\mathcal{O}(t^{-1+2\rho})\nonumber\\
	=&u^r(x,t;\tilde{\mathcal{D}}) +\mathcal{O}(t^{-1+2\rho}),
\end{align}
and
\begin{align}
	x(y,t)&=y+c_+(x,t)=y-\ln\left( \frac{M_{12}(i)+M_{22}(i)}{M_{11}(i)+M_{21}(i)}\right)\nonumber\\
	&=y+2\ln\left(  T(i) \right)+c_+^r(x,t;\tilde{\mathcal{D}}) +\mathcal{O}(t^{-1+2\rho}),
\end{align}
where  $T(z)$ and $u^r(x,t;\tilde{\mathcal{D}})$ and $c_+^r(x,t;\tilde{\mathcal{D}})$ are show in Propositions \ref{proT} and Corollary (\ref{sol}) respectively.

2. when $\xi\in(-0.25,2)$,
\begin{align}
u(x,t)=&u(y(x,t),t)=\lim_{z\to i}\frac{1}{z-i}\left(1- \dfrac{(M_{11}(z)+M_{21}(z))(M_{12}(z)+M_{22}(z)) }{(M_{11}(i)+M_{21}(i))(M_{12}(i)+M_{22}(i))}\right) \nonumber\\
=&u^r(x,t;\tilde{\mathcal{D}}) +f_{11}t^{-1/2}+\mathcal{O}(t^{-3/4}),
\end{align}
and
\begin{align}
x(y,t)&=y+c_+(x,t)=y-\ln\left( \frac{M_{12}(i)+M_{22}(i)}{M_{11}(i)+M_{21}(i)}\right)\nonumber\\
&=y-2\ln\left(  T(i) \right)+c_+^r(x,t;\tilde{\mathcal{D}}) +f_{12}t^{-1/2}+\mathcal{O}(t^{-3/4}),
\end{align}
where  $T(z)$ and $u^r(x,t;\tilde{\mathcal{D}})$, $c_+^r(x,t;\tilde{\mathcal{D}})$, $f_{11}$ and $f_{12}$ are show in Propositions \ref{proT}, Corollary (\ref{sol}), (\ref{f11}) and (\ref{f12}) respectively.
\end{theorem}
Our results also show that the poles on curve soliton solutions of mCH  equation has dominant contribution to the solution as $t\to\infty$.\vspace{6mm}

\noindent\textbf{Acknowledgements}

This work is supported by  the National Science
Foundation of China (Grant No. 11671095,  51879045).

\hspace*{\parindent}
\\


\begin{thebibliography}{10}

\bibitem{Gardner1967}  C.S. Gardner,  J.M. Green,  M.D.  Kruskal  and R.M.  Miura, Method
for solving the Korteweg-de Vries equation, Phys. Rev. Lett. 19(1967), 1095-1097.



\bibitem{Zakharov1974} V.E. Zakharov  and A.B.  Shabat, A scheme for integrating the nonlinear
equations of mathematical physics by the method of the inverse scattering
problem, Funk. Anal. Pril. 6(1974),   43-53.



\bibitem{Zakharov1979}  V.E. Zakharov  and A.B. Shabat, A scheme for integrating the nonlinear
equations of mathematical physics by the method of the inverse scattering
problem. II, Funk. Anal. Pril. 13(1979),   13-22.

\bibitem{Manakov1974} S.V. Manakov, Nonlinear Fraunhofer diffraction,  Sov. Phys.-JETP 38(1974), 693-696.


\bibitem{ZM1976}
V. E. Zakharov, S. V.  Manakov,
\newblock {Asymptotic behavior of nonlinear wave systems integrated by the inverse scattering method},
\newblock {\em  Soviet Physics JETP,}   44(1976), 106-112.



\bibitem{SPC}
P. C. Schuur,
\newblock {Asymptotic analysis of soliton products},
\newblock {\em Lecture Notes in Mathematics,}  1232, 1986.


\bibitem{BRF}
R. F. Bikbaev,
\newblock {Asymptotic-behavior as t-infinity of the solution to the cauchy-problem for the landau-lifshitz equation},
\newblock {\em  Theor. Math. Phys,}  77(1988), 1117-1123.

\bibitem{Foka}
R. F. Bikbaev,
\newblock {Soliton generation for initial-boundary-value problems,}
\newblock {\em  Phys. Rev. Lett.}, 68(1992), 3117-3120.





\bibitem{RN6}
X. Zhou, P. Deift,
\newblock  A steepest descent method for oscillatory Riemann-Hilbert problems.
\newblock {\em Ann. Math.}, 137(1993),  295-368.


\bibitem{RN9}
X. Zhou, P. Deift,
\newblock  Long-time behavior of the non-focusing nonlinear Schr$\ddot{o}$dinger equation--a case study,
\newblock {\em Lectures in Mathematical Sciences}, Graduate School of Mathematical Sciences, University of Tokyo, 1994.


\bibitem{RN10}
P. Deift, X. Zhou,
\newblock Long-time asymptotics for solutions of the NLS equation with initial data in a weighted Sobolev space,
\newblock {\em Comm. Pure Appl. Math.}, 56(2003), 1029-1077.




\bibitem{Grunert2009}
K. Grunert,  G. Teschl,
\newblock   Long-time asymptotics for  the Korteweg de Vries equation  via  noninear  steepest descent.
\newblock {\em Math. Phys. Anal. Geom.},     12(2009), 287-324.

\bibitem{MonvelCH}
A. B. de Monvel, A. Kostenko, D. Shepelsky, G. Teschl,
\newblock  Long-time asymptotics for the Camassa-Holm equation,
\newblock {\em SIAM J. Math. Anal}, 41(2009),  1559-1588.

\bibitem{Monvel1}  Boutet de Monvel, A., Its, A., Kotlyarov, V.: Long-time asymptotics for the focusing NLS equation with
time-periodic boundary condition on the half-line. Commun. Math. Phys. 290(2): 479¨C522 (2009)


\bibitem{Monvel2}  Boutet de Monvel, A., Lenells, J., Shepelsky, D.: Long-time asymptotics for the Degasperis-Procesi
equation on the half-line. Ann. Inst. Fourier 69(2019), 171-230.

\bibitem{xu2015}
J. Xu, E. G. Fan,
\newblock  Long-time asymptotics for the Fokas-Lenells equation with decaying initial value problem: Without solitons,
\newblock {\em  J. Differential Equations,}    259(2015), 1098-1148.


\bibitem{xusp}
J. Xu,
\newblock  Long-time asymptotics for the short pulse equation,
\newblock {\em  J. Differential Equations,}  265(2018), 3494-3532.


\bibitem{Geng3}  Liu, H., Geng, X.G., Xue, B.: The Deift-Zhou steepest descent method to long-time asymptotics for the
Sasa¨CSatsuma equation. J. Differential Equations 265(2018), 5984-6008

\bibitem{XF2020} J. Xu, E. G. Fan,  Long-time asymptotic behavior for the complex short pulse equation
JianXua, EnguiFanb.  J. Differential Equations  269(2020), 10322-10349

\bibitem{MandM2006}
K. T. R. McLaughlin, P. D. Miller,
\newblock {The $\bar{\partial}$ steepest descent method and the asymptotic behavior of polynomials orthogonal on
the unit circle with fixed and exponentially varying non-analytic weights},
\newblock {\em Int.  Math. Res. Not.}, (2006), Art. ID 48673.

\bibitem{MandM2008}
K. T. R. McLaughlin, P. D. Miller,
\newblock {The $\bar{\partial}$ steepest descent method for orthogonal polynomials on the real line with varying weights},
\newblock {\em  Int. Math. Res. Not.},  (2008), Art. ID  075.

\bibitem{DandMNLS}
M. Dieng, K. D. T. McLaughlin,
\newblock {Dispersive asymptotics for linear and integrable equations by the Dbar steepest descent method},
\newblock {\em } Nonlinear dispersive partial differential equations and inverse scattering,
253-291, Fields Inst. Commun., 83, Springer, New York,  2019

\bibitem{fNLS}
M. Borghese, R. Jenkins, K. T. R. McLaughlin, Miller P,
\newblock { Long-time asymptotic behavior of the focusing nonlinear Schr$\ddot{o}$dinger equation, }
\newblock {\em  Ann. I. H. Poincar$\acute{e}$ Anal}, 35(2018), 887-920.



\bibitem{Liu3}
R. Jenkins, J. Liu, P. Perry, C. Sulem,
\newblock   Soliton resolution for the derivative nonlinear Schr$\ddot{o}$dinger equation,
\newblock {\em Commun. Math. Phys.},  363(2018), 1003-1049.

\bibitem{SandRNLS}
S. Cuccagna, R. Jenkins,
\newblock {On asymptotic stability of N-solitons of the defocusing nonlinear Schr$\ddot{o}$dinger equation, }
\newblock {\em  Comm. Math. Phys}, 343(2016), 921-969.

\bibitem{YF1} Y. L. Yang, E. G. Fan,  Soliton resolution for the short-pulse equation,  J. Differential Equations, 280(2021), 644-689


\bibitem{YF2}  Y. L. Yang, E. G. Fan, Soliton resolution for the three-wave resonant interaction equation, arXiv:2101.03512

\bibitem{YF3}  Y. L. Yang, E. G. Fan, Long-time asymptotic behavior of the modified Camassa-Holm equation, arXiv:2101.02489

\bibitem{YF4}  Q.Y. Cheng, E. G. Fan,   Soliton resolution for the focusing Fokas-Lenells equation with weighted Sobolev initial data, arXiv:2010.08714








\bibitem{Fokas}
 A. S. Fokas,
\newblock  On a class of physically important integrable equations,
\newblock {\em  Phys. D},  87(1995), 145-150.


\bibitem{BF1996}
 B. Fuchssteiner,
\newblock  Some tricks from the symmetry-toolbox for nonlinear equations: generalizations of the Camassa-Holm equation,
\newblock {\em  Physica D},  95(1996), 229-243.

\bibitem{PP1996}
P. J. Plver and P. Rosenau,
\newblock  Tri-Hamiltonian duality between solitons and
solitary-wave solutions having compact support,
\newblock {\em  Phys. Rev. E.},   53 (1996), 1900-1906.


\bibitem{Qiao}
 Z. Qiao,
\newblock   A new integrable equation with cuspons and W/M-shape-peaks solitons,
\newblock {\em  J. Math. Phys., }  47(2006), 112701.

\bibitem{THFan}
Y. Hou, E. Fan  and Z. Qiao,
\newblock  The algebro-geometric solutions for the Fokas-Olver-Rosenau-Qiao (FORQ) hierarchy,
\newblock {\em   J. Geom. Phys., }  117(2017), 105-133.


\bibitem{WLM}
G.H.  Wang, Q. P. Liu and H. Mao,
\newblock  The modified Camassa-Holm equation: Backlund transformation and nonlinear superposition formula,
\newblock {\em  J. Phys. A: Math. Theor., }  53 (2020) 294003 (15pp).


\bibitem{CS1}
X. K. Chang, J. Szmigielski,
\newblock  Lax integrability of the modified Camassa-Holm equation and the concept of peakons,
\newblock {\em   J. Nonlinear Math. Phys., }  23(2016), 563-572.
\bibitem{CS2}
X. K. Chang, J. Szmigielski,
\newblock  Lax integrability and the peakon problemfor the modified Camassa-Holm equation,
\newblock {\em   Comm. Math. Phys., }  358(2018), 295-341.


\bibitem{Gao2018}
Y. Gao and J. G. Liu,
\newblock  The modified Camassa-Holm equation in Lagrangian coordinates,
\newblock {\em   Discrete Contin.
	Dyn. Syst. Ser. B, }  23(2018), 2545-2592.

\bibitem{Gui2013}
G. Gui, Y. Liu, P. J. Olver  and C. Qu,
\newblock  Wave-breaking and peakons for a modified Camassa-Holm equation,
\newblock {\em   Comm. Math. Phys., }  319(2013), 731-759.




\bibitem{Mon}
A. Boutet de Monvel,  I. Karpenko, and D. Shepelsky,
\newblock A Riemann-Hilbert approach to the modified Camassa-Holm equation with nonzero boundary conditions,
\newblock {\em  J. Math. Phys., }    61(2020), 031504.1-25.

\bibitem{Xurhp}
J. Xu and E. Fan,
\newblock   Long-time asyptotics behavior for the
integrable modified camassa-holm equation with cubic nonlinearity,
\newblock {\em arXiv} 1911.12554,  2020.


\bibitem{Mon2}
A. Boutet de Monvel,  I. Karpenko  and D. Shepelsky,
\newblock The modified Camassa-Holm equation on a  nonzero background: Large-time asymptotics for the Cauchy problem,
\newblock {\em  Arxiv, }  arXiv:2011.13235v1, 2020.
















\bibitem{I1980}
 Y. H. Ichikawa, K. Konno, M. Wadati, H. Sanuki,
\newblock {Spiky soliton in circular polarized Alfv$\acute{e}$n wave, }
\newblock {\em  J. Phys. Soc. Jpn.}, 48(1980), 279-286.

\bibitem{AVKitaev}
A. V. Kitaev, A. H. Vartanian
\newblock {Leading-order temporal asymptotics of the modified nonlinear Schr$\ddot{o}$dinger equation: solitonless sector},
\newblock {\em  Inverse Prooblems}, 13(1997), 1311-1339.

\bibitem{Xu2013}
J. Xu, E. Fan, Y. Chen,
\newblock   Long-time asymptotic for for the derivative nonlinear Schr$\ddot{o}$dinger equation  with step-like initial value,
\newblock {\em Math. Phys. Anal. Geometry},  16(2013), 253-288.



\bibitem{CH}
A. B. de Monvel, D. Shepelsky,
\newblock   Riemann-Hilbert approach for the Camassa-Holm equation on the line,
\newblock {\em Comptes Rendus Mathematique},  343(2006), 627-632.

\bibitem{CH1}
A. Boutet de Monvel, D. Shepelsky,
\newblock   Riemann-Hilbert problem in the inverse scattering for the Camassa-Holm equation on the line,
\newblock {\em Math. Sci. Res. Inst. Publ.},  55(2007), 53-75.

\bibitem{RHPsp}
A. Boutet de Monvel, D. Shepelsky   and L. Zielinski,
\newblock    The short pulse equation by a Riemann-Hilbert approach,
\newblock {\em Lett. Math. Phys.},  107(2017), 1-29.

\bibitem{HG2009}
H. Kr\"uger and G. Teschl,
\newblock    Long-time asymptotics of the Toda lattice for decaying initial data revisited,
\newblock {\em Rev. Math. Phys.}, 21(2009), 61-109.



\end{thebibliography}
\end{document}